\documentclass[a4paper,12pt, reqno]{amsart}

\makeatletter

\theoremstyle{plain}

\newtheorem{thm}{Theorem}[section]

\newtheorem{cor}[thm]{Corollary} 

\newtheorem{lem}[thm]{Lemma} 
\newtheorem{prop}[thm]{Proposition} 

\theoremstyle{definition}
\newtheorem{defn}[thm]{Definition}
\newtheorem{rem}[thm]{Remark}
\newtheorem{rems}[thm]{Remarks}
\theoremstyle{remark}

\theoremstyle{definition}
\newtheorem{question}[thm]{Question}
\newtheorem{example}[thm]{Example}
\newtheorem{examples}[thm]{Examples}

\theoremstyle{definition}

\theoremstyle{plain}

\theoremstyle{definition}

\theoremstyle{remark}

\theoremstyle{remark}

\theoremstyle{definition}

\theoremstyle{remark}
  \newtheorem*{acknowledgement*}{Acknowledgement}

\newcommand{\R}{{\mathbb R}}

\newcommand{\C}{{\mathbb C}}
\newcommand{\N}{{\mathbb N}}
\newcommand{\Z}{{\mathbb Z}}

\newcommand{\e}{\varepsilon}

\newcommand{\id}{\mathrm{id}}

\DeclareMathOperator{\supp}{supp}

\DeclareMathOperator{\dist}{dist}

\DeclareMathOperator{\dom}{dom}

\newcommand{\ndim}{\mathrm{dim}_{\mathrm{nuc}}}
\newcommand{\ddim}{\mathrm{dim}_{\mathrm{diag}}}
\newcommand{\tdim}{\mathrm{dim}_{\mathrm{tow}}}
\newcommand{\ftdim}{\mathrm{dim}_{\mathrm{ftow}}}
\newcommand{\dad}{\mathrm{dad}}
\newcommand{\asdim}{\mathrm{asdim}}
\newcommand{\unit}{\mathbf{1}}
\newcommand{\cn}{\mathcal{N}}

\def\freeprod{\font\bigsymbolsfont=cmsy10 scaled \magstep3
 \setbox0=\hbox{\bigsymbolsfont\char'003 }\mathord{\lower1pt\box0}}\relax\ignorespaces

\newcommand{\Hawaii}{Hawai\kern.05em`\kern.05em\relax i}

\usepackage{amssymb} \usepackage{amscd} \usepackage{hyperref}  
\usepackage{tikz} \usetikzlibrary{matrix,arrows}
\usepackage[all]{xy}
\usepackage{enumitem}
\numberwithin{equation}{section}

\makeatother

\begin{document}

\title{The Diagonal dimension of sub-$\mathrm{C}^*$-algebras}

\author{Kang Li}
\address{Centre for Mathematical Sciences, Lund University, \newline Box 118, SE-221 00 Lund, Sweden}
\email{kang.li@math.lth.se}

\author{Hung-Chang Liao}
\address{Department of Mathematics and Statistics, University of Ottawa,
\newline  585 King
Edward Avenue, Ottawa, Ontario, KiN 6N5, Canada}
\email{hliao@uottawa.ca}

\author{Wilhelm Winter}
\address{Mathematisches Institut der WWU M\"unster,
\newline Einsteinstrasse 62, 48149 M\"unster, Deutschland}
\email{wwinter@uni-muenster.de}

\date{today}

\begin{abstract}
We introduce \emph{diagonal dimension}, a version of nuclear dimension for diagonal sub-$\mathrm{C}^*$-algebras (sometimes also referred to as diagonal $\mathrm{C}^*$-pairs). Our concept has good permanence properties and detects more refined information than nuclear dimension. In many situations it is precisely how dynamical information is encoded in an associated $\mathrm{C}^*$-pair. 

For free actions on compact Hausdorff spaces, diagonal dimension of the crossed product with its canonical diagonal is bounded above by a product involving Kerr's tower dimension of the action and covering dimension of the space. It is bounded below by the dimension of the space, by the asymptotic dimension of the group, and by the fine tower dimension of the action. For a locally compact, Hausdorff, \'etale groupoid, diagonal dimension of the groupoid $\mathrm{C}^*$-algebra is bounded below by the dynamic asymptotic dimension of the groupoid.  For free Cantor dynamical systems, diagonal dimension (defined at the level of the crossed product $\mathrm{C}^*$-algebra) and  tower dimension (an entirely dynamical notion) agree on the nose. Similarly, for a finitely generated group diagonal dimension of its uniform Roe algebra with the canonical diagonal agrees precisely with asymptotic dimension of the group. This statement also holds for uniformly bounded metric spaces. We apply the lower bounds above to a number of further examples  which show how diagonal dimension keeps track of information not seen by nuclear dimension. 
 \end{abstract}

\maketitle

\renewcommand*{\thethm}{\Alph{thm}}

\section*{Introduction}

\noindent
With every topological dynamical system $G \curvearrowright X$, where $G$ is a discrete group and $X$ is a compact Hausdorff space, one can naturally and canonically associate $\mathrm{C}^*$-algebras via various crossed product constructions.

Important and notoriously hard questions then are to what extent $\mathrm{C}^*$-algebras remember the underlying dynamical systems and how to extract dynamical information from the algebras. In this paper we will be interested in a particular aspect of the second question: How can dimension type information of dynamical systems be read off from the associated $\mathrm{C}^*$-algebras? 

We will see that one can expect only very limited answers from crossed products $C(X) \rtimes G$ alone, and that one should be prepared to also keep track of the canonical inclusion of $C(X)$ (which exists because we assume the group to be discrete). Then the underlying space is recorded as the spectrum of the abelian subalgebra, and our initial questions now ask what the position of this subalgebra can tell us about the dynamics. The problem remains hard, but a handle is provided by the set of \emph{normalisers}, i.e., by elements of the crossed product which conjugate $C(X)$ into itself.

Rigidity questions as above from the $\mathrm{C}^*$-algebra point of view are particularly relevant when the group or at least the action is amenable. In this situation there is essentially only one crossed product, which is \emph{nuclear}, and we have strong tools to analyse and compare such $\mathrm{C}^*$-algebras. If the action is free and minimal, the crossed product is simple and we even have far-reaching classification results available on the $\mathrm{C}^*$-side. Such classification theorems allow us to decide whether two crossed products, at least if they have \emph{finite nuclear dimension}, are isomorphic as $\mathrm{C}^*$-algebras by just computing their (ordered) K-theory and determining their tracial state spaces. Here, tracial states correspond to invariant Borel probability measures on the underlying topological space, and we have efficient tools to compute K-theory (which is usually easier to describe in terms of the algebra than the dynamical system). Nuclear dimension is a notion of noncommutative covering dimension which is defined in terms of approximations of the $\mathrm{C}^*$-algebra by a uniformly bounded number (of which we think as colours) of finite dimensional $\mathrm{C}^*$-algebras. If the underlying space is finite dimensional and the group is sufficiently nice, e.g.\ the integers, then the crossed product does have finite nuclear dimension, so classification as above applies. 

Dimension type properties occur at the level of dynamical systems in many and sometimes subtle and surprising ways. On the one hand there is covering dimension of the space, and on the other hand the group (even if it is discrete hence zero dimensional as a topological space) may have interesting dimension like features, most notably asymptotic dimension in the sense of Gromov. More surprisingly, these properties often interact and yield stunning phenomena and applications, e.g.\ to embedding problems in topological dynamics (see \cite{Lind:PubIHES,Gut:PLMS}), or to the structure of flow spaces which in turn has been  relevant for the Farrell--Jones conjecture (see \cite{BLR08,Bar17}).

It is an intriguing consequence of classification that one can have free and minimal, uniquely ergodic actions of the integers on non-homeomorphic compact spaces with isomorphic crossed product $\mathrm{C}^*$-algebras. Moreover, there are classes of dynamical systems with very different dimension type properties, yet their crossed products all have the same nuclear dimension (namely one). This shows that a crossed product $\mathrm{C}^*$-algebra\footnote{We have \emph{reduced} crossed products in mind, and keep track of this point of view in our notation, even though we will stick to situations where the groups or at least the actions are amenable, hence the full and the reduced crossed products coincide.} $C(X) \rtimes_{\mathrm{r}} G$ in itself can only be expected to carry limited information about the underlying dynamics, and it is one of the reasons to also consider the canonical inclusion of $C(X)$ and the set of normalisers.

When the action is free, the sub-$\mathrm{C}^*$-algebra $(C(X) \subset C(X) \rtimes_\mathrm{r} G)$ is in fact a \emph{diagonal}, i.e., $C(X)$ is maximal abelian, it is the image of a faithful conditional expectation, normalisers generate all of $C(X) \rtimes_\mathrm{r} G$ as a $\mathrm{C}^*$-algebra, and every pure state on $C(X)$ extends uniquely to $C(X) \rtimes_\mathrm{r} G$. One can define diagonal pairs in this sense also when there is no underlying dynamical system; the ambient $\mathrm{C}^*$-algebra of such a pair $(D \subset A)$ is then a groupoid $\mathrm{C}^*$-algebra. We study normalisers of sub-$\mathrm{C}^*$-algebras and maps between them in Section~\ref{sec1}. 
\bigskip

Our main definition (Definition~\ref{defn:dim_diag} below) describes a noncommutative version of covering dimension in terms of colouring numbers of completely positive approximations, which at the same time keep track of the given abelian subalgebra and its normalisers.

\begin{defn}\label{intro-ddim}
	Let $(D\subset A)$ be a sub-$\mathrm{C}^*$-algebra with $D$ abelian. We say $(D\subset A)$ has \emph{diagonal dimension} at most $d$, $\ddim (D\subset A) \le d$, if for every finite subset $\mathcal{F}\subset A$ and $\e >0$ there exist a finite-dimensional $\mathrm{C}^*$-algebra $F= F^{(0)} \oplus \ldots \oplus F^{(d)}$ with a diagonal subalgebra $D_F = D^{(0)} \oplus \ldots \oplus D^{(d)}$ and completely positive  maps
	$$
		A\stackrel{\psi}{\longrightarrow} F \stackrel{\varphi}{\longrightarrow} A
	$$
	such that
	\begin{enumerate}
		\item $\psi$ is contractive,
		\item $\| \varphi \psi(a) - a\| < \e$ for every $a\in \mathcal{F}$,
		\item for each $i=0,\ldots,d$, the map $\varphi\vert_{F^{(i)}}$ is completely positive contractive with order zero, i.e., it preserves orthogonality,
		\item $\psi(D)\subset D_F$,
		\item for each $i$, $\varphi$ maps every normaliser of $D^{(i)}$ in $F^{(i)}$ to a normaliser of $D$ in $A$.
	\end{enumerate}
\end{defn}

When $D = \{0\}$, then conditions (4) and (5) are trivially satisfied and diagonal dimension of $(D \subset A)$  precisely agrees with nuclear dimension of $A$, so our notion indeed generalises nuclear dimension to sub-$\mathrm{C}^*$-algebras. If $D$  contains an approximate unit for $A$, then finite diagonal dimension implies that $(D \subset A)$ is a diagonal in the sense explained above. We introduce this concept and derive its basic properties in Section~\ref{section:diagonal-dimension}.
\bigskip

As one should expect, diagonal dimension has good permanence properties with respect to direct sums, tensor products, hereditary subalgebras, unitisations, quotients, inductive limits, and stabilisations. The zero-dimensional case can be characterised in terms of AF algebras with canonical diagonals. We derive these results in Sections \ref{sec:permanence} and \ref{section-AF}.
\bigskip

There is more than one reason why we are excited about the notion of diagonal dimension. 

First, from a $\mathrm{C}^*$-algebraist's point of view, we are confident that it will lead to a better understanding of underlying dynamical structures of otherwise well-understood $\mathrm{C}^*$-algebras. For example, there is the distinct possibility that the CAR algebra $M_{2^\infty}$ contains (up to conjugacy) only one diagonal with Cantor spectrum and finite diagonal dimension. Such a result would be a strong classification statement, and a highly satisfying $\mathrm{C}^*$-algebra version of the famous Connes--Feldman--Weiss theorem (see \cite{CFW}). (An even more ambitious long-term goal would be to classify all Cantor spectrum diagonals with finite diagonal
dimension of a given simple nuclear $\mathrm{C}^*$-algebra in terms of K-theoretic invariants.) 

From a dynamical systems point of view, we believe diagonal dimension is an exciting notion because it allows to recover purely dynamical information from purely $\mathrm{C}^*$-algebraic data, thus encoding a certain amount of rigidity. 
Let us recall the notion of \emph{tower dimension} introduced by Kerr in \cite[Definition 4.3]{Ker17}, in order to explain this in more detail (a full account is given in Section~\ref{sec5}). Suppose $G \curvearrowright X$ is an action so that for every finite subset $E \subset G$ there are finite families $(V_j)_{j\in J}$ of open subsets of $X$ and $(S_j)_{j\in J}$ of finite subsets of $G$,  and a partition $J = J^{(0)}\sqcup \ldots \sqcup J^{(d)}$ such that 
	\begin{enumerate}
		\item $\bigcup_{j\in J}S_jV_j = X$,
		\item for every $x\in X$ there are $j\in J$ and $t\in S_j$ such that $x\in tV_j$ and $Et\subset S_j$,
		\item for each $i\in \{0,\ldots,d\}$ the sets $sV_j$ are pairwise disjoint for $s \in S_j$, $j\in J^{(i)}$.
	\end{enumerate}
Then, the action is said to have tower dimension at most $d$, written $\tdim (X,G) \le d$.

We can now state our first main result  (Theorem~\ref{thm:diag_vs_tower} below), which relates the dynamical and $\mathrm{C}^*$-algebraic notions of covering dimension described above. (We use the suggestive $\dim^{+1}$ notation whenever our statements involve products of dimensions, since then what matters is the number of colours, i.e., the value of the dimension plus one.)

\begin{thm}
	Let $\alpha:G\curvearrowright X$ be an action of a countable, discrete, amenable group on a compact Hausdorff space. Then 
	\begin{align}
	&\tdim^{+1}(X,G) \nonumber \\
	&\qquad \qquad \leq  \ddim^{+1}( C(X)\subset C(X)\rtimes_\mathrm{r} G ) \nonumber\\
	&\qquad \qquad \qquad \qquad \leq   \tdim^{+1}(X,G)\cdot \dim^{+1}(X). \nonumber
	\end{align}
	In particular, if $X$ is zero-dimensional then
	\begin{equation}
	\tdim(X,G) = \ddim( C(X) \subset C(X)\rtimes_\mathrm{r} G). \nonumber
	\end{equation}
\end{thm}
\bigskip

Let us now describe how diagonal dimension is relevant for coarse geometry, which studies `large-scale' properties of metric spaces. We restrict ourselves to discrete metric spaces of bounded geometry, which cover the important motivating examples of finitely generated discrete groups equipped with word-length metrics. Following Gromov, such a space $X$ is said to have \emph{asymptotic dimension} at most $d$, $\asdim (X) \le d$, if for every $R \ge 1$ there is a cover $\mathcal{U}$ of $X$ such that the members of $\mathcal{U}$ have uniformly bounded diameter such that every $R$-ball in $X$ intersects at most $d+1$ members of $\mathcal{U}$   

The \emph{uniform Roe algebra} $\mathrm{C}^*_{\mathrm{u}}(X)$ associated with $X$ can then be defined as the $\mathrm{C}^*$-algebra generated by operators with finite propagation on the Hilbert space $\ell^2(X)$; see \cite{Roe03}. In the case of a finitely generated group $G$, it can be  identified with the crossed product $\ell^\infty(G) \rtimes_{\mathrm{r}} G$ (with the action being left translation). It was shown only recently that the uniform Roe algebra determines $X$ as a metric space up to \emph{coarse equivalence} (see \cite{BFKVW:2021}), a notion under which asymptotic dimension is invariant. This in particular means that the asymptotic dimension of $X$ is encoded in the uniform Roe algebra; the question then is how to read off this information. In \cite{WW18}, abstract properties of the sub-$\mathrm{C}^*$-algebra $(\ell^\infty(X)  \subset \mathrm{C}^*_{\mathrm{u}}(X))$ were identified  which ensure that every such \emph{Roe Cartan subalgebra} yields a coarse equivalence of the underlying spaces. With diagonal dimension at hand we can now read off the precise value of asymptotic dimension from the uniform Roe algebra (combining Theorem~\ref{thm:uniformRoe} and Corollary~\ref{cor:RoeCartan}) as follows:  

\begin{thm}
For $X$ a discrete metric space with bounded geometry, we have
\[
\ddim(\ell^\infty(X) \subset \mathrm{C}^*_{\mathrm{u}}(X)) = \asdim(X). 
\] 
Moreover, we have 
\[
\ddim(B \subset \mathrm{C}^*_{\mathrm{u}}(X)) = \asdim(X) 
\]
for every Roe Cartan subalgebra $(B \subset \mathrm{C}^*_{\mathrm{u}}(X))$.
\end{thm}
\bigskip

Topological dynamical systems and metric spaces of bounded geometry can be studied in the joint framework of (locally compact, Hausdorff, \'etale) \emph{groupoids}; see \cite{Renault:LNM1980,GWY17} and our Section~\ref{sec6} below. With such a groupoid $\mathcal{G}$ one can associate a sub-$\mathrm{C}^*$-algebra $( C_0(\mathcal{G}^{(0)})\subset \mathrm{C}^*_{\mathrm{r}}(\mathcal{G}) )$, and it turns out that every diagonal sub-$\mathrm{C}^*$-algebra (hence in particular every sub-$\mathrm{C}^*$-algebra with finite diagonal dimension) indeed comes from a groupoid which is \emph{principal} (a notion analogous to freeness for dynamical systems) and possibly \emph{twisted} (a notion which we will not require in our applications); see Proposition~\ref{prop:principal-groupoid}:

\begin{prop}Let $(D_A\subset A)$ be a nondegenerate sub-$\mathrm{C}^*$-algebra with finite diagonal dimension. Then there is an -- up to isomorphism uniquely determined -- twisted, \'etale, locally compact, Hausdorff, principal groupoid $(\mathcal{G},\Sigma)$ such that $(D_A\subset A)$ is isomorphic to $(C_0(\mathcal{G}^{(0)}) \subset \mathrm{C}_{\mathrm{r}}^*(\mathcal{G},\Sigma))$.
\end{prop}

Guentner, Willett, and Yu have generalised asymptotic dimension to the notion of \emph{dynamic asymptotic dimension} for groupoids (written $\dad(\mathcal{G})$; spelling out the definition requires some preparation and we postpone it to Definition~\ref{defn:dad} below). The concept has been related to tower dimension in \cite{Ker17}, and just like tower dimension, it provides a lower bound for diagonal dimension of the associated sub-$\mathrm{C}^*$-algebra; see Theorem~\ref{thm:diag_vs_dad} below:

\begin{thm}
Let $\mathcal{G}$ be a locally compact, Hausdorff, \'etale groupoid. Then
	\[
	\dad(\mathcal{G})\leq \ddim( C_0(\mathcal{G}^{(0)})\subset \mathrm{C}^*_{\mathrm{r}}(\mathcal{G}) ).
	\]
\end{thm}

\bigskip 
The theorems above consist of upper and lower bounds for diagonal dimension. The upper bounds yield a plethora of examples of sub-$\mathrm{C}^*$-algebras with finite diagonal dimension. The proofs are essentially contained in existing ones for nuclear dimension -- we more or less just have to follow those arguments and make sure they can be adjusted to also keep track of diagonal subalgebras.
 
The lower bounds show that diagonal dimension is genuinely a more sensitive invariant than nuclear dimension. The proofs, albeit sometimes technical to write down, are natural in the sense that the covers required by tower dimension or by (dynamic) asymptotic dimension can be constructed explicitly from completely positive approximations in the sense of Definition~\ref{intro-ddim}. This process requires a certain amount of rigidity, which is ensured by the order zero condition (3) in tandem with the normaliser condition (5) of the definition.   

\bigskip

Our examples for which diagonal dimension carries interesting information range from Cantor actions of locally finite, or virtually nilpotent, or Grigorchuk groups, over actions of amenable and residually finite groups on profinite completions,  to universal minimal flows  with finite asymptotic dimension. We work through these examples in Section~\ref{sec:examples}.

\newpage
\subsection*{Acknowledgments}
This work was partially funded by: the Deutsche Forschungsgemeinschaft (DFG, German Research Foundation) under Germany's Excellence Strategy -- EXC 2044 -- 390685587, Mathematics M\"unster -- Dynamics -- Geometry -- Structure; the Deutsche Forschungsgemeinschaft (DFG, German Research Foundation) -- SFB 878; the Deutsche Forschungsgemeinschaft (DFG, German Research Foundation) -- Project-ID 427320536 -- SFB 1442; ERC Advanced Grant 834267 -- AMAREC.  

We are grateful to Selcuk Barlak, Christian B\"onicke, David Kerr, Xin Li, Stuart White, and Rufus Willett for helpful comments and many inspiring discussions.

Last but not least, we would like to thank the referees for their meticulous reading of the manuscript and for their very helpful comments.

\setcounter{tocdepth}{1}
\tableofcontents
\numberwithin{thm}{section}
\newpage

\section{\texorpdfstring{Sub-$\mathrm{C}^*$}{Sub-C*}-algebras and normalisers}
\label{sec1}

\noindent
For a $\mathrm{C}^*$-algebra $A$, we write $A_+$ for the set of positive elements in $A$ and $A^1$ and $A^1_+$ for the norm-closed unit balls of $A$ and $A_+$, respectively.

\begin{defn}
	If $A$ is a $\mathrm{C}^*$-algebra containing another $\mathrm{C}^*$-algebra $D$ we write $(D \subset A)$ for this setup and call it a \emph{sub-$\mathrm{C}^*$-algebra}, or a \emph{pair} of $\mathrm{C}^*$-algebras, or just a \emph{$\mathrm{C}^*$-pair}. 
A sub-$\mathrm{C}^*$-algebra $(D\subset A)$ is said to be \emph{nondegenerate} if $D$ contains an approximate unit for $A$.
\end{defn}

\begin{rems}
(i) Our notation is reminiscent of that for subfactors; it indicates that all three pieces of data -- $D$, $A$, and the position of $D$ in $A$ -- carry crucial information.

(ii) A sub-$\mathrm{C}^*$-algebra $(D\subset A)$ is nondegenerate if and only if the positive open unit ball of $D$ forms an approximate unit for $A$. If $(D\subset A)$ is a nondegenerate sub-$\mathrm{C}^*$-algebra and $A$ is unital, then $D$ contains the unit of $A$. 
\end{rems}

\begin{defn}
	Let $(D\subset A)$ be a sub-$\mathrm{C}^*$-algebra. An element $a\in A$ is called a \emph{normaliser} of $D$ in $A$ if $aDa^* + a^*Da\subset D$. We also say $a$ \emph{normalises} $D$. The collection of normalisers of $D$ in $A$ is denoted by $\mathcal{N}_A(D)$.	
\end{defn}

Note that $\mathcal{N}_A(D)$ is closed under multiplication, involution, and norm-limits. It is in general not closed under addition.

\begin{defn} [see {\cite{Kum86, Ren08, Renault09}}] \label{defn:diagonal}
	Let $(D\subset A)$ be a sub-$\mathrm{C}^*$-algebra. $D$ is a \emph{Cartan subalgebra} of $A$ if \begin{enumerate}
		\item[(0)] $(D\subset A)$ is nondegenerate,
		\item[(1)] $D$ is a maximal abelian $^*$-subalgebra of $A$ (a \emph{masa}, for short),
		\item[(2)] $D$ is \emph{regular}, in the sense that $\cn_A(D)$ generates $A$ as a $\mathrm{C}^*$-algebra, and
		\item[(3)] there exists a \emph{faithful conditional expectation} from $A$ onto $D$, i.e., a completely positive contractive map $\Phi: A \to D$ which is injective on $A_+$ and satisfies $\Phi|_D = \id_D$.
	\end{enumerate}
	If, in addition, $D$ has the \emph{unique extension property} relative to $A$, that is, every pure state on $D$ extends uniquely to a pure state on $A$, then $D$ is said to be a \emph{diagonal} in $A$. We note that condition (0) above was shown in \cite{Pit21} to be redundant in connection with conditions (1) and (2). Moreover, the original definition of a diagonal reads slightly different from the above; the equivalence follows from \cite[Proposition 4\textdegree]{Kum86} in connection with \cite[Proposition 5.11]{Ren08}.
\end{defn}

\begin{example}
	Let $(D_F \subset F)$ be a sub-$\mathrm{C}^*$-algebra with $F$ (hence also $D_F$) finite-dimensional. If $D_F$ is a masa, it is a diagonal. Moreover, any two diagonals of $F$ are unitary conjugates of each other. 
\end{example}

\begin{defn}
If $F$ is a finite-dimensional $\mathrm{C}^*$-algebra and $D_F$ is a masa of $F$, we say an element $v$ in $F$ is a \emph{matrix unit with respect to $D_F$} if $v^*v$ and $vv^*$ are minimal projections in $D_F$. Note that every matrix unit is a normaliser of $D_F$, and that, if $v$ is a matrix unit, then so is $\lambda \cdot v$ for every $\lambda \in \mathbb{C}$ with $|\lambda|=1$.	
\end{defn}

The set of normalisers is obviously closed under multiplication, but not at all under addition. 
The next proposition, taken from \cite[Example~2$^\circ$]{Kum86},  characterises normalisers of diagonals in finite-dimensional $\mathrm{C}^*$-algebras. It nicely illustrates the role of orthogonality in this context, and provides first evidence why order zero maps will play a role. We will return to the matter in Proposition~\ref{prop:maps-normalisers} and in Section~\ref{section:diagonal-dimension}.

\begin{prop}
\label{prop:orthogonal-normalisers}
If $F$ is a finite dimensional $\mathrm{C}^*$-algebra with a diagonal $D_F$, then any normaliser for $D_F$ is a linear combination of pairwise orthogonal matrix units with respect to $D_F$.
\end{prop}

It is not hard to conclude from the proposition above that any \emph{positive} normaliser of $D_F$ in fact belongs to $D_F$. This fails if $D_F$ is not maximal abelian (take for example $D_F= \{0\}$). However, any positive normaliser will at least commute with $D_F$. This was shown as a general fact also outside the finite dimensional setting in \cite[Proposition~2.1]{Pit21}. We include below a different proof which may be interesting in its own right. We also use this to observe that  continuous functions of positive normalisers are again normalisers.

\begin{lem} \label{lem:positive_normaliser}
	Let $(D\subset A)$ be a sub-$\mathrm{C}^*$-algebra with $D$ abelian. If $e$ is a positive contraction in $\mathcal{N}_A(D)$ then $e$ belongs to $A \cap D'$ (so if $D$ is a masa, then $e \in D$). Moreover, $f(e)$ normalises $D$ for any continuous function $f$ on the spectrum of $e$ with $f(0) = 0$.
\end{lem}
\begin{proof}
	Assume for a contradiction that $\mathrm{C}^*(D, e)$ is not abelian. Then there exists an irreducible representation 
\[
\pi: \mathrm{C}^*(D,e) \longrightarrow \mathcal{B}(\mathcal{H})	
\]
such that $\mathcal{H} \ncong \mathbb{C}$. 

Now if $\pi(D)$ was contained in $\mathbb{C} \cdot \unit_{\mathcal{H}}$, then $\pi|_{\mathrm{C}^*(e)}$ was an irreducible representation of an abelian $\mathrm{C}^*$-algebra on a Hilbert space of dimension strictly greater than one; this is impossible, so $\pi(D) \not\subset \mathbb{C} \cdot \unit_{\mathcal{H}}$. 

It follows that the spectrum of $\pi^\sim(D^\sim)$ (where $D^\sim$ is the smallest unitisation of $D$ and $\pi^\sim$ is the unitisation of $\pi$) contains at least two points, at least one of which is also in the spectrum of $\pi(D)$. 
But then there are $d \in D^1_+$ and $b \in (D^\sim)^1_+$ such that $0 \neq \pi(d)$, $0 \neq \pi^\sim(b)$, and $db = 0$.

For $k \in \mathbb{N}$ we now compute
\begin{align*}
\|be^kd\|^2 & = \|b e^k d^2 e^k b\| \\
& \le \|b e^k d e^k b\| \\
& = \|b e^k d e^k b^2 e^k d e^k b\|^\frac{1}{2} \\
& = \|b^2 e^k d e^k e^k d e^k b^2\|^\frac{1}{2} \\
& \le \|b^2 e^k d e^k d e^k b^2\|^\frac{1}{2} \\
& = \|b e^k d e^k b d e^k b^2\|^\frac{1}{2} \\
& = 0,
\end{align*}
where we have used that $e$ is a positive contraction normalising $D$, so that in particular $e^k d e^k \in D$ commutes with $b \in D^\sim$. As a consequence we have 
\begin{equation}
\label{positive_normaliser_orthogonal}
b \, \mathrm{C}^*(D,e) \, d = \{0\}.
\end{equation}

Because both $\pi(d)$ and $\pi^\sim(b)$ are nonzero, so are the closed linear subspaces $\overline{\pi(\mathrm{C}^*(D,e)d) \mathcal{H}}$ and $\overline{\pi^\sim(b)\mathcal{H}}$ of $\mathcal{H}$.
At the same time, they are orthogonal by \eqref{positive_normaliser_orthogonal}, and the subspace $\overline{\pi(\mathrm{C}^*(D,e)d) \mathcal{H}}$ is invariant under $\pi(\mathrm{C}^*(D,e))$. This shows that $\pi$ is not irreducible and we have reached the desired contradiction, proving that $e$ and $D$ commute.

\bigskip

The second assertion is a trivial consequence of maximality together with the first statement. 

\bigskip

For the third statement, note that for any $d \in D$ and $0 \neq k,m \in \mathbb{N}$,
\[
e^{2k} d e^{2m} = e^{k+m} d e^{k+m} \in D
\]
since $e$ commutes with $D$ and $e^{k+m}$ is a normaliser of $D$. It follows that for any even polynomial $p$ with vanishing constant term we have
\[
p(e)^* d p(e) = \bar{p}(e) d p(e) \in D.
\]
But any $f \in C_0(\sigma(e) \setminus \{0\})$ can be approximated in norm by even polynomials, so $f(e)$ normalises $D$. 
\end{proof}

In the remainder of this section we discuss circumstances under which sub-$\mathrm{C}^*$-algebras and their normalisers are preserved under maps.

The first observation is that one cannot expect to say much without  suitable nondegeneracy conditions: If $(D \subset A)$ is a sub-$\mathrm{C}^*$-algebra such that the normaliser $\mathcal{N}_A(D)$ is not all of $A$, then the identity map $\mathrm{id}: (\{0\} \subset A) \to (D \subset A)$ preserves the sub-$\mathrm{C}^*$-algebra structure, but $A = \mathcal{N}_A(\{0\}) \not\subset  \mathcal{N}_A(D)$. Conversely, $\mathrm{id}: (D \subset A) \to (\{0\} \subset A)$ sends $\mathcal{N}_A(D)$ to $\mathcal{N}_A(\{0\})$, but not $D$ to $\{0\}$.

The second observation concerns the types of maps which may preserve normalisers. Since the definition of normalisers involves the multiplicative structure, one cannot expect general statements for maps which are just linear. On the other hand, we will need to consider maps more general than $^*$-homomorphisms. Lemma~\ref{lem:positive_normaliser} suggests that positivity will play a role, and in view of Proposition~\ref{prop:orthogonal-normalisers}, it seems natural to consider maps preserving orthogonality (but not necessarily the full multiplicative structure).

Recall that a completely positive (c.p.\ for short) map $\varphi:A\to B$ between $\mathrm{C}^*$-algebras is said to be \emph{order zero} if it preserves orthogonality, i.e., $\varphi(a)\varphi(b) = 0$ whenever $a,b\in A_+$ satisfy $ab= 0$. By the structure theorem for order zero maps (see \cite[Theorem 3.3]{WZ09} which in turn depends heavily on \cite{Wolff}), every c.p.\ order zero map $\varphi:A\to B$ has the form $\varphi(\, . \,) = h\pi_{\varphi}(\, . \,)$, where $\pi_{\varphi}$ is a $^*$-homomorphism from $A$ into a larger $\mathrm{C}^*$-algebra (one can take the bidual $B^{**}$, for example) and $h$ is a positive contraction in that larger $\mathrm{C}^*$-algebra which commutes with the image of $\pi_\varphi$. (If $A$ is unital then in fact we may take $h = \varphi(\unit_A)$.) We say $\pi_\varphi$ is a \emph{supporting $^*$-homomorphism} for $\varphi$. 

\begin{prop}\label{prop:maps-normalisers}
	Let $(D_A\subset A)$ and $(D_B\subset B)$ be two sub-$\mathrm{C}^*$-algebras with $D_A$ and $D_B$ abelian. Let $\varphi: A\to B$ be a positive linear map.

{\rm(i)} If $(D_B \subset B)$ is nondegenerate and $\varphi(\mathcal{N}_A(D_A)) \subset \mathcal{N}_B(D_B)$, then $\varphi(D_A)\subset D_B$. 

{\rm(ii)} If $\varphi$ is c.p.\ order zero with $\varphi(D_A) \subset D_B$, and if $D_A$ admits an approximate unit $(u_\alpha)_\alpha$ for $A$ such that $\varphi(u_\alpha) D_B \varphi(u_\alpha) \subset \varphi(D_A)$ for all $\alpha$, then $\varphi(\mathcal{N}_A(D_A)) \subset \mathcal{N}_B(D_B)$. 

{\rm(iii)} 	If $(D_A\subset A)$ is nondegenerate and $\varphi$ is a $^*$-homo\-morphism such that $\varphi(D_A)$ is a hereditary subalgebra of $D_B$, then $\varphi(\mathcal{N}_A(D_A)) \subset \mathcal{N}_B(D_B)$.
\end{prop}

\begin{proof}
	(i) (Cf.\ \cite[Lemma 1.6]{LiTi:almost-finite}.) By hypothesis, $\varphi$ maps positive normalisers to positive normalisers. Now if $(u_\beta)_\beta \subset D_B$ is an approximate unit for $B$, we have for any $0 \le a \in D_A$
\[
\textstyle
\varphi(a) = (\varphi(a)^2)^{\frac{1}{2}} = (\lim_\beta \varphi(a) u_\beta \varphi(a))^{\frac{1}{2}} \subset D_B.
\]
By linearity this implies $\varphi(D_A) \subset D_B$.

	(ii) 	Given $a\in \mathcal{N}_A(D_A)$ and $d\in D_B$, we will show that $\varphi(a) d \varphi(a)^*$ belongs to $D_B$. By assumption for each $\alpha$ there exists an element $d_\alpha$ in $D_A$ such that $\varphi(d_\alpha) = \varphi(u_\alpha) d \varphi(u_\alpha)$. Let $\pi_\varphi:A\to B^{**}$ be a supporting $^*$-homomorphism for $\varphi$. Then
	\begin{align*}
	\varphi(a d_\alpha a^*) & = \pi_\varphi(a) \varphi(d_\alpha) \pi_\varphi(a^*) \\
	& =  \pi_\varphi(a) \varphi(u_\alpha) d \varphi(u_\alpha) \pi_\varphi(a^*) \\
	& =  \varphi(au_\alpha) d \varphi(u_\alpha a^*), 
	\end{align*}
	which implies 
	\[
	\textstyle
	\varphi(a) d \varphi(a)^* = \lim_\alpha \varphi(ad_\alpha a^*) \in \varphi(D_A) \subset D_B.
	\]
		
	(iii) Since $\varphi(D_A) \subset D_B$ is a hereditary subalgebra, we have 
	\[
	\varphi(u_\alpha) D_B \varphi(u_\alpha) \subset \overline{\varphi(D_A) D_B \varphi(D_A)} = \varphi(D_A)
	\]
	for any approximate unit $(u_\alpha)_\alpha$ for $D_A$ and the assertion follows from (ii).
\end{proof}

\begin{rem}
One can prove a slightly more involved version of Proposition~\ref{prop:maps-normalisers}(iii) for order zero maps. We restrict ourselves to the case of $^*$-homo\-morphisms since it is simpler and covers the application we particularly care about: If $(D_B \subset B)$ is a sub-$\mathrm{C}^*$-algebra with $D_B$ abelian, then $(D_B \subset \overline{D_B B D_B})$ is a nondegenerate sub-$\mathrm{C}^*$-algebra which satisfies 
\[
\mathcal{N}_{B}(D_B) = \mathcal{N}_{\overline{D_B B D_B}}(D_B).
\]
This makes it possible to pass to nondegenerate sub-$\mathrm{C}^*$-algebras without changing the involved sets of normalisers. 
\end{rem}

\section{Diagonal dimension}
\label{section:diagonal-dimension}

\noindent
Below we define the diagonal dimension of a sub-$\mathrm{C}^*$-algebra as an approximation property and establish some of its basic features. We derive an equivalent characterisation which also applies to nuclear dimension and which only involves the incoming maps. A large part of this section will be spent on showing that finite diagonal dimension in fact implies that the subalgebra is diagonal in the sense of Definition~\ref{defn:diagonal}; see Theorem~\ref{principal grp}.

\begin{defn} \label{defn:dim_diag}
	Let $(D_A\subset A)$ be a sub-$\mathrm{C}^*$-algebra with $D_A$ abelian. We say $(D_A\subset A)$ has \emph{diagonal dimension} at most $d$, written as $\ddim(D_A\subset A) \leq d$, if for every finite subset $\mathcal{F}\subset A$ and $\e >0$ there exist a finite-dimensional $\mathrm{C}^*$-algebra $F$ with a masa $D_F$ and c.p.\ maps
	$$
	A\stackrel{\psi}{\longrightarrow} F \stackrel{\varphi}{\longrightarrow} A
	$$
	such that
	\begin{enumerate}
		\item $\psi$ is contractive,
		\item $\| \varphi \psi(a) - a\| < \e$ for every $a\in \mathcal{F}$,
		\item $F$ decomposes into $F = F^{(0)}\oplus \ldots \oplus F^{(d)}$ such that $\varphi^{(i)} := \varphi\vert_{F^{(i)}}$ is completely positive contractive (c.p.c.\ for short) order zero for each $i=0,\ldots,d$,
		\item $\psi(D_A)\subset D_F$,
		\item $\varphi$ maps every matrix unit with respect to $D_F$ into $\mathcal{N}_A(D_A)$.
	\end{enumerate}

As common for notions of dimension we write $\ddim(D_A\subset A) = d$ if $d$ is the least integer such that $\ddim(D_A\subset A) \leq d$. If no such $d$ exists we write $\ddim(D_A\subset A) = \infty$. 
\end{defn}

By a \emph{system of c.p.\ approximations witnessing $\ddim(D_A\subset A) \leq d$} we mean a net
$(F_\lambda,  D_{F_\lambda},\allowbreak \psi_\lambda, \varphi_\lambda )_{\lambda\in \Lambda}$ of approximations as above with $\varphi_\lambda\psi_\lambda\to \id_A$ in the point-norm topology.

Similarly, a single approximation $( F, D_F, \psi, \varphi )$ as in Definition \ref{defn:dim_diag} with respect to a given finite subset $\mathcal{F}\subset A$ and $\e > 0$ will be called a \emph{c.p.\ approximation witnessing $\ddim(D_A\subset A)\leq d$ for $(\mathcal{F},\e)$} (or for $\mathcal{F}$ within $\e$).

\begin{rems} \label{rem:almost_order_zero} 
(i) Upon dropping conditions (4) and (5) of Definition \ref{defn:dim_diag} one recovers the \emph{nuclear dimension} of $A$ (as in \cite[Definition 2.1]{WZ10}), whence 
\[
\ndim A \le \ddim(D_A\subset A)
\]
 for any sub-$\mathrm{C}^*$-algebra $(D_A \subset A)$ with $D_A$ abelian.  

If $D_A = \{0\}$ then conditions (4) and (5)  are automatically satisfied, so that nuclear dimension and diagonal dimension agree in this situation, 
\[
\ndim A = \ddim(\{0\}\subset A).
\] 

If $(D_A \subset A)$ is nondegenerate and $\ddim (D_A \subset A)$ is finite, then for a system of approximations $(F_\lambda,  D_{F_\lambda},\allowbreak \psi_\lambda, \varphi_\lambda )_{\lambda\in \Lambda}$ as in Definition \ref{defn:dim_diag} we have $\psi_\lambda(D_A) \subset D_{F_\lambda}$ and (by Proposition \ref{prop:maps-normalisers}(i)) $\varphi_\lambda(D_{F_\lambda}) \subset D_A$. Therefore, $(D_{F_\lambda}, \psi_\lambda|{D_A}, \varphi|_{F_\lambda})$ is a system of c.p.\ approximations witnessing 
\[
\ndim D_A \le \ddim (D_A \subset A).
\]
(It will follow from Theorem \ref{thm:permanence}(iii) that $\ddim (D_A \subset \overline{D_A A D_A})\le \ddim (D_A \subset A)$, whence non-degeneracy is a red herring in the statement above.) 

It also follows directly from the definitions that  
\[
\ndim(C_0(X)) = \ddim(C_0(X)\subset C_0(X))
\]
for any abelian $\mathrm{C}^*$-algebra $C_0(X)$.

(ii) If $\ddim(D_A\subset A) = d$ then the maps $\varphi$ in Definition \ref{defn:dim_diag} satisfy $\|\varphi\| \leq d+1$ (but are not necessarily contractive).  Moreover, when $(D_A\subset A)$ is nondegenerate, the same argument as in \cite[Remark 2.2 (iv)]{WZ10} shows that in Definition \ref{defn:dim_diag} one may assume that the composition $\varphi\psi$ is contractive. On the other hand, given a positive contraction $h\in A$ we can define
\begin{gather}
	\hat{F} := \mathrm{her}(\psi(h)) \subset F, \nonumber\\
	\hat{\psi}:= \psi(h)^{-\frac{1}{2}}\psi(\, . \, )\psi(h)^{-\frac{1}{2}}:A\longrightarrow \hat{F}, \nonumber\\
	\hat{\varphi}:= \varphi( \psi(h)^{\frac{1}{2}}\, . \, \psi(h)^{\frac{1}{2}} ):\hat{F}\longrightarrow A. \label{2-2-1}
\end{gather}
Then $\hat{\varphi}$ becomes contractive and $\hat{\varphi}\hat{\psi}(a) = \varphi\psi(a)$ for all $a\in A$ satisfying $ha = a = ah$ (see the proof of \cite[Proposition 4.3]{WZ10}). Note that, however, the map $\hat{\varphi}$ in general is no longer a sum of order zero maps.

(iii) The proof of \cite[Proposition 3.2]{WZ10} shows that (possibly after throwing away some summands of $F$) we may assume the maps $\psi_\lambda$ in a system of c.p.\ approximations witnessing $\ddim(D_A\subset A)\leq d$ to be almost order zero, i.e., 
\[
\| \psi_{\lambda}(a)\psi_{\lambda}(b)\| \longrightarrow 0
\]
whenever $a,b\in A_+$ satisfy $ab =0$.

Moreover, by \cite[Proposition~4.2]{Win12} and its proof, if $A$ is unital then for any system of c.p.\ approximations witnessing $\ddim(D_A\subset A)\leq d$ we have 
\[
\|  \varphi_\lambda^{(i)} \psi_\lambda^{(i)}(\mathbf{1}_A) \, a - \varphi_\lambda^{(i)} \psi_\lambda^{(i)}(a) \| \longrightarrow 0
\]
for all $a \in A$ and $i \in \{0,\ldots,d\}$ (the separability assumption in \cite[Proposition~4.2]{Win12} is not essential at this point).

(iv) Although we required $D_A$ to be abelian in Definition \ref{defn:dim_diag}, this is in fact automatic when $(D_A\subset A)$ has finite diagonal dimension. Indeed, let $(F_\lambda, D_{F_\lambda}, \psi_\lambda,\varphi_\lambda)_{\lambda\in \Lambda}$ be a system of c.p.\ approximations witnessing $\ddim(D_A\subset A) \leq d$. By Remark \ref{rem:almost_order_zero}(iii) we may assume the maps $\psi_\lambda$ to be almost order zero. In particular the map
\[
\bar{\psi}:D_A \longrightarrow \prod\nolimits_{\Lambda} D_{F_\lambda} \big/ \bigoplus\nolimits_\Lambda D_{F_\lambda}
\]
induced by the $\psi_\lambda\vert_{D_A}$ has order zero. Since $\varphi_\lambda\psi_\lambda \to \id_A$ pointwise, the map $\bar{\psi}$ is faithful. Now a supporting  $^*$-homomorphism $\pi_{\bar{\psi}}$ maps $D_A$ into an abelian $\mathrm{C}^*$-algebra. As $\bar{\psi}$ is faithful, the map $\pi_{\bar{\psi}}$ is necessarily an embedding. This shows that $D_A$ is abelian.

(v) In Definition~\ref{defn:dim_diag}, if $(D_A \subset A)$ is nondegenerate then it follows from \cite[Lemma~1.8]{LiTi:almost-finite}\footnote{This lemma is phrased for order zero maps out of matrix algebras, but it directly carries over to order zero maps out of direct sums of matrix algebras.} that one may replace condition (5) by the -- formally stronger, but in fact equivalent -- condition
	\begin{enumerate}
		\item[(5')] $\varphi^{(i)}(  \mathcal{N}_{F^{(i)}}(D_{F^{(i)}}))\subset \mathcal{N}_A(D_A), \, i=0,\ldots,d$.
	\end{enumerate}
	In this way one can avoid using matrix units in the definition and phrase it more symmetrically. 
We nonetheless preferred to use (5) in \ref{defn:dim_diag}, since matrix units will be used heavily in proofs. 

Note that the condition in (5') is required for each colour $i$ separately. This cannot be improved: For example, the sub-$\mathrm{C}^*$-algebra $(D_A \subset A) = (C([0,1], D_2) \subset C([0,1], M_2))$ (with $(D_2 \subset M_2)$ the standard diagonal) has diagonal dimension 1, but using Proposition~\ref{prop:orthogonal-normalisers} one can show that diagonal dimension 1 cannot be witnessed by approximations with $\varphi(\mathcal{N}_F(D_F)) \subset \mathcal{N}_A(D_A)$. 

(vi) If $\ddim(D_A\subset A)$ is finite, then $D_A$ is regular in $A$: Indeed, if $( F_\lambda, D_{F_\lambda}, \psi_{\lambda}, \varphi_{\lambda})$ is a system of c.p.\ approximations witnessing $\ddim(D_A\subset A) = d$, then the span of the union $\bigcup_{\lambda} \varphi_\lambda(F_\lambda)$ is dense in $A$. Since $\varphi_\lambda(v)$ belongs to $\mathcal{N}_A(D_A)$ for each matrix unit $v$ in $F_\lambda$ and these elements span the subspace $\varphi_{\lambda}(F_\lambda)$, we see that $A$ is generated by $\mathcal{N}_A(D_A)$.
\end{rems}

Next we give characterisations of nuclear dimension and of diagonal dimension which do not involve the maps $\psi: A \to F$. In \cite[Theorem 6.2]{Sato:JOT}, Sato gave an explicit one-sided characterisation of decomposition rank (cf.\ \cite{KW04}) for unital separable $\mathrm{C}^*$-algebras in terms of sequence algebras. This result could be adapted to characterise nuclear dimension as well, and also to cover the nonunital case. Such a result would then also yield the statement about nuclear dimension in the proposition below, but it is not at all obvious how to handle diagonal dimension along these lines. It is also worth mentioning that our method is different in the sense that it bypasses the heavy machinery involving Connes' theorem used in \cite{Sato:JOT}. In the same vein, the characterisation of nuclear dimension from Proposition \ref{prop:dim-without-psi} easily passes to quotients, thus giving a proof that finite nuclear dimension passes to quotients which does not involve the -- notoriously heavy -- respective statement for nuclearity.

\begin{prop}\label{prop:dim-without-psi}
For a $\mathrm{C}^*$-algebra $A$, $\ndim A \le d$ if and only if the following holds:

If $\mathcal{F} \subset A^1$ is a finite subset admitting a positive contraction $h \in A^1_+$ such that $h a = a h = a$ for all $a \in \mathcal{F}$, then for every $\e >0$ there is a linear map $\varphi:F \to A$ such that 
\begin{enumerate}
\item[\rm{(1)}] $F = F^{(0)} \oplus \ldots \oplus F^{(d)}$ is a finite dimensional $\mathrm{C}^*$-algebra,
\item[\rm{(2)}] $\varphi^{(i)}:= \varphi|_{F^{(i)}}$ is c.p.c.\ order zero for each $i$,
\item[\rm{(3)}] for every $a \in \mathcal{F} \cup \{h\}$ there is $b_a= b_a^{(0)} \oplus \ldots \oplus b_a^{(d)} \in F^1$ with 
\[
\|\varphi(b_a) - a\| < \e 
\]
and, for each $i$ and each $a \in \mathcal{F}$,
\[
\|\varphi^{(i)}(b^{(i)}_{h}) a - \varphi^{(i)}(b_a^{(i)})\| < \e.
\]
\end{enumerate}

When $(D_A \subset A)$ is a nondegenerate sub-$\mathrm{C}^*$-algebra with $D_A$ abelian, then we have $\ddim(D_A \subset A) \le d$ if and only if the following holds:  

If $\mathcal{F} \subset A^1$ is a finite subset admitting a positive contraction $h \in (D_A)^1_+$ such that $h a = a h = a$ for all $a \in \mathcal{F}$, then for every $\e >0$ there is a linear map $\varphi:F \to A$ satisfying conditions {\rm (1), (2)} and {\rm (3)} as above and, in addition, there is a masa $D_F = D_F^{(0)} \oplus \ldots \oplus D_F^{(d)} \subset F= F^{(0)} \oplus \ldots \oplus F^{(d)}$ such that for each $i \in \{0,\ldots,d\}$
\begin{equation}
\label{prop:dim-without-psi-eq}
\varphi^{(i)}(  \mathcal{N}_{F^{(i)}}(D_{F^{(i)}}))\subset \mathcal{N}_A(D_A)
\end{equation}
and such that $b^{(i)}_{h}$ can be chosen to lie in $D_{F^{(i)}}$.

Note that, if $A$ is unital, the element $h$ above can simply be taken to be the unit of $A$.
\end{prop}
\begin{proof}
For the forward implication of the statement about nuclear dimension let us first consider the case where $A$ is unital. We may then assume $\mathbf{1}_A \in \mathcal{F}$, which implies $h = \mathbf{1}_A \in \mathcal{F}$. Then take a c.p.c.\ approximation $(F,\psi,\varphi)$ witnessing $\ndim A \le d$ for $\mathcal{F}$ within $\e$; by \cite[Proposition~4.2]{Win12} (also cf.\ Remark~\ref{rem:almost_order_zero}(iii)) we may assume that the approximation in addition satisfies $\|  \varphi^{(i)} \psi^{(i)}(\mathbf{1}_A) \, a - \varphi^{(i)} \psi^{(i)}(a) \| < \e$. With $b_a:= \psi(a)$ for $a \in \mathcal{F}$ this yields the characterisation of nuclear dimension in the proposition.

If $A$ is not unital, essentially the same argument works upon modifying the proof of \cite[Proposition~4.2]{Win12} to yield a nonunital version as follows: First, replace $\mathbf{1}_A$ by $h$ in all places and delete the (unique occurrence of the) word ``unital''. Second, choose the system $(F_p,\psi_p, \varphi_p)$ so that (22) of \cite[Proposition~4.2]{Win12} holds for $b \in \mathcal{F} \cup \{h\} \cup (\mathcal{F} \cup \{h\})^2$ (as opposed to all $b \in A$). Third, change the definition of the maps $\hat{\psi}_p$ to 
\[
\hat{\psi}_p( \, . \,):= \psi_p(h)^{-\frac{1}{2}} \psi_p(h^{\frac{1}{2}}\, . \,h^{\frac{1}{2}}) \psi_p(h)^{-\frac{1}{2}}.
\] 
In this way, (23) and (24) of \cite[Proposition~4.2]{Win12} hold for $b \in \mathcal{F}$ and $\mathbf{1}_A$ replaced by $h$, so that we can indeed reach 
$\|  \varphi^{(i)} \psi^{(i)}(h) \, a - \varphi^{(i)} \psi^{(i)}(a) \| < \e$ for $a \in \mathcal{F}$. With $b_a:= \psi(a)$ for $a \in \mathcal{F}$ this yields the characterisation of nuclear dimension in the proposition also in the nonunital case.

We now prove the reverse.

Let $\mathcal{F} \subset A^1_+$ finite and $\eta>0$ be given. We have to produce a c.p.c.\ approximation for $\mathcal{F}$ within $\eta$ witnessing $\ndim(A) \le d$. (The assumption $\mathcal{F} \subset A^1_+$ causes no loss of generality.) By a routine argument (essentially, conjugating the elements of $\mathcal{F}$ with an idempotent approximate unit), we may moreover assume that there is $h \in A^1_+$ such that $h a = a h = a$ for all $a \in \mathcal{F}$. Choose 
\begin{equation}
\label{2-3-7}
0<\e< \eta/3
\end{equation}
 so small that the conditions in (3) will imply 
\begin{equation}
\label{2-3-2}
\|(\varphi^{(i)}(b^{(i)}_{h}))^{\frac{1}{2}} \, a \, (\varphi^{(i)}(b^{(i)}_{h}))^{\frac{1}{2}}- \varphi^{(i)}(b_a^{(i)})\| < \frac{\eta^2}{9(d+1)^2}
\end{equation}
(this is possible since $(\varphi^{(i)}(b^{(i)}_{h}))^{\frac{1}{2}}$ can be approximated by polynomials in $\varphi^{(i)}(b^{(i)}_{h})$, which in turn will almost commute with $a$; see \cite[Proposition~1.8]{Win10}, for instance). 

Now take $\varphi:F \to A$ as in the proposition for this $\mathcal{F}$, $h$, and $\varepsilon$. It follows from the structure theorem that an order zero map  (just like a $^*$-homomorphism) out of a simple C$^*$-algebra is either zero or injective. Therefore, by dropping some matrix summands of $F$ if necessary, we may assume each of the order zero maps $\varphi^{(i)}$ to be injective. 

Now if $E$ is a matrix block of $F^{(i)}$, then $\mathbf{1}_E$ is a minimal central projection in $F^{(i)}$ and the map
\[
\bar{\varphi}_E := \|\varphi^{(i)}(\mathbf{1}_E)\|^{-1} \cdot \varphi^{(i)}|_E(\, . \, ): E \longrightarrow A
\]
is isometric. Note that we can also write this map as
\begin{equation}
\label{2-3-7-1}
\bar{\varphi}_E (x) = \varphi^{(i)}(d_E^{-1} \, x)
\end{equation}
for $x \in E$, where $d_E:= \| \varphi^{(i)}(\mathbf{1}_E)\| \cdot \mathbf{1}_E  \in Z(F^{(i)})$ is an element in the centre of $F^{(i)}$ which is invertible in $E$. Again by the structure theorem for order zero maps there is an isomorphism $\theta: \mathrm{C}^*(\bar{\varphi}_E(E)) \to C_0(W,E)$ for some closed subset $W \subset (0,1]$ such that $\theta \circ \bar{\varphi}_E (x) = \mathrm{id}_W \cdot x$ for $x \in E$. Since we forced $\bar{\varphi}_E$ to be isometric, this implies that $1 \in W$, and we have a $^*$-homomorphism
\[
\bar{\psi}_E := \mathrm{ev}_1 \circ \theta: \mathrm{C}^*(\bar{\varphi}_E(E)) \longrightarrow E
\] 
(where $\mathrm{ev}_1$ denotes evaluation at $1$) satisfying
\begin{equation}
\label{2-3-5}
\bar{\psi}_E \circ \bar{\varphi}_E = \mathrm{id}_E.
\end{equation}
Now we may use Arveson's theorem to find a c.p.c.\  extension 
\[
\tilde{\psi}_E: A \longrightarrow E
\]
of $\bar{\psi}_E$ to all of $A$.

We may do the same for each matrix block of each $F^{(i)}$ to obtain invertible elements $d^{(i)} \in Z(F^{(i)})_+^1$ such that the c.p.\ order zero maps
\begin{equation}
\label{2-3-4}
\bar{\varphi}^{(i)} (\, . \, ):=  \varphi^{(i)}((d^{(i)})^{-1}  \, . \, ): F^{(i)} \longrightarrow A
\end{equation}
are isometric, and such that there are $^*$-homomorphisms 
\[
\bar{\psi}^{(i)}: \mathrm{C}^*(\bar{\varphi}^{(i)}(F^{(i)})) \longrightarrow F^{(i)}
\]
which extend to c.p.c.\ maps $\tilde{\psi}^{(i)}: A \to F^{(i)}$ with  $\tilde{\psi}^{(i)} \circ \bar{\varphi}^{(i)} = \mathrm{id}_{F^{(i)}}$. Define projections
\begin{equation}
\label{2-3-3}
q^{(i)}:= \chi_{(\frac{\eta}{3(d+1)},1]}(d^{(i)}) \in Z(F^{(i)})
\end{equation}
(with $\chi_{(\frac{\eta}{3(d+1)},1]}$ the characteristic function on the interval $(\frac{\eta}{3(d+1)},1]$)  and note that 
\begin{equation}
\label{2-3-6}
\|\varphi^{(i)}(q^{(i)} x) - \varphi^{(i)}(x)\| \le \frac{\eta}{3(d+1)} \|x\|
\end{equation}
for $x \in F^{(i)}$. We now define c.p.\ maps $\psi^{(i)}: A \to F^{(i)}$ by
\begin{equation}
\label{2-3-1}
\psi^{(i)}(\, . \,) := q^{(i)} (d^{(i)})^{-1} \tilde{\psi}^{(i)} ((\varphi^{(i)} (b^{(i)}_{h}))^\frac{1}{2} \; . \; (\varphi^{(i)} (b^{(i)}_{h}))^\frac{1}{2} )
\end{equation}
and $\psi:= \oplus_{i=0}^d \psi^{(i)}: A \to F$. 
The $\psi^{(i)}$ and therefore also $\psi$ are contractive since, for any $b \in A^1_+$,
\begin{align*}
\psi^{(i)}(b) & 
\stackrel{\eqref{2-3-1}}{\le} q^{(i)} (d^{(i)})^{-1} \bar{\psi}^{(i)} (\varphi^{(i)} (b^{(i)}_{h}))\\
& \stackrel{\phantom{\eqref{2-3-1}}}{=} q^{(i)} (d^{(i)})^{-1} \bar{\psi}^{(i)} \varphi^{(i)} ((d^{(i)})^{-1} d^{(i)} b^{(i)}_{h})\\
& \stackrel{\;\eqref{2-3-7-1}\;}{=} q^{(i)} (d^{(i)})^{-1} \bar{\psi}^{(i)} \bar{\varphi}^{(i)} ( d^{(i)} b^{(i)}_{h})\\
& \stackrel{\;\eqref{2-3-5}\;}{=} q^{(i)} (d^{(i)})^{-1}  d^{(i)} b^{(i)}_{h}\\
& \stackrel{\phantom{\eqref{2-3-1}}}{\le} \mathbf{1}_{F^{(i)}}.
\end{align*}
We now compute for $a \in \mathcal{F}$
\begin{align*}
\| a - \varphi \psi(a)\|  &  \stackrel{\ref{prop:dim-without-psi}(3)}{\le} \Big\| \sum\nolimits_i \varphi^{(i)}(b_a^{(i)}) - \varphi^{(i)} \psi^{(i)}(a)\Big\| + \e \\
 &  \stackrel{\eqref{2-3-1}}{=} \Big\|\sum\nolimits_i \varphi^{(i)}  (b_a^{(i)}) \\
 & \mathrel{\phantom{\stackrel{\eqref{2-3-1}}{=}}}{}  - \varphi^{(i)} (q^{(i)} (d^{(i)})^{-1}  \tilde{\psi}^{(i)} ((\varphi^{(i)} (b^{(i)}_{h}))^\frac{1}{2} \; a \; (\varphi^{(i)} (b^{(i)}_{h}))^\frac{1}{2} ))\Big\| \\
 &  \mathrel{\phantom{\stackrel{\eqref{2-3-1}}{=}}}{}+ \e \\
&  \stackrel{\;\eqref{2-3-2}\;}{\le} \Big\|\sum\nolimits_i \varphi^{(i)}  (b_a^{(i)}) - \varphi^{(i)} (q^{(i)} (d^{(i)})^{-1}  \bar{\psi}^{(i)} ( \varphi^{(i)} (b^{(i)}_{a}) ))\Big\| \\
 &   \mathrel{\phantom{\stackrel{\eqref{2-3-1}}{=}}}{}   + (d+1) \Big(\frac{\eta}{3(d+1)}\Big)^{-1} \frac{\eta^2}{9(d+1)^2} + \e \\
& \stackrel{\;\eqref{2-3-4}\;}{=}  \Big\|\sum\nolimits_i \varphi^{(i)}  (b_a^{(i)}) - \varphi^{(i)} (q^{(i)} (d^{(i)})^{-1}  \bar{\psi}^{(i)} ( \bar{\varphi}^{(i)} (d^{(i)} b^{(i)}_{a}) ))\Big\| \\
 &   \mathrel{\phantom{\stackrel{\eqref{2-3-1}}{=}}}{}    + \frac{\eta}{3} + \e \\
& \stackrel{\;\eqref{2-3-5}\;}{=}  \Big\|\sum\nolimits_i \varphi^{(i)}  (b_a^{(i)}) - \varphi^{(i)} (q^{(i)} b^{(i)}_{a} )\Big\|     + \frac{\eta}{3} + \e \\
 &  \stackrel{\;\eqref{2-3-6}\;}{\le}  (d+1) \frac{\eta}{3(d+1)}  + \frac{\eta}{3}  + \e\\
 & \stackrel{\;\eqref{2-3-7}\;}{<} \eta.
\end{align*}
This shows that $(F, \psi, \varphi)$ is a c.p.\ approximation witnessing $\ndim A \le d$ for $\mathcal{F}$ within $\eta$. \\

We now turn to the statement about diagonal dimension. The forward implication works exactly as above, upon noting that the existence of $D_F$ and \eqref{prop:dim-without-psi-eq} is already built into the initial approximations for $(D_A \subset A)$ and survives the modification along the lines of \cite[Proposition 4.2]{Win12} carried out at the beginning of this proof; cf.\ Remark~\ref{rem:almost_order_zero}(iii). 

The reverse implication essentially follows from the construction above; it only remains to show that $\psi(D_A) \subset D_F$ provided we have $\varphi^{(i)}(  \mathcal{N}_{F^{(i)}}(D_{F^{(i)}}))\subset \mathcal{N}_A(D_A)$ and $b^{(i)}_{h} \in D_{F^{(i)}}$ for some masa $D_F \subset F$.

To see this, suppose for a contradiction that $\psi(D_A) \not\subset D_F$. Since $D_F \subset F$ is maximal abelian, this implies there are $a \in D_A$ and a rank one projection $e \in D_{F^{(i)}}$ for some $i$ such that 
\[
e \psi^{(i)}(a) \neq \psi^{(i)}(a) e.
\] 
Since $q^{(i)}$ and $d^{(i)}$ are central in $F^{(i)}$, by \eqref{2-3-1} this furthermore implies that
\[
e \, \tilde{\psi}^{(i)} ((\varphi^{(i)} (b^{(i)}_{h}))^\frac{1}{2} \; a \; (\varphi^{(i)} (b^{(i)}_{h}))^\frac{1}{2} ) \neq \tilde{\psi}^{(i)} ((\varphi^{(i)} (b^{(i)}_{h}))^\frac{1}{2} \; a \; (\varphi^{(i)} (b^{(i)}_{h}))^\frac{1}{2} ) \, e,
\]
which also reads as 
\[
e \, \bar{\psi}^{(i)} (\varphi^{(i)} (b^{(i)}_{h}) \; a ) \neq \bar{\psi}^{(i)} ( a \; \varphi^{(i)} (b^{(i)}_{h}) ) \, e,
\]
because $a$ and $\varphi^{(i)} (b^{(i)}_{h})$ lie in $D_A$, where we have used that by Proposition~\ref{prop:maps-normalisers}(i) $\varphi^{(i)}(D_{F^{(i)}}) \subset D_A$ for each $i$. 
Next note that by \eqref{2-3-5}
\[
e = \bar{\psi}^{(i)} \bar{\varphi}^{(i)}(e)
\]
and that $\bar{\varphi}^{(i)}(e)$ lies in the multiplicative domain of $\bar{\psi}^{(i)}$. 
But then we have
\[
\bar{\psi}^{(i)} (\bar{\varphi}^{(i)} (e) \, \varphi^{(i)} (b^{(i)}_{h}) \; a ) \neq \bar{\psi}^{(i)} ( a \; \varphi^{(i)} (b^{(i)}_{h}) \, \bar{\varphi}^{(i)}(e)).
\]
Since $\bar{\varphi}^{(i)}(D_{F^{(i)}}) = \varphi^{(i)}(D_{F^{(i)}}) \subset D_A$ we have now reached 
a contradiction to $a, \varphi^{(i)} (b^{(i)}_{h}), \bar{\varphi}^{(i)}(e) \in D_A$ and $D_A$ being abelian.
\end{proof}

\begin{rem}
In Definition~\ref{defn:dim_diag}, instead of asking for condition (4) (i.e., $\psi(D_A) \subset D_F$), at least in the unital case one can equivalently ask for the \emph{a priori} weaker condition
\begin{enumerate}
\item[(4')] $\psi(\mathbf{1}_A) \subset D_F$.
\end{enumerate}
This follows from Proposition~\ref{prop:dim-without-psi} in connection with Remark~\ref{rem:almost_order_zero}(iii). It will be useful when we derive permanence properties for diagonal dimension in Section~\ref{sec:permanence}.
\end{rem}

Diagonal dimension is particularly relevant in the nondegenerate case. The proposition below gives a useful characterisation of this situation. 

\begin{prop} \label{prop:abelian}
Let $(D_A\subset A)$ be a sub-$\mathrm{C}^*$-algebra with $D_A$ abelian. If $\ddim(D_A\subset A) < \infty$, then the following are equivalent:
\begin{enumerate}[label=\rm{(\roman*})]
	\item $(D_A\subset A)$ is nondegenerate;
	\item every system of approximations $(F_\lambda, D_{F_\lambda}, \psi_\lambda, \varphi_\lambda )_{\lambda\in \Lambda}$ witnessing $\ddim(D_A\subset A) < \infty$ satisfies $\varphi_\lambda(D_{F_\lambda})\subset D_A$ for all $\lambda$;
	\item there exists a system of approximations $(F_\lambda, D_{F_\lambda}, \psi_\lambda, \varphi_\lambda )_{\lambda\in \Lambda}$ witnessing $\ddim(D_A\subset A) < \infty$ such that $\varphi_\lambda(D_{F_\lambda})\subset D_A$ for all $\lambda$.
\end{enumerate}
\end{prop}
\begin{proof}
(i) $\Longrightarrow$ (ii) follows from Proposition~\ref{prop:maps-normalisers}(i) in connection with condition (5') of \ref{rem:almost_order_zero}(v). 

(ii) $\Longrightarrow$ (iii) is trivial. 

(iii) $\Longrightarrow$ (i): Let $(D_A)_+^{<1}$ be the upward-directed set of all positive elements in $D_A$ with norm strictly less than one. We claim that $(D_A)^{<1}_+$ is an approximate unit for $A$. To see this, let $\mathcal{F} \subset A^1_+$ be a finite subset and $\e > 0$. There is $h \in A^1_+$ such that for each $a \in \mathcal{F}$ there is a $b_a \in A^1_+$ with $hb_a = b_ah = b_a$ and $\|a - b_a\| < \e$. For each $\lambda$ define c.p.c.\ maps $\hat{\psi}_\lambda : A \to \hat{F}_\lambda$ and $\hat{\varphi}_\lambda : \hat{F}_\lambda \to A$ as in Remark~\ref{rem:almost_order_zero}(ii). We then have $\|\hat{\varphi}_\lambda \hat{\psi}_\lambda (b_a) - b_a\| \to 0$, $\|\hat{\varphi}_\lambda \hat{\psi}_\lambda (b_a^2) - b_a^2\| \to 0$ and $\|\hat{\varphi}_\lambda (\mathbf{1}_{\hat{F}_\lambda}) - h\| = \|\varphi_\lambda \psi_\lambda (h) - h\| \to 0$ for each $a \in \mathcal{F}$. By \cite[Lemma 3.6]{KW04} we have for each $a \in \mathcal{F}$ 
\[
\|\hat{\varphi}_\lambda \hat{\psi}_\lambda (b_a) -\hat{\varphi}_\lambda(\mathbf{1}_{\hat{F}_\lambda})  \hat{\varphi}_\lambda \hat{\psi}_\lambda (b_a)\| 
\longrightarrow 0.
\]
As a consequence, there is an index $\lambda \in \Lambda$ such that
\[
\| \hat{\varphi}_\lambda(\mathbf{1}_{\hat{F}_\lambda})a - a \| \leq 5\e
\]
for all $a \in \mathcal{F}$. 

Furthermore we have 
\[
0 \le \hat{\varphi}_\lambda(\mathbf{1}_{\hat{F}_\lambda}) = \varphi_\lambda \psi_\lambda (h) \le \varphi_\lambda(\mathbf{1}_{F_\lambda}),
\]
and so $\hat{\varphi}_\lambda(\mathbf{1}_{\hat{F}_\lambda})$ lies in the hereditary subalgebra generated by $\varphi_\lambda(\mathbf{1}_{F_\lambda})$. But then there is $d \in \mathrm{C}^*(\varphi_\lambda(\mathbf{1}_{F_\lambda}))_+$ with $\|d\| < 1$ such that 
\[
\| d \hat{\varphi}_\lambda(\mathbf{1}_{\hat{F}_\lambda}) - \hat{\varphi}_\lambda(\mathbf{1}_{\hat{F}_\lambda})\| < \e,
\]
whence
\[
\|da - a \| \le \|d \hat{\varphi}_\lambda(\mathbf{1}_{\hat{F}_\lambda}) a - \hat{\varphi}_\lambda(\mathbf{1}_{\hat{F}_\lambda}) a\| + 10 \e < 11 \e. 
\]
By the hypothesis of (iii) we have $\varphi_\lambda(\mathbf{1}_{F_\lambda}) \in D_A$, whence $d \in D_A$ and we have shown that $D_A^{<1}$ is an approximate unit for $A$.
\end{proof}

In the remainder of this section we show that a nondegenerate sub-$\mathrm{C}^*$-algebra $(D_A \subset A)$ with finite diagonal dimension is indeed a diagonal (Theorem~\ref{principal grp}). The proof amounts to verifying the conditions of Definition \ref{defn:diagonal} one by one. We are grateful to Sel\c{c}uk Barlak and Xin Li for drawing our attention to the unique extension property and for showing us the argument for the proposition below.

\begin{prop}\label{unique ext}
	Let $(D_A\subset A)$ be a nondegenerate sub-$\mathrm{C}^*$-algebra with $\ddim(D_A\subset A) = d < \infty$. Then $D_A \subset A$ is a masa with the unique extension property.
\end{prop}
\begin{proof}
We only verify the unique extension property, since by \cite[Corollary 2.7]{ABG82} and \cite[Remark 2.6(iii)]{ABG82} this will imply that $D_A$ is a masa in $A$. (There is also a direct argument which we do not spell out at this point, because it proceeds along similar lines as Proposition~\ref{prop:cond_exp_formula}.)

Suppose $\varrho_1, \varrho_2 \in \mathrm{PS}(A)$ are pure states on $A$ extending a pure state $\sigma \in \mathrm{PS}(D_A)$ on $D_A$. Then $D_A$ is in the multiplicative domain of the $\varrho_i$ since $\sigma$ is a character on $D_A$, i.e., $\varrho_i(bd) = \varrho_i(b) \varrho_i(d)$ for  $b \in A$, $d \in D_A$. (The statement is an easy consequence of Stinespring's theorem; cf.\ \cite[Lemma~3.5]{KW04}.) But then for any $b \in A$, $d \in D_A$ and $i=1,2$ one has 
\[
\varrho_i(bd) = \varrho_i(b) \varrho_i(d) =  \varrho_i(d) \varrho_i(b) = \varrho_i(db),
\]
which means that the $\varrho_i$ vanish on $\text{span}[D_A,A]$. Therefore, in order to prove $\varrho_1 = \varrho_2$ it suffices to show that the linear subspace $D_A + \text{span}[D_A,A]$ is dense in $A$. 

So let $a\in A$ and $\e > 0$ be given. Find a c.p.\ approximation $(F, D_F, \psi,\varphi)$ witnessing $\ddim(D_A\subset A) = d$ for $(\{a\}, \e)$. If we identify $(D_F\subset F)$ with the direct sum $\big(\bigoplus_{j=1}^N D_{r^{(j)}} \subset  \bigoplus_{j=1}^N M_{r^{(j)}}\big)$ and write $\{ e^{(j)}_{k,\ell} \}$ for the standard matrix units, then
	\[
	\psi(a) = \sum\nolimits_{j=1}^N \sum\nolimits_{k,\ell=1}^{r^{(j)}} \psi(a)^{(j)}_{k,\ell} \cdot  e^{(j)}_{k,\ell},
	\]
	where each complex number $\psi(a)^{(j)}_{k,\ell}$ is the $(k,\ell)$-matrix entry of the $j$th summand of $\psi(a)$.	For each $j \in \{1,\ldots,N\}$ and $k,\ell\in \{1,\ldots,r^{(j)}\}$ with $k\neq \ell$, 
	\begin{align*}
		\varphi(e^{(j)}_{k,\ell}) &= \varphi(e^{(j)}_{k,k} e^{(j)}_{k,\ell} ) - \varphi( e^{(j)}_{k,\ell}e^{(j)}_{k,k})  \\
		&= \varphi^{\frac{1}{2}}(e^{(j)}_{k,k})\varphi^{\frac{1}{2}}(e^{(j)}_{k,\ell}) - \varphi^{\frac{1}{2}}(e^{(j)}_{k,\ell}) \varphi^{\frac{1}{2}}(e^{(j)}_{k,k}) \in [D_A,A]
	\end{align*}
	(here $\varphi^{\frac{1}{2}}$ is defined using the functional calculus for order zero maps; cf.\ \cite{WZ09}). Since
	\begin{align*}
		\varphi\psi(a) & = \sum\nolimits_{j=1}^N \sum\nolimits_{k,\ell=1}^{r^{(j)}} \psi(a)^{(j)}_{k,\ell} \cdot \varphi(e^{(j)}_{k,\ell}) \\
		&  = \sum\nolimits_{j=1}^N \sum\nolimits_{k=1}^{r^{(j)}} \psi(a)^{(j)}_{k,k} \cdot \varphi(e^{(j)}_{k,k}) + \sum\nolimits_{j=1}^N \sum\nolimits_{k\neq \ell} \psi(a)^{(j)}_{k,\ell} \cdot \varphi(e^{(j)}_{k,\ell}),
	\end{align*}
and since $\varphi(e^{(j)}_{k,k}) \in D_A$ by Proposition~\ref{prop:abelian}, the element $\varphi\psi(a)$ belongs to the linear subspace $D_A + \text{span}[D_A,A]$.
\end{proof}

We now turn to the conditional expectation from $A$ onto $D_A$. Just like maximality of the abelian subalgebra one can deduce its existence from the unique extension property (see \cite[Corollary 2.7]{ABG82}). Below we give a direct argument which yields a concrete formula for the expectation; this will allow us to also prove faithfulness in Proposition \ref{prop:faithful_cond}.

\begin{prop} \label{prop:cond_exp_formula}
	Let $(D_A\subset A)$ be a nondegenerate sub-$\mathrm{C}^*$-subalgebra with $\ddim(D_A\subset A) = d <\infty$. Let $( F_\lambda, D_{F_\lambda}, \psi_\lambda, \varphi_{\lambda} )_{\lambda \in \Lambda}$ be a system of c.p.\ approximations witnessing $\ddim(D_A\subset A) = d$, and for each $\lambda$ let $E_\lambda$ be the (unique) conditional expectation from $F_\lambda$ onto $D_{F_\lambda}$.
	
	Then the map $\Phi:A\to D_A$ given by the formula
	\begin{equation}
	\label{cond_exp_formula}
	\Phi(a) = \lim_\lambda \varphi_{\lambda} E_\lambda\psi_\lambda(a)
	\end{equation}
	is a well-defined conditional expectation from $A$ onto $D_A$. The conditional expectation is uniquely determined and does not depend on the particular choice of the system $( F_\lambda, D_{F_\lambda}, \psi_\lambda, \varphi_{\lambda} )_{\lambda \in \Lambda}$.
\end{prop}
\begin{proof}
Consider the quotient $\mathrm{C}^*$-algebra 
\[
Q:= \prod\nolimits_\Lambda D_A \big/ \bigoplus\nolimits_\Lambda D_A
\]
and note that the $\mathrm{C}^*$-norm on $Q$ is given by $\|[(a_\lambda)_{\lambda \in \Lambda}]\|_Q = \limsup_\lambda \|a_\lambda\|$.

Note also that we have a c.p.\ map 
\[
\bar{\Phi}: A \longrightarrow Q
\]
given by 
\[
\bar{\Phi}(a):= [(\varphi_\lambda E_\lambda \psi_\lambda(a))_{\lambda \in \Lambda}]
\]
(each composition $\varphi_\lambda E_\lambda \psi_\lambda$ is c.p.\ and has norm at most $d+1$, and we have $\varphi_\lambda(D_{F_\lambda}) \subset D_A$ by Proposition~\ref{prop:maps-normalisers}(i)). 

Now since $E_\lambda|_{D_{F_\lambda}} \circ \psi_\lambda|_{D_A}  = \psi_\lambda|_{D_A}$, and since $\varphi_\lambda \psi_\lambda \to \mathrm{id}_A$, we see that for $d \in D_A$
\begin{equation}\label{expectation1}
\bar{\Phi}(d) = \iota_{D_A}(d),
\end{equation}
where $\iota_{D_A}:D_A \hookrightarrow Q$ is the canonical embedding. Since $D_A$ contains an approximate unit for $A$ this also implies that $\bar{\Phi}$ is in fact contractive. Moreover, since $\bar{\Phi}$ is multiplicative on $D_A$, Stinespring's theorem implies that $D_A$ is in the multiplicative domain of $\bar{\Phi}$, i.e., 
\[
\bar{\Phi}(ad) = \bar{\Phi}(a) \bar{\Phi}(d)
\]
for $a \in A$, $d\in D_A$; cf.\ \cite[Lemma~3.5]{KW04}.

Now fix some index $\lambda_0 \in \Lambda$ and let $v \in F^{(i)}_{\lambda_0}$ be some off-diagonal matrix unit (so that $vv^*$ and $v^*v$ are orthogonal rank one projections in $D_{F_{\lambda_0}^{(i)}}$). Define the c.p.c.\ order zero map $(\varphi^{(i)}_{\lambda_0})^\frac{1}{3}$ using order zero functional calculus (cf.\ \cite{WZ09}), then 
\[
\varphi^{(i)}_{\lambda_0}(v) = (\varphi^{(i)}_{\lambda_0})^\frac{1}{3}(vv^*) (\varphi^{(i)}_{\lambda_0})^\frac{1}{3}(v) (\varphi^{(i)}_{\lambda_0})^\frac{1}{3}(v^*v)
\] 
and 
\[
(\varphi^{(i)}_{\lambda_0})^\frac{1}{3}(D_{F_{\lambda_0}^{(i)}}) \subset D_A.
\]
It follows that
\begin{align*}
\bar{\Phi}(\varphi^{(i)}_{\lambda_0}(v)) & = \bar{\Phi}((\varphi^{(i)}_{\lambda_0})^\frac{1}{3}(vv^*)) \, \bar{\Phi}((\varphi^{(i)}_{\lambda_0})^\frac{1}{3}(v)) \, \bar{\Phi}((\varphi^{(i)}_{\lambda_0})^\frac{1}{3}(v^*v)) \\
& = \iota_{D_A}((\varphi^{(i)}_{\lambda_0})^\frac{1}{3}(vv^*)\, (\varphi^{(i)}_{\lambda_0})^\frac{1}{3}(v^*v)) \, \bar{\Phi}((\varphi^{(i)}_{\lambda_0})^\frac{1}{3}(v)) \\
& = 0.
\end{align*}
From this and \eqref{expectation1} we conclude 
\[
\bar{\Phi} \circ \varphi_{\lambda_0} = \iota_{D_A} \circ \varphi_{\lambda_0} \circ E_{\lambda_0}.
\]
Since $\lambda_0$ was arbitrary and $A = \overline{\bigcup_\lambda \varphi_\lambda \psi_\lambda(A)}$ we now have 
\[
\bar{\Phi}(A) \subset \iota_{D_A}(D_A)
\]
and, for $a \in A$,
\begin{align*}
\bar{\Phi} (a) & = \bar{\Phi} (\lim_\lambda \varphi_\lambda \psi_\lambda(a))\\
& = \lim_\lambda \bar{\Phi}  \varphi_\lambda \psi_\lambda(a)\\
& = \lim_\lambda \iota_{D_A}  \varphi_\lambda E_\lambda \psi_\lambda(a)\\
& = \iota_{D_A} \lim_\lambda \varphi_\lambda E_\lambda \psi_\lambda(a),
\end{align*}
where for the second equality we have used continuity of $\bar{\Phi}$ and for the last equality we have used that $\iota_{D_A}$ is an isometry. But this means that
\[
(a \mapsto \Phi(a) := \lim_\lambda \varphi_\lambda E_\lambda \psi_\lambda(a))
\]
is indeed a well-defined c.p.c.\ map $A \to \iota_{D_A} (D_A) \cong D_A$. $\Phi$ is a conditional expectation since $\Phi|_{D_A} = \text{id}_{D_A}$. As a consequence of the unique extension property (cf.\ Proposition~\ref{unique ext}) , there can be only one conditional expectation from $A$ onto $D_A$.
\end{proof}

Let us next turn to faithfulness of $\Phi$. The proof is nontrivial and ultimately relies on the rigidity provided by the normaliser condition \ref{defn:dim_diag}(5). We try to highlight this phenomenon in the following Lemma, which we think is worth pointing out since it turns approximate commutation into exact commutation, independently of the matrix sizes. This is special since commutativity relations are generally not at all robust under small permutations, as was demonstrated for example by Voiculescu with his famous almost commuting unitaries.

\begin{lem}
\label{lem:rank1}
Let $r \in \mathbb{N}$, $0 \le \alpha < {1}/{144}$, and let $q \in M_r$ be a rank one projection such that, for every $c \in D_r^1$, $\|qc - cq\| \le \alpha$.

Then, there is a uniquely determined rank one projection $d \in D_r$ with $\|d - q\| \le 6 \alpha^\frac{1}{2}$. 
\end{lem}

\begin{proof}
Let $e_k \in D_r$, $k=1,\ldots,r$ be the standard rank one projections and set 
\[
p_i := \sum\nolimits_{k=1}^i e_k \in D_r, \; p_0 := 0.
\]
For every $i = 1,\ldots,r$ by our hypothesis we have 
\[
\textstyle
\|p_i q p_i - p_i q p_i p_i q p_i \| \le \alpha \le {1}/{4}.
\]
But then the interval $(\frac{1}{2} - (\frac{1}{4} - \alpha)^\frac{1}{2}, \frac{1}{2} + (\frac{1}{4} - \alpha)^\frac{1}{2})$ has empty intersection with the spectrum of $p_i q p_i$. As a consequence, $\|q p_i q\| = \|p_i q p_i\|$ is either smaller than $\frac{1}{2} - (\frac{1}{4} - \alpha)^\frac{1}{2}$ or larger than $\frac{1}{2} + (\frac{1}{4} - \alpha)^\frac{1}{2}$.

Next observe that $0 = \|q p_0 q\| \le \|q p_1 q\| \le \ldots \le \|q p_r q\| = 1$, and so there is some $\bar{\imath} \in \{1,\ldots,r\}$ such that $\|q p_{\bar{\imath}-1} q\| \le \frac{1}{2} - (\frac{1}{4} - \alpha)^\frac{1}{2}$ and $\|q p_{\bar{\imath}} q\| \ge \frac{1}{2} + (\frac{1}{4} - \alpha)^\frac{1}{2}$. Now
\[
\textstyle
\|q e_{\bar{\imath}} q\| = \|q p_{\bar{\imath}} q - q p_{\bar{\imath}-1} q \| \ge 2  ({1}/{4} - \alpha)^\frac{1}{2},
\]
whence ($q$ has rank one)
\[
\textstyle
\|q - q e_{\bar{\imath}} q\| = 1 - \| q e_{\bar{\imath}} q\| \le 1 - 2  ({1}/{4} - \alpha )^\frac{1}{2}.
\]
Since $\|q e_{\bar{\imath}} q\| = \| e_{\bar{\imath}} q e_{\bar{\imath}}\|$, in the same manner one gets 
\[
\textstyle
\| e_{\bar{\imath}} - e_{\bar{\imath}} q e_{\bar{\imath}}\| \le 1 - 2  ({1}/{4} - \alpha )^\frac{1}{2}.
\]

We can now estimate
\begin{align*}
\|q - e_{\bar{\imath}} \| & \le \|q - q e_{\bar{\imath}}q \| + \|q e_{\bar{\imath}} q - e_{\bar{\imath}} q e_{\bar{\imath}}\| + \|e_{\bar{\imath}} q e_{\bar{\imath}} - e_{\bar{\imath}}\| \\
& \textstyle \le  1 - 2  ({1}/{4} - \alpha)^\frac{1}{2}  + 2 \alpha + 1 - 2  ({1}/{4} - \alpha)^\frac{1}{2} \\
& \le 2 \alpha^\frac{1}{2} + 2 \alpha + 2 \alpha^\frac{1}{2} \\
&\textstyle \le 6 \alpha^\frac{1}{2} \\
&< {1}/{2}
\end{align*}
and take $d := e_{\bar{\imath}}$.

For $i \neq \bar{\imath}$ we have $\|q - e_i\| \ge \|e_i - e_{\bar{\imath}} \| - \|q - e_{\bar{\imath}}\| \ge 1 - 6 \alpha^\frac{1}{2} > {1}/{2}$, so $e_{\bar{\imath}}$ is the only rank one projection in $D_r$ with $\|q - e_{\bar{\imath}}\| \le 6 \alpha^\frac{1}{2}$.
\end{proof}

\begin{prop} \label{prop:faithful_cond}
	Let $(D_A\subset A)$ be a nondegenerate sub-$\mathrm{C}^*$-algebra with $\ddim(D_A\subset A) = d < \infty$. Then the conditional expectation $\Phi:A\to D_A$, as defined in Proposition \ref{prop:cond_exp_formula}, is faithful.
\end{prop}
\begin{proof}
Let $( F_\lambda, D_{F_\lambda},  \psi_\lambda, \varphi_{\lambda} )_{\lambda \in \Lambda}$ be a system of c.p\ approximations witnessing $\ddim(A,D_A) = d$. We may assume that the $F_\lambda$ are nonzero (if $A= \{0\}$ there is nothing to show); by Remarks~\ref{rem:almost_order_zero}(ii) and (iii) we may also assume that the compositions $\varphi_{\lambda} \psi_\lambda$ are contractive and that the $\psi_\lambda$ are approximately order zero, so that the induced map $\bar{\psi}: A \to \prod F_\lambda / \bigoplus F_\lambda$ is c.p.c.\ order zero with $\bar{\psi}(D_A) \subset \prod D_{F_\lambda} / \bigoplus D_{F_\lambda}$. We need to show that $\Phi(a) \neq 0$ for any positive nonzero $a \in A$; we may assume $\|a\| = 1$.

Choose $\lambda_0 \in \Lambda$ such that
\begin{equation}
\label{2-9-6}
\| \varphi_{\lambda_0} \psi_{\lambda_0} (a) - a \| < \frac{1}{16(d+1)^3}.
\end{equation}
Since $\varphi_{\lambda_0}$ is a sum of $d+1$ c.p.c.\ order zero maps, there is a matrix summand $(M_R, D_R)$ of $(F_{\lambda_0},D_{F_{\lambda_0}})$ such that 
\[
\| \varphi_{\lambda_0}(\mathbf{1}_R \psi_{\lambda_0}(a)) \| \ge \frac{1}{2(d+1)},
\] 
and so 
\begin{equation}
\label{2-9-7}
\|\mathbf{1}_R \psi_{\lambda_0}(a)\| \ge \frac{1}{2(d+1)}, \quad \mbox{ and } \quad \| \varphi_{\lambda_0}(\mathbf{1}_R ) \| \ge \frac{1}{2(d+1)}.
\end{equation}
Moreover, $\varphi_{\lambda_0}|_{M_R}$ is c.p.c.\ order zero with $\varphi_{\lambda_0}(D_R) \subset D_A$, and we have just seen that $\| \varphi_{\lambda_0}|_{M_R}\| \ge {1}/{(2(d+1))}$.

Note that $\bar{\psi} \varphi_{\lambda_0}|_{M_R}: M_R \to \prod F_\lambda / \bigoplus F_\lambda$ is c.p.c.\ order zero and that, for each $\lambda$, $\psi_\lambda \varphi_{\lambda_0}(D_R) \subset D_{F_\lambda}$, whence $\bar{\psi} \varphi_{\lambda_0}(D_R) \subset \prod D_{F_\lambda} / \bigoplus D_{F_\lambda} \subset \prod F_\lambda / \bigoplus F_\lambda$. 

Also, since $\|\varphi_\lambda \psi_\lambda(\varphi_{\lambda_0} (e_{k,l})) - \varphi_{\lambda_0} (e_{k,l})\| \to 0$ (where $e_{k,l}$ denote the standard matrix units of $M_R$), and since each $\varphi_\lambda \psi_\lambda = \sum_{i=0}^d \varphi_\lambda^{(i)} \psi_\lambda^{(i)}$, there is $\lambda_1 \in \Lambda$ such that for each $\lambda \ge \lambda_1$ there is $i \in \{0,\ldots,d\}$ with
\begin{equation}
\label{faithful-1}
\textstyle
\|\psi_\lambda^{(i)}\varphi_{\lambda_0}(e_{k,l}) \| \ge \| \varphi_\lambda^{(i)} \psi_\lambda^{(i)}\varphi_{\lambda_0}(e_{k,l})\| \ge {1}/{(4 (d+1)^2)}.
\end{equation}
Now for each $\lambda \ge \lambda_1$ there is a rank one projection $d_{1,\lambda} \in D_{F_\lambda}$ such that 
\[
\textstyle
d_{1,\lambda} \psi_\lambda \varphi_{\lambda_0}(e_{1,1}) = \psi_\lambda \varphi_{\lambda_0}(e_{1,1}) d_{1,\lambda} = \| \psi_\lambda \varphi_{\lambda_0}(e_{1,1}) \| \cdot d_{1,\lambda}  
\]
and such that 
\begin{equation}
\label{faithful-4}
\|\varphi_\lambda(d_{1,\lambda})\| \ge \frac{1}{4 (d+1)^2}.
\end{equation}
Set
\[
d_1 := [(d_{1,\lambda})_\lambda] \in \prod D_{F_\lambda} \big/ \bigoplus D_{F_\lambda}
\]
(the values of $d_{1,\lambda}$ for $\lambda \le \lambda_1$ do not matter in the quotient), then $d_1$ is a projection and we have 
\begin{equation}
\label{faithful-3}
d_1 \ge d_1 \bar{\psi} \varphi_{\lambda_0}(e_{1,1}) =   \bar{\psi} \varphi_{\lambda_0}(e_{1,1}) d_1 \ge \frac{1}{4 (d+1)^2} \cdot d_1.
\end{equation}
This also implies that there is a positive element $u$ such that
\[
d_1 \le u \le 4(d+1)^2 \cdot d_1 \in \prod D_{F_\lambda} \big/ \bigoplus D_{F_\lambda}
\]
with 
\[
u \bar{\psi} \varphi_{\lambda_0}(e_{1,1}) =  \bar{\psi} \varphi_{\lambda_0}(e_{1,1}) u = d_1.
\]
Moreover, $u$ has a lift $(u_\lambda)_\lambda \in \prod D_{F_\lambda}$ with $d_{1,\lambda} \le u_\lambda \le 4(d+1)^2 \cdot d_{1,\lambda}$ for $\lambda \ge \lambda_1$. (This in fact forces the $u_\lambda$ to be scalar multiples of the $d_{1,\lambda}$ for $\lambda \ge \lambda_1$; we can set $u_\lambda = 0$ for all other values of $\lambda$.)

For $k=2,\ldots,R$ we now define 
\begin{equation}
\label{2-9-1}
d_k := \bar{\psi} \varphi_{\lambda_0} (e_{k,1}) u d_1 u \bar{\psi} \varphi_{\lambda_0} (e_{1,k}) \in \prod F_\lambda \big/ \bigoplus F_\lambda,
\end{equation}
then
\begin{align}
d_k^2 & =  \bar{\psi} \varphi_{\lambda_0} (e_{k,1}) u d_1 u \bar{\psi} \varphi_{\lambda_0} (e_{1,k})  \bar{\psi} \varphi_{\lambda_0} (e_{k,1}) u d_1 u \bar{\psi} \varphi_{\lambda_0} (e_{1,k}) \nonumber \\
& = \bar{\psi} \varphi_{\lambda_0} (e_{k,1}) u d_1 u \bar{\psi} \varphi_{\lambda_0} (e_{1,1})  \bar{\psi} \varphi_{\lambda_0} (e_{1,1}) u d_1 u \bar{\psi} \varphi_{\lambda_0} (e_{1,k}) \nonumber \\
& = \bar{\psi} \varphi_{\lambda_0} (e_{k,1}) u d_1 u \bar{\psi} \varphi_{\lambda_0} (e_{1,k}) \nonumber \\
& = d_k, \label{2-9-2}
\end{align}
where we have used that $\bar{\psi} \varphi_{\lambda_0}$ is order zero and that $d_1$ is a projection. Therefore, each $d_k$ is a projection which is Murray--von Neumann equivalent to $d_1$ via the partial isometry $v_k:= u \bar{\psi} \varphi_{\lambda_0} (e_{1,k})$. 

If $(\tilde{d}_{k,\lambda})_\lambda \in \prod F_\lambda$ is a positive contractive lift for $d_k$, then we may take $\bar{d}_{k,\lambda} := \chi_{[1/2,1]}(\tilde{d}_{k,\lambda}) \in F_\lambda $ for each $\lambda$; it is clear that $(\bar{d}_{k,\lambda})_\lambda$ is a lift of $d_k$ as well. Moreover, the $\bar{d}_{k,\lambda}$ are projections of rank one for $\lambda \ge \lambda_2$ for some sufficiently large $\lambda_2$, since by \eqref{2-9-1} and \eqref{2-9-2} the $\bar{d}_{k,\lambda}$ and $d_{1,\lambda}$ are Murray--von Neumann equivalent for $\lambda$ large enough. 

For each $k=2, \ldots,R$ we now have lifts of the $d_k$ in $\prod F_\lambda$ consisting of rank one projections (at least for $\lambda$ large enough). However, we will need lifts in $\prod D_{F_\lambda}$, as we have for $d_1$. Arranging this is our next task.\\

\noindent
{\bf Claim:} For each $k=2, \ldots,R$ and $\lambda \in \Lambda$ there are projections $d_{k,\lambda} \in D_{F_\lambda}$ of rank at most one such that $(d_{k,\lambda})_\lambda$ lifts $d_k$.\\

 To prove the claim, note first that 
 \[
 \| \bar{d}_{k,\lambda} - \psi_\lambda \varphi_{\lambda_0}(e_{k,1}) u_\lambda^2 \psi_\lambda  \varphi_{\lambda_0}(e_{1,k}) \| \stackrel{\lambda}{\longrightarrow} 0,
 \]
whence 
 \[
 \| \varphi_\lambda (\bar{d}_{k,\lambda}) - \varphi_\lambda(\psi_\lambda \varphi_{\lambda_0}(e_{k,1}) u_\lambda^2 \psi_\lambda  \varphi_{\lambda_0}(e_{1,k})) \| \stackrel{\lambda}{\longrightarrow} 0.
 \]
But
\[
\|  \varphi_\lambda(\psi_\lambda \varphi_{\lambda_0}(e_{k,1}) u_\lambda^2 \psi_\lambda  \varphi_{\lambda_0}(e_{1,k})) - \varphi_{\lambda_0}(e_{k,1}) \varphi_\lambda (u_\lambda^2) \varphi_{\lambda_0}(e_{1,k}) \| \stackrel{\lambda}{\longrightarrow} 0
\]
by \cite[Lemma~3.6]{KW04} (the $\psi_\lambda \varphi_{\lambda_0}(e_{k,l})$ are approximately in the multiplicative domain of the $\varphi_\lambda$), and $\varphi_{\lambda_0}(e_{k,1}) \varphi_\lambda (u_\lambda^2) \varphi_{\lambda_0}(e_{1,k}) \in D_A$ for each $\lambda$, since the $\varphi_{\lambda_0}(e_{1,k})$ normalise $D_A$. This in particular implies that for every given 
\[
0 < \beta < \frac{1}{144 \cdot 16(d+1)^4}
\]
there is $\lambda_3 \ge \lambda_2 \in \Lambda$ such that, for every $\lambda \ge \lambda_3$ and every $c \in D_{F_\lambda}^1$ in the same matrix summand (say $B_\lambda$) as the rank one projection $\bar{d}_{k,\lambda}$, we have 
\begin{equation}
\label{2-9-3}
\| (\varphi_\lambda|_{B_\lambda})^2 ([\bar{d}_{k,\lambda},c])\| = \| [\varphi_\lambda(\bar{d}_{k,\lambda}), \varphi_\lambda(c)]\| < \beta.
\end{equation}
Here we have used that $\varphi_\lambda$ is order zero when restricted to any matrix summand of $F_\lambda$, so that $(\varphi_\lambda|_{B_\lambda})^2 (\bar{d}_{k,\lambda} \, c) = \varphi_\lambda (\bar{d}_{k,\lambda})  \, \varphi_\lambda (c)$.
For each $k \in \{2,\ldots,R\}$ and $\lambda \ge \lambda_3$, $\bar{d}_{k,\lambda}$ is in the same matrix summand $B_\lambda$ as $d_{1,\lambda}$, and so for any $x \in B_\lambda$ by \eqref{faithful-4} we have 
\begin{equation}
\label{2-9-4}
\textstyle
\| (\varphi_\lambda|_{B_\lambda})^2(x) \| = \| \varphi_\lambda(x) \| \| \varphi_\lambda(\mathbf{1}_{B_\lambda}) \| \ge \|x\| \cdot {1}/{(16(d+1)^4)}.
\end{equation}
From \eqref{2-9-3} and \eqref{2-9-4} we see that for $\lambda \ge \lambda_3$ and $c \in D_{F_\lambda}^1$ as above 
\[
\textstyle
\| [ \bar{d}_{k,\lambda}, c] \| \le \beta \cdot 16 (d+1)^4 < {1}/{144}.
\]
Now Lemma~\ref{lem:rank1} implies that for every $\lambda \ge \lambda_3$ there is a uniquely determined rank one projection $d_{k,\lambda} \in D_{F_\lambda}$ such that
\[
\| d_{k,\lambda} - \bar{d}_{k,\lambda} \| \le 24 \beta^\frac{1}{2} (d+1)^2.
\]

 Now the $d_{k,\lambda}$ are uniquely determined for $\lambda \ge \lambda_3$; upon decreasing $\beta$, the same argument shows that 
 $\| \bar{d}_{k,\lambda} - d_{k,\lambda} \| \stackrel{\lambda}{\longrightarrow} 0$. (Decreasing $\beta$ requires increasing $\lambda_3$, but this does not cause a problem since the $d_{k,\lambda}$ are uniquely determined.) We have now found projections $d_{k,\lambda} \in D_{F_\lambda}$ of rank at most one (at least for $\lambda \ge \lambda_3$ -- for all other $\lambda$ we may take $d_{k,\lambda} = 0$) such that $[(d_{k,\lambda})_\lambda] = d_k$, and so the claim above is established. \\

We are now ready to define a map 
\[
\Delta: \prod F_\lambda \big/ \bigoplus F_\lambda \longrightarrow \prod F_\lambda \big/ \bigoplus F_\lambda
\]
by
\[
\Delta( \, . \,) := \Big(\sum\nolimits_{k=1}^R d_k \Big) (\, . \, ) \Big(\sum\nolimits_{k=1}^R d_k \Big).
\]
By our construction, and since $\bar{\psi} \varphi_{\lambda_0}|_{M_R}$ is order zero, we have for $m,n \in \{1,\ldots,R\}$
\begin{align*}
 \Big(\sum\nolimits_{k=1}^R d_k \Big) \bar{\psi} \varphi_{\lambda_0}(e_{m,n}) & = d_m \bar{\psi} \varphi_{\lambda_0}(e_{m,n}) \\
 & = \bar{\psi} \varphi_{\lambda_0}(e_{m,n}) d_n \\
 &  = \bar{\psi} \varphi_{\lambda_0}(e_{m,n})  \Big(\sum\nolimits_{k=1}^R d_k \Big).
\end{align*}
It follows that $\Delta \bar{\psi} \varphi_{\lambda_0}|_{M_R}$ is also c.p.c.\ order zero, and from \eqref{faithful-3} we see that for $x \in M_R$
\begin{equation}
\label{faithful-2}
\|\Delta \bar{\psi} \varphi_{\lambda_0} (x)\|  \ge \frac{1}{4(d+1)^2} \|x\|.
\end{equation}

By the formula \eqref{cond_exp_formula} for the conditional expectation $\Phi$ in connection with \eqref{faithful-4} and the fact that $\varphi_\lambda|_{B_\lambda}$ is order zero we have
\begin{align}
\|\Phi(a)\| & \stackrel{\eqref{cond_exp_formula}}{=} \lim_\lambda \|\varphi_\lambda E_\lambda \psi_\lambda(a) \| \nonumber\\
& \stackrel{\phantom{\eqref{cond_exp_formula}}}{\ge} \max_k \limsup_\lambda  \| \varphi_\lambda (d_{k,\lambda} \psi_\lambda(a) d_{k,\lambda}) \|\nonumber \\
& \stackrel{\eqref{faithful-4}}{\ge} \frac{1}{4(d+1)^2} \max_k \|d_k \bar{\psi}(a) d_k \| \nonumber \\
& \stackrel{\phantom{\eqref{cond_exp_formula}}}{\ge} \frac{1}{4(d+1)^2} \max_k \|d_k \bar{\psi}(a) \|^2 . \label{2-9-9}
\end{align}
Suppose for a contradiction that
\begin{equation}
\label{2-9-10}
\max_k \|d_k \bar{\psi}(a)  \| \le \frac{1}{R \cdot 16 \cdot (d+1)^3},
\end{equation}
then 
\begin{equation}
\label{2-9-5}
 \Big\| \sum\nolimits_{k=1}^R d_k \bar{\psi}(a)  \Big\| \le \frac{1}{16 \cdot (d+1)^3},
\end{equation} 
whence
\[
\|\Delta \bar{\psi} \varphi_{\lambda_0}(\mathbf{1}_R \psi_{\lambda_0}(a)) \| \le \|\Delta \bar{\psi} \varphi_{\lambda_0}( \psi_{\lambda_0}(a)) \| \stackrel{\eqref{2-9-6},\eqref{2-9-5}}{<}  \frac{2}{ 16 \cdot (d+1)^3}
\]
and, with \eqref{faithful-2}, 
\begin{equation}
\label{2-9-8}
\frac{1}{4(d+1)^2} \cdot \| \mathbf{1}_R \psi_{\lambda_0}(a) \| < \frac{2}{ 16 \cdot (d+1)^3}.
\end{equation}
From this we get
\[
\frac{1}{8(d+1)^3} \stackrel{\eqref{2-9-7}}{\le} \frac{1}{4(d+1)^2} \|\mathbf{1}_R \psi_{\lambda_0}(a)\| \stackrel{\eqref{2-9-8}}{<} \frac{2}{ 16 \cdot (d+1)^3},
\]
a contradiction. This means \eqref{2-9-10} cannot hold, i.e., 
\begin{equation}
\label{2-9-11}
\max_k \|d_k \bar{\psi}(a)  \| > \frac{1}{R \cdot 16 \cdot (d+1)^3}.	
\end{equation}
We therefore have 
\[
\| \Phi(a) \| \stackrel{\eqref{2-9-9},\eqref{2-9-11}}{\ge} \frac{1}{4(d+1)^2} \cdot \frac{1}{(R \cdot 16 \cdot (d+1)^3)^{2}}  > 0.
\]
\end{proof}

By Propositions~\ref{unique ext}, \ref{prop:cond_exp_formula}, and \ref{prop:faithful_cond}, finite diagonal dimension of a nondegenerate sub-$\mathrm{C}^*$-algebra implies the unique extension property and the existence of a unique and faithful conditional expectation. Regularity was observed in Remark~\ref{rem:almost_order_zero}(vi). Therefore the conditions of Definition~\ref{defn:diagonal} are satisfied and we get:

\begin{thm}\label{principal grp}
	Let $(D_A\subset A)$ be a nondegenerate sub-$\mathrm{C}^*$-algebra with finite diagonal dimension.  Then $D_A$ is a diagonal in $A$. \end{thm}

\section{Permanence properties}
\label{sec:permanence}

\noindent
In this section we study permanence properties of diagonal dimension. Our proofs largely follow those of the respective statements for nuclear dimension. However, since the approximating maps have to preserve subalgebras exactly and not just approximately, we have to make some additional effort when it comes to quotients and inductive limits. On the other hand, for hereditary subalgebras the argument becomes slightly easier due to the built-in rigidity of diagonal dimension. For the reader's convenience we collect all the permanence properties in the theorem below. The proofs will be given in the corresponding subsections below.

\begin{thm} \label{thm:permanence} Let $(D_A\subset A)$ and $(D_B\subset B)$ be two sub-$\mathrm{C}^*$-algebras with $D_A$ and $D_B$ abelian. Then we have the following permanence properties. 
\begin{enumerate}[label=\rm{(\roman*})]
	\item {\rm Direct sums:} 
	\begin{align*}
	&\ddim(D_A\oplus D_B\subset A\oplus B) \\
	&
	= \max\{ \ddim(D_A\subset A), \ddim(D_B\subset B) \}.
	\end{align*}
	
	\item {\rm Tensor products:}\footnote{In case the dimensions take finite values the involved algebras are nuclear, so there is no need to specify which tensor product we are working with.  Also, recall that we follow G\'abor Szab\'o's notation and write $\dim^{+1}(\, .\,)$ for $\dim(\, .\,)+1$ for various dimension theories whenever there are products involved.} 
	\begin{align*}
	&\ddim^{+1}(D_A\otimes D_B\subset A\otimes B) \\
	&
	\leq \ddim^{+1}(D_A\subset A)\cdot \ddim^{+1}{(D_B\subset B)}.
	\end{align*}
		
	\item {\rm Hereditary subalgebras:} Suppose $B\subset A$ and $D_B\subset D_A$ are hereditary subalgebras and that $(D_B\subset B)$ is nondegenerate. Then
	\[
	\ddim( D_B\subset B  )\leq  \ddim(D_A\subset A).
	\]
	
		\item {\rm Unitisations:} Suppose $(D_A\subset A)$ is nondegenerate and let $A^\sim$ be the smallest unitisation of $A$. Under the canonical identification of the unitisation $D_A^\sim$ with $D_A+\C \unit_{A^\sim}$ we then have 
	\[
	\ddim(D_A\subset A) = \ddim(D_A^\sim \subset A^\sim ).
	\]

	\item {\rm Quotients:} Suppose $(D_A\subset A)$ is nondegenerate, and suppose there is a surjective $^*$-homomorphism $\pi:A\to B$ such that $\pi(D_A) = D_B$. Then 
	\[
	\ddim(D_B\subset B)\leq \ddim(D_A\subset A).
	\]
	
	\item {\rm Inductive limits:} Let $((D_i \subset A_i), \varrho_{i,j})_{i,j \in I}$ be an inductive system of sub-$\mathrm{C}^*$-algebras with $D_i$ abelian and nondegenerate, and such that  $\varrho_{i,j}$ maps $\cn_{A_i}(D_i)$ into $\cn_{A_j}(D_j)$ for all $i$ and $j$.
	Then $( \varinjlim D_i \subset \varinjlim A_i   )$ is canonically a sub-$\mathrm{C}^*$-algebra and we have
	\[
	\ddim( \varinjlim D_i \subset \varinjlim A_i   ) \leq \liminf_{i\in I} ( \ddim(D_i \subset A_i ) ).
	\]
	
	\item {\rm Stabilisations:} Diagonal dimension is invariant under stabilisation with matrices or algebras of compact operators with their canonical diagonals. More precisely, 
	\begin{align*}
	&\ddim(D_A \subset A ) \\
	&= \ddim(D_A\otimes D_n \subset A\otimes M_n ) \\
	&= \ddim( D_A\otimes c_0(\N) \subset A\otimes \mathcal{K}(\ell^2(\N)) ).
	\end{align*}

\end{enumerate}
\end{thm}

\vspace{1em}
\noindent
{\it Proof} of Theorem \ref{thm:permanence}(i): Direct sums.
\vspace{1em}

\noindent
The argument is exactly the same as for nuclear dimension; see \cite[Proposition~2.3]{WZ10} and \cite[Proposition~2.10]{Win:cpr}. Conditions \ref{defn:dim_diag}(4) and (5) will automatically be satisfied since we may take $D_{F_A \oplus F_B} = D_{F_A} \oplus D_{F_B}$ and we have $\mathcal{N}_{A \oplus B}(D_A \oplus D_B) = \mathcal{N}_{A}(D_A)  \oplus \mathcal{N}_{B}(D_B)$. \hfill $\square$

\vspace{1em}
\noindent
{\it Proof} of Theorem \ref{thm:permanence}(ii): Tensor products.
\vspace{1em}

\noindent
Again the proof is the same as for nuclear dimension; see \cite[Proposition 2.3]{WZ10} and \cite[Proposition~3.1.4]{Win09}. Condition \ref{defn:dim_diag}(4) will be satisfied since $\psi(D_A \otimes D_B) \subset D_{F_A} \otimes D_{F_B} = D_{F_A \otimes F_B}$. Condition \ref{defn:dim_diag}(5) follows automatically since matrix units of $(D_{F_A} \otimes D_{F_B} \subset F_A \otimes F_B)$ are tensor products of matrix units of $(D_{F_A} \subset F_A)$ and $(D_{F_B} \subset F_B)$, respectively, hence map to tensor products of normalisers under the order zero maps $\varphi_A^{(i)} \otimes \varphi_B^{(j)}$, which in turn are again normalisers in $\mathcal{N}_{A \otimes B}(D_A \otimes D_B)$. \hfill $\square$

\vspace{1em}
\noindent
{\it Proof} of Theorem \ref{thm:permanence}(iii): Hereditary subalgebras.
\vspace{1em}

\noindent
Let $\mathcal{F} \subset B^1_+$ finite and $0< \e < 1$ be given. We may assume without loss of generality that there is $h \in (D_B)^1_+$ such that $hb = b$ for all $b \in \mathcal{F}$. 

Fix $\eta:= \left( \frac{\e}{3d+4} \right)^4$ and take an approximation $(F, D_F, \psi, \varphi)$ witnessing $\ddim(D_A \subset A) \le d$ for $\mathcal{F} \cup \{h\}$ within $\eta$.

Set $\beta:= 2 \eta^\frac{1}{2}$ and define projections $q := \chi_{(\beta,1]}(\psi(h)) \in D_F$ and $q^{(i)} := \chi_{(\beta,1]}(\psi^{(i)}(h)) \in D_{F^{(i)}}$ (where $\chi_{(\beta,1]}$ denotes the characteristic function of the interval $(\beta,1]$). 

Define $\bar{F} := q Fq$ with summands $\bar{F}^{(i)} := q^{(i)} F^{(i)} q^{(i)}$, $D_{\bar{F}} := q D_F q$, and a map $\bar{\psi}( \, . \, ):= q \psi(\, .\, )q: B \to \bar{F}$ with summands $\bar{\psi}^{(i)}: B \to \bar{F}^{(i)}$. Note that $D_{\bar{F}}$ is diagonal in $\bar{F}$, that $\bar{\psi}$ is c.p.c., and that $\bar{\psi}(D_B) \subset D_{\bar{F}}$. 

With $f_\beta \in C_0((0,1])$ given by $f_\beta(t) := (t-\beta)_+$ define c.p.c.\ order zero maps $\bar{\varphi}^{(i)} := f_\beta(\varphi^{(i)})|_{\bar{F}^{(i)}}: \bar{F}^{(i)} \to A$ and a c.p.\ map $\bar{\varphi} := \sum_{i=0}^d \bar{\varphi}^{(i)}: \bar{F} \to A$. 

Note that for $b \in \mathcal{F}$ we have
\begin{align*}
\| \bar{\psi}(b) - \psi(b) \| & = \| q \psi(hbh)q - \psi(hbh) \| \\
& \le \| q \psi(hbh)(\mathbf{1} - q) \| + \| (\mathbf{1} - q) \psi(hbh) \| \\
& \le 2 \| (\mathbf{1} - q) \psi(hbh) \| \\
& =  2 \| (\mathbf{1} - q) \psi(hbh)^2 (\mathbf{1} - q) \|^\frac{1}{2} \\
& \le 2 \| (\mathbf{1} - q) \psi(hbh) (\mathbf{1} - q) \|^\frac{1}{2} \\
& \le 2 \| (\mathbf{1} - q) \psi(h) (\mathbf{1} - q) \|^\frac{1}{2} \\
& \le 2 \beta^\frac{1}{2}.
\end{align*}
As a consequence,
\begin{align*}
\|b - \bar{\varphi}\bar{\psi}(b) \| & < \|\varphi \psi(b) - \bar{\varphi}\bar{\psi}(b) \| + \eta \\
& \le \|\varphi \psi(b) - \varphi \bar{\psi}(b) \| + \eta + (d+1) \beta \\
& \le (d+1) 2 \beta^\frac{1}{2} + \eta + (d+1) \beta \\
& < (3 d + 4) \eta^\frac{1}{4} \\
& = \e,
\end{align*}
and so $(\bar{F},\bar{\psi},\bar{\varphi})$ is indeed a c.p.\ approximation for $\mathcal{F}$ within $\e$.

Since matrix units of $(D_{\bar{F}^{(i)}} \subset \bar{F}^{(i)})$ are also matrix units of $(D_{F^{(i)}} \subset F^{(i)})$, the conditions on $\varphi^{(i)}$ (see condition (5) of Definition~\ref{defn:dim_diag}) imply that for every matrix unit $e$ of $(D_{\bar{F}^{(i)}} \subset \bar{F}^{(i)})$ the image $\varphi^{(i)}(e)$ normalises $D_A$. Since $D_B \subset D_A$ is hereditary, it follows that $\varphi^{(i)}(e)$ normalises $D_B$. Similarly, $\varphi^{(i)}(q^{(i)})$ normalises $D_B$ for each $i$, and so does every element of $\mathrm{C}^*(\varphi^{(i)}(q^{(i)}))_+$ by Lemma~\ref{lem:positive_normaliser}. But by the structure theorem for order zero maps and the definition of order zero functional calculus we have  
\[
\bar{\varphi}^{(i)}(e) = f_\beta(\varphi^{(i)})(e) \in \overline{\mathrm{C}^*(\varphi^{(i)}(q^{(i)}))_+ \varphi^{(i)}(e)},
\]
and so $\bar{\varphi}^{(i)}(e)$ normalises $D_B$ for every $i$ and every matrix unit $e$ of $(D_{\bar{F}^{(i)}} \subset \bar{F}^{(i)})$. 

It remains to check that $\bar{\varphi}(\bar{F}) \subset B$; for this it will suffice to prove that $\bar{\varphi}^{(i)}(q^{(i)}) \in B$ for each $i$.

Define $\bar{D}^{(i)} := \mathrm{C}^*( \varphi^{(i)}(q^{(i)}), D_B)$ and note that by Lemma~\ref{lem:positive_normaliser} $\bar{D}^{(i)}$ is commutative. Moreover, we have
\[
0 \le \bar{\varphi}^{(i)}(q^{(i)}) = f_\beta (\varphi^{(i)}(q^{(i)})) \in \bar{D}^{(i)}_+.
\]
Let $\varrho$ be a character on $\bar{D}^{(i)}$, and suppose $\varrho(\bar{\varphi}^{(i)}(q^{(i)})) > 0$. Then
\[
\textstyle
\beta \le \varrho(\varphi^{(i)}(q^{(i)})) \le \frac{1}{\beta} \varrho(\varphi^{(i)}(\psi^{(i)}(h))), 
\]
whence
\[
\beta^2 \le \varrho(\varphi \psi(h)) \le \varrho(h) + \eta < \varrho(h) + \beta^2.
\]
It follows that $\varrho(h) > 0$, and since $\varrho$ was arbitrary and $B \subset A$ is hereditary this now implies that $\bar{\varphi}^{(i)}(q^{(i)}) \in \overline{h \bar{D}^{(i)} h} \subset \overline{h A h} \subset B$, as desired. \hfill $\square$

\vspace{1em}
\noindent
{\it Proof} of Theorem \ref{thm:permanence}(iv): Unitisations.
\vspace{1em}

\noindent
Since $A$ and $D_A$ are ideals of $A^\sim$ and $D_A^\sim$, respectively, Theorem~\ref{thm:permanence}(iii) shows that $\ddim(D_A\subset A)\leq \ddim(D_A^\sim \subset A^\sim )$. 
	
	To prove the reverse inequality, we essentially follow the argument for decomposition rank in \cite[Proposition 3.11]{KW04}. We spell out the details, since here we cannot assume the downwards map $\psi$ to be almost multiplicative, and at the same time we have to keep track of the diagonal.
	
It is enough to show that given $0 < \e \le1$ and $b_1,\ldots,b_m \in A_+^1$ we can find a c.p.\ approximation $(\bar{F}, D_{\bar{F}}, \bar{\psi}, \bar{\varphi}  )$ witnessing $\ddim(D_A^\sim \subset A^\sim ) \leq  \ddim(D_A \subset A) =: d < \infty$ for $\{b_1,\ldots,b_m, \unit_{A^\sim}  \}$ within $\e$. By nondegeneracy of $D_A$, upon slightly perturbing the $b_j$ we may moreover assume that there exist positive contractions $h_0,h_1\in D_A$ such that $h_0h_1 = h_1$ and $h_1b_j = b_j$ for each $j\in \{0,\ldots,m\}$.
	
	Choose $0< \gamma < (\varepsilon/6)^4$, and set $\delta := \gamma^4 /4$. Take a c.p.\ approximation 
	\[
	\Big( F = \bigoplus\nolimits_{i=0}^d F^{(i)} , D_F = \bigoplus\nolimits_{i=0}^d D_{F^{(i)}} , \psi, \varphi  \Big)
	\]
	 witnessing $\ddim(D_A\subset A) = d$ for $\{h_0,h_1,b_1,\ldots,b_m\}$ within $\delta$. 	Set
	\[
	p := \chi_{[\delta^{\frac{1}{2}},1]}(\psi(h_1))\in F,
	\]
	where $\chi_{[\delta^{\frac{1}{2}},1]}$ denotes the characteristic function of the interval $[\delta^{\frac{1}{2}}, 1]$.
	
	Notice that we have $p \leq \delta^{-\frac{1}{2}} \cdot \psi(h_1)$. Let $p^{(0)} := \unit_{F^{(0)}}p$ and $q := p^{(0)}+ \sum_{i=1}^d \unit_{F^{(i)}}$. Observing that $\varphi(p^{(0)})$, $\varphi(1_F)$, $h_0$, and $\unit_{A^\sim}$ all pairwise commute, we then have
	\begin{align}
		\| \varphi(p^{(0)})(\unit_{A^\sim} - h_0) \| &= \| (\unit_{A^\sim} - h_0)^{\frac{1}{2}} \varphi(p^{(0)}) (\unit_{A^\sim} - h_0)^{\frac{1}{2}}\| \nonumber \\
		&\leq \| (\unit_{A^\sim}-h_0)^{\frac{1}{2}} \varphi(p)(\unit_{A^\sim}-h_0)^{\frac{1}{2}} \| \nonumber \\
		&\leq \| (\unit_{A^\sim}-h_0)^{\frac{1}{2}} \varphi\psi(h_1)(\unit_{A^\sim} - h_0)^{\frac{1}{2}} \| \cdot \frac{1}{\delta^{\frac{1}{2}}} \nonumber \\
		&< \| (\unit_{A^\sim} - h_0)^{\frac{1}{2}} h_1 (\unit_{A^\sim}-h_0)^{\frac{1}{2}} \|\cdot \frac{1}{\delta^{\frac{1}{2}}} + \delta^{\frac{1}{2}} \nonumber \\
		&= \delta^{\frac{1}{2}}	 \label{3-1-unitisations}	
	\end{align}
	and
	\begin{align}
		\| \varphi(p^{(0)})(\unit_{A^\sim} - \varphi(q)) \| &\stackrel{\phantom{\eqref{3-1-unitisations}}}{=} \| \varphi(p^{(0)})( \unit_{A^\sim} - \varphi(\unit_F) + \varphi( \unit_{F^{(0)}} - p^{(0)})    ) \| \nonumber \\
		&\stackrel{\phantom{\eqref{3-1-unitisations}}}{=} \|  \varphi(p^{(0)})(\unit_{A^\sim} - \varphi(\unit_F)) \| \nonumber \\
		 &\stackrel{\phantom{\eqref{3-1-unitisations}}}{=} \| \varphi(p^{(0)})^{\frac{1}{2}} (\unit_{A^\sim} - \varphi(\unit_F) )  \varphi(p^{(0)})^{\frac{1}{2}} \| \nonumber \\
		&\stackrel{\phantom{\eqref{3-1-unitisations}}}{\leq} \|  \varphi(p^{(0)})^{\frac{1}{2}} (\unit_{A^\sim} - \varphi\psi(h_0) )  \varphi(p^{(0)})^{\frac{1}{2}} \| \nonumber \\	
		&\stackrel{\phantom{\eqref{3-1-unitisations}}}{<} \|  \varphi(p^{(0)})^{\frac{1}{2}} (\unit_{A^\sim} - h_0 )  \varphi(p^{(0)})^{\frac{1}{2}} \| + \delta \nonumber\\
		&\stackrel{\phantom{\eqref{3-1-unitisations}}}{=} \| \varphi(p^{(0)})(\unit_{A^\sim} - h_0) \| + \delta \nonumber \\
		&\stackrel{\eqref{3-1-unitisations}}{<} 2\delta^{\frac{1}{2}}. \label{v10-3-2}
	\end{align}
Define a continuous function $g \in C([0,1])$ by
	\begin{equation*} \label{eq:g-function}
g(t) = \begin{cases} 0 & 0\leq t\leq 2^\frac{1}{2} \delta^\frac{1}{4} \\
t & 2 \delta^\frac{1}{4}\leq t \leq  1 \\
\text{linear} & 2^\frac{1}{2} \delta^\frac{1}{4} < t < 2 \delta^\frac{1}{4}
\end{cases}
\end{equation*}
and note that (by \eqref{v10-3-2}, with functional calculus and since $\varphi(p^{(0)})$ and $\unit_{A^\sim} - \varphi(q)$ commute)
\[
 g(\varphi(p^{(0)})) \perp g(\unit_{A^\sim} - \varphi(q))
\]
and that
\[
	   \| g(\varphi(p^{(0)})) - \varphi(p^{(0)}) \|, \| g(\unit_{A^\sim} -\varphi(q) ) - (\unit_{A^\sim} -\varphi(q))   \|    < 2^{\frac{1}{2}} \delta^{\frac{1}{4}} = \gamma.
\]
Applying order zero functional calculus to the c.p.c.\ order zero map $\varphi\vert_{p^{(0)}Fp^{(0)}}:p^{(0)}Fp^{(0)}\to A$ we obtain a c.p.c.\ order zero map
	\[
	\hat{\varphi}^{(0)}:= g(\varphi\vert_{p^{(0)}Fp^{(0)}}) : p^{(0)}Fp^{(0)} \to \overline{ g(\varphi(p^{(0)}))A g(\varphi(p^{(0)})) }  
	\]
satisfying
\begin{enumerate}
		\item $\hat{\varphi}^{(0)} ( p^{(0)}D_F p^{(0)}  ) \subset  \overline{ g(\varphi(p^{(0)}))D_A g(\varphi(p^{(0)})) }$,
		\item $\hat{\varphi}^{(0)}(v)D_A \hat{\varphi}^{(0)}(v)^* \subset   \overline{ g(\varphi(p^{(0)}))D_A g(\varphi(p^{(0)})) }$ for every matrix unit $v$ in $p^{(0)}F p^{(0)}$ with respect to $p^{(0)}D_F p^{(0)}$, and
		\item $\| \hat{\varphi}^{(0)}(x) - \varphi(x) \| \le \| g(\varphi(p^{(0)})) - \varphi(p^{(0)})\| \leq 2^\frac{1}{2}\delta^{\frac{1}{4}} = \gamma$ for any positive contraction $x$ in $p^{(0)}Fp^{(0)}$.
	\end{enumerate}
Here, for (2) we have used Proposition~\ref{prop:maps-normalisers}(ii), and for (3) we have used order zero functional calculus.

	We are now ready to construct the desired c.p.\ approximation for $(D_A^\sim \subset A^\sim )$. Let 
	\[
	\bar{F} := qFq \oplus \C = p^{(0)}Fp^{(0)}\oplus \Big(   \bigoplus\nolimits_{i=1}^d F^{(i)}   \Big) \oplus \C
	\]
	and 
	\[
	D_{\bar{F}} := qD_Fq \oplus \C.
	\]
	
	Define a map $\bar{\psi}:A^\sim\to \bar{F}$ by
	\[
	\bar{\psi}(a+\lambda \unit_{A^\sim}) := q\psi(a)q + \lambda(q \oplus \unit_{\C}).
	\]
	Then $\bar{\psi}$ is unital and completely positive (see, for example, \cite[Proposition 2.2.1]{BO08}) and satisfies $\bar{\psi}(D_A^\sim)\subset D_{\bar{F}}$. 
	
	Define $\bar{\varphi}:\bar{F}\to A^\sim$ by
	\[
	\bar{\varphi}\vert_{ p^{(0)}Fp^{(0)} } := \hat{\varphi}^{(0)}, \; \bar{\varphi}\vert_{  \bigoplus_{i=1}^d F^{(i)} } := \varphi\vert_{  \bigoplus_{i=1}^d F^{(i)} } , \; \bar{\varphi}(\unit_{\C}) := g(\unit_{A^\sim} - \varphi(q)).   
	\]
	Then by construction  and 
	
	Then $\bar{\varphi}^{(i)}:= \bar{\varphi}\vert_{F^{(i)}}$, $i = 1,\ldots,d$, are all c.p.c.\ order zero; $\bar{\varphi}^{(0)}:= \bar{\varphi}\vert_{ p^{(0)}Fp^{(0)}\oplus \C}$ is also c.p.c.\ order zero since $\hat{\varphi}^{(0)}$ is, and since
	\[
	\bar{\varphi}^{(0)}(\unit_{\bar{F}^{(0)}}) = \hat{\varphi}^{(0)}(p^{(0)}) = g(\varphi(p^{(0)})) \perp g(\unit_{A^\sim} - \varphi(q)) = \bar{\varphi}^{(0)}(\unit_{\C}).
	\]
	It follows from Proposition~\ref{prop:maps-normalisers}(ii) that $\mathcal{N}_A(D_A)\subset  \mathcal{N}_{A^\sim}(D_A^\sim)$, and so we see that $\bar{\varphi}(v) \in \mathcal{N}_{A^\sim}(D_A^\sim)$ for each matrix unit in $\bar{F}$ with respect to $D_{\bar{F}}$. 
	
	It remains to show that $\bar{\varphi}\bar{\psi}$ approximates the identity map on the elements $b_1,\ldots,b_m$ and on $\unit_{A^\sim}$. For the unit, we have
	\begin{align*}
		&\| \bar{\varphi}\bar{\psi}(\unit_{A^\sim}) - \unit_{A^\sim} \| \\
		& = \| \bar{\varphi}(q\oplus \unit_{\C}) - \unit_{A^\sim} \| \\
		&  = \Big\|  \hat{\varphi}^{(0)}( p^{(0)}  ) + \varphi \Big( \sum\nolimits_{i=1}^d \unit_{F^{(i)}} \Big)  + g(\unit_{A^\sim} - \varphi(q)) -\unit_{A^\sim}  \Big\| \\
		& < \| \varphi(q) + g(\unit_{A^\sim} - \varphi(q))  - \unit_{A^\sim} \| + \gamma \\
		&< 2\gamma  < \e.
	\end{align*}
	For the elements $b_j$, we first compute
	\begin{align*}
		& \| \varphi( (\unit_{F^{(0)}}-p^{(0)} )\psi(b_j))\| \\
		& = \| \varphi( (\unit_{F^{(0)}}-p^{(0)} )\psi(b_j)) \cdot \varphi(\psi(b_j) (\unit_{F^{(0)}}-p^{(0)} )) \|^{\frac{1}{2}} \\
		&\leq \| \varphi(   (\unit_{F^{(0)}}-p^{(0)})\psi(b_j)^2(\unit_{F^{(0)}}-p^{(0)})          ) \|^{\frac{1}{2}} \\
		&\leq \| \varphi(   (\unit_{F^{(0)}}-p^{(0)})\psi(h_1)(\unit_{F^{(0)}}-p^{(0)})          ) \|^{\frac{1}{2}} \\
		&\leq \delta^{\frac{1}{4}}.
	\end{align*}
	It follows that
	\[
	\| \varphi(p^{(0)}\psi(b_j)   ) - \varphi( p^{(0)}\psi(b_j) p^{(0)} ) \| \leq 2\delta^{\frac{1}{4}},
	\]
	whence
	\[
	\|  \varphi( p^{(0)}\psi(b_j) p^{(0)} ) - \varphi(  \unit_{F^{(0)}}\psi(b_j) \unit_{F^{(0)}}                 ) \| \leq 3\delta^{\frac{1}{4}}.
	\]
	We now show the approximation for $b_j$:
	\begin{align*}
		&\| \bar{\varphi}\bar{\psi}(b_j) - b_j \| \\
		&= \|  \bar{\varphi}(q\psi(b_j)q ) - b_j \| \\
		& = \Big\| \bar{\varphi}^{(0)}( p^{(0)}\psi(b_j)p^{(0)}  ) + \varphi\Big(  \sum\nolimits_{i=1}^d \unit_{F^{(0)}}\psi(b_j)   \Big) - b_j \Big\| \\
		& \leq  \Big\|  \varphi( p^{(0)}\psi(b_j)p^{(0)}  ) + \varphi\Big(  \sum\nolimits_{i=1}^d \unit_{F^{(0)}}\psi(b_j)   \Big) - b_j      \Big\| + \gamma\\
		&\leq \| \varphi\psi(b_j) - b_j \| + \gamma + 3\delta^{\frac{1}{4}} \\
		&< \delta + \gamma + 3\delta^{\frac{1}{4}} \\
		&\leq 5\gamma^{\frac{1}{4}} \\
		& < \e.
	\end{align*}
	This completes the proof.\hfill $\square$

\vspace{1em}
\noindent
{\it Proof} of Theorem \ref{thm:permanence}(v): Quotients.
\vspace{1em}

\noindent
Upon replacing $D_A$, $A$, $D_B$, $B$ and $\pi$ by their smallest unitisations, in view of Theorem \ref{thm:permanence}(iv) we may assume that the algebras and the quotient map are unital. Now let a finite subset $\unit_B \in \mathcal{F}_B \subset B^1_+$ and $\e >0$ be given. Choose a finite subset $\unit_A \in \mathcal{F}_A \subset A^1_+$ lifting $\mathcal{F}_B$. Now if $\ddim(D_A \subset A) = d < \infty$ choose $D_F = D_F^{(0)} \oplus \ldots \oplus D_F^{(d)} \subset F= F^{(0)} \oplus \ldots \oplus F^{(d)}$ and $\varphi: F \to A$ as in Proposition~\ref{prop:dim-without-psi} (with $\mathcal{F}_A$ in place of $\mathcal{F}$). It is then clear that with $B$ in place of $A$ and $\pi \circ \varphi$ in place of $\varphi$ conditions (i), (ii) and (iii) of Proposition~\ref{prop:dim-without-psi} are satisfied. Moreover, since quotient maps send normalisers to normalisers, condition \eqref{prop:dim-without-psi-eq} also holds \emph{mutatis mutandis} (i.e., with $A$ in place of $B$ and $\pi \circ \varphi^{(i)}$ in place of $\varphi^{(i)}$), and so the reverse direction of Proposition~\ref{prop:dim-without-psi} entails that $\ddim(D_B \subset B) \le d$. \hfill $\square$

\vspace{1em}
\noindent
{\it Proof} of Theorem \ref{thm:permanence}(vi): Inductive limits.
\vspace{1em}

\noindent
Let us first reduce to the unital case. 
We have canonical inclusions of ideals $D_i \subset D_i^\sim$ and $A_i \subset A_i^\sim$, and likewise for the inductive limits $\varinjlim(D_i,\varrho_{i,j}) \subset (\varinjlim(D_i,\varrho_{i,j}))^\sim$ and $\varinjlim(A_i,\varrho_{i,j}) \subset (\varinjlim(A_i,\varrho_{i,j}))^\sim$. Denoting the canonical unitisation of $\varrho_{i,j}$ by  $\varrho_{i,j}^\sim$, there are canonical isomorphisms $(\varinjlim(D_i,\varrho_{i,j}))^\sim \cong \varinjlim(D_i^\sim,\varrho_{i,j}^\sim)$ and $(\varinjlim(A_i,\varrho_{i,j}))^\sim \cong\varinjlim(A_i^\sim,\varrho_{i,j}^\sim)$. Now from \ref{thm:permanence}(iv) it follows that 
\[
\ddim( \varinjlim D_i \subset \varinjlim A_i   ) = \ddim( \varinjlim D_i^\sim \subset \varinjlim A_i^\sim   ).
\]
Of course we also have
\[
\liminf_{i\in I} ( \ddim(D_i \subset A_i ) ) = \liminf_{i\in I} ( \ddim(D_i^\sim \subset A_i ^\sim) ).
\]
Next observe that by Proposition~\ref{prop:maps-normalisers}(i) we have $\varrho_{i,j}(D_i) \subset D_j$, hence also $\varrho_{i,j}^\sim(D_i^\sim) \subset D_j^\sim$. But then from Proposition~\ref{prop:maps-normalisers}(ii) it follows that the $\varrho_{i,j}^\sim$ map normalisers to normalisers. 

From this discussion we see that it suffices to prove the assertion for the unitised system, i.e.
\[
\ddim( \varinjlim D_i^\sim \subset \varinjlim A_i^\sim   ) \le \liminf_{i\in I} ( \ddim(D_i^\sim \subset A_i ^\sim) ),
\]
and that the unitised system still satisfies the crucial hypothesis, namely that the connecting maps preserve normalisers. In conclusion, we may impose the additional assumption that the algebras $D_i$ and $A_i$, as well as the connecting maps $\varrho_{i,j}$ are all unital to begin with.

\bigskip

Let $\unit_{\varinjlim A_i} \in \mathcal{F} \subset (\varinjlim A_i)^1_+$ be a finite subset and let $\e >0 $ be given. 
For each $j \in I$, let $\varrho_j: A_j \to \varinjlim A_i$ denote the limit map. Since $\bigcup_j \varrho_j(A_j)$ is dense in $\varinjlim A_i$, we may as well assume that $\mathcal{F} \subset \varrho_j(A_j)$ for some $j \in I$. But then $\mathcal{F}$ has a finite preimage $\bar{\mathcal{F}} \subset (A_j)^1_+$. Now apply Proposition~\ref{prop:dim-without-psi} to $(D_j \subset A_j)$ to find  $D_F \subset F = F^{(0)} \oplus \ldots \oplus F^{(d)}$ (where $\ddim(D_j \subset A_j) \le d < \infty$), a set $\{b_a \mid a \in \bar{\mathcal{F}}\} \subset F^1_+$ and a map $\varphi = \sum_{i=0}^d \varphi^{(i)}: F \to A_j$ satisfying \ref{prop:dim-without-psi}(i), (ii), (iii) (with $(D_j \subset A_j)$ in place of $(D_A \subset A)$ and with $\bar{\mathcal{F}}$ in place of $\mathcal{F}$) and 
\[
\varphi^{(i)}(  \mathcal{N}_{F^{(i)}}(D_{F^{(i)}}))\subset \mathcal{N}_{A_j}(D_j).
\]
But then the same set $\{b_a \mid a \in \bar{\mathcal{F}}\}$ and the map $\varrho_j \circ \varphi: F \to \varinjlim A_i$ satisfy \ref{prop:dim-without-psi}(i), (ii), (iii) (this time with $(\varinjlim D_i \subset \varinjlim A_i)$ in place of $(D_A \subset A)$ and with $\mathcal{F}$). Moreover, since the connecting maps preserve normalisers it is straightforward to see that so do the $\varrho_j$, i.e., $\varrho_j(\mathcal{N}_{A_j}(D_j)) \subset \mathcal{N}_{\varinjlim A_i}(\varinjlim D_i)$. Therefore also the $\varrho_j \circ \varphi^{(i)}$ preserve normalisers, thus confirming \eqref{prop:dim-without-psi-eq}. Now since $\mathcal{F} \subset (\varinjlim A_i)^1_+$ and $\e >0 $ were arbitrary, Proposition~\ref{prop:dim-without-psi} shows that 
\[
\ddim( \varinjlim D_i \subset \varinjlim A_i   ) \le \liminf_{i\in I} ( \ddim(D_i \subset A_i ) ),
\]
as desired. \hfill $\square$

\vspace{1em}
\noindent
{\it Proof} of Theorem \ref{thm:permanence}(vii): Stabilisations.
\vspace{1em}

\noindent
Taking the identity map as approximation we clearly have $\ddim (D_n \subset M_n) = 0$. From \ref{thm:permanence}(ii) it follows that $\ddim(D_A \otimes D_n \subset A \otimes M_n) \le \ddim(D_A \subset A)$. The estimate $\ddim(D_A \subset A) \le \ddim(D_A \otimes D_n \subset A \otimes M_n)$ follows from the permanence property for hereditary subalgebras \ref{thm:permanence}(iii). 

The equality $\ddim(D_A \subset A) = \ddim( D_A\otimes c_0(\N) \subset A\otimes \mathcal{K}(\ell^2(\N)) )$ follows in the same manner after noticing that $(c_0(\N) \subset \mathcal{K}(\ell^2(\N)) )$ can be written as inductive limit of sub-C$^*$-algebras $(D_n \subset M_n)$ with upper left corner embeddings, so that $\ddim(c_0(\N) \subset \mathcal{K}(\ell^2(\N)) )=0$ by \ref{thm:permanence}(vi).  \hfill $\square$

\begin{rem}
In \cite{MR3841841} Matsumoto defined the concept of {\it relative Morita equivalence for relative $\sigma$-unital sub-$\mathrm{C}^*$-algebras} (see \cite[Definition~2.1 and Definition~3.5]{MR3841841}). A relative $\sigma$-unital sub-$\mathrm{C}^*$-algebras is always nondegenerate. Examples include all unital sub-$\mathrm{C}^*$-algebras and $(c_0(\N) \subset \mathcal{K}(\ell^2(\N)))$ (see \cite[Example~2.5]{MR3841841} for more details). It follows from \ref{thm:permanence}(vii) and \cite[Theorem~1.1]{MR3841841} that the diagonal dimension of relative $\sigma$-unital sub-$\mathrm{C}^*$-algebras is invariant under relative Morita equivalences.
\end{rem}

\section{Dimension zero}
\label{section-AF}

\noindent
Recall that a $\mathrm{C}^*$-algebra has nuclear dimension zero precisely when it is locally finite dimensional  (see \cite[Remark 2.2(iii)]{WZ10} and \cite[Theorem~3.4]{Win:cpr}). In the separable case this is equivalent to being an AF algebra, i.e., an inductive limit of finite dimensional $\mathrm{C}^*$-algebras. In this section we give the respective characterisations for (nondegenerate) sub-$\mathrm{C}^*$-algebras with diagonal dimension zero. We note that the implication from (ii) to (iii) in the separable case is closely related to Bratteli's seminal work \cite[Theorem 2.2]{Bratteli}.

\begin{thm}
\label{thm:zero_ddim}
Let $(D_A \subset A)$ be a nondegenerate sub-$\mathrm{C}^*$-algebra with $D_A$ abelian. Then the following are equivalent:
	\begin{enumerate}[label=\rm{(\roman*})]
	\item $\ddim(D_A \subset A)=0$.
	\item For every finite subset $\mathcal{F}\subset A^1_+$ and $\e > 0$ there exists a finite-dimensional $\mathrm{C}^*$-algebra $F \subset A$ together with a diagonal subalgebra $D_F \subset F$ such that
	\begin{enumerate}
		\item[\rm{(1)}] $\mathrm{dist}(\mathcal{F}, F^1_+) < \e$ and 
		\item[\rm{(2)}] $\mathcal{N}_F(D_F) \subset \mathcal{N}_A(D_A)$.
	\end{enumerate}
	\end{enumerate}
	If $A$ is separable, then the conditions above are equivalent to:
	\begin{enumerate}
	\item[\rm{(iii)}] There exists an increasing sequence of finite-dimensional $\mathrm{C}^*$-subalgebras $F_n$ of $A$ together with diagonal subalgebras $D_{F_n}$ of $F_n$ such that  
	\begin{equation}
	\label{AF-lim}
	\overline{ \bigcup\nolimits_{n=1}^\infty F_n } = A, \quad \overline{ \bigcup\nolimits_{n=1}^\infty D_{F_n }} = D_A,
	\end{equation}
	and such  that each inclusion $F_n \hookrightarrow F_{n+1}$ maps $\cn_{F_n}(D_{F_n})$ into $\cn_{F_{n+1}}(D_{F_{n+1}})$.
	\end{enumerate}
\end{thm}

\begin{proof}
(i) $\Longrightarrow$ (ii): Assume first that $A$ is unital. Take $\unit_A \in \mathcal{F} \subset A^1_+$ and $0 < \e < 1$ as in (ii) and apply Proposition~\ref{prop:dim-without-psi} with $\e/2$ in place of $\e$ to obtain a $D_F \subset F$, a c.p.c.\ order zero map $\varphi:F \to A$ (since we assume the diagonal dimension to be zero there is only one colour $i$) and a set $\{b_a \mid a \in \mathcal{F}\} \subset F^1_+$ as in \ref{prop:dim-without-psi} (but in \ref{prop:dim-without-psi}(3) replacing $\e$ with $\e/2$). We then have 
\[
(1 - \e/2) \cdot \unit_A \le \varphi(b_{\unit_A}) \le \varphi(\unit_F) \le \unit_A.
\]
Therefore, the spectrum of $\varphi(\unit_F)$ does not contain zero and we may apply order zero functional calculus to $\varphi$ with the constant function $1$. This will yield a unital c.p.\ order zero map $\bar{\varphi}: F \to A$ with $\|\bar{\varphi}(x) - \varphi(x)\| \le \e/2$ for all $x \in F^1_+$. But unital order zero maps are just $^*$-homomorphisms, and we may take $\bar{\varphi}(F) \subset A$ as our finite dimensional subalgebra in \ref{thm:zero_ddim}(ii) above. Condition \ref{thm:zero_ddim}(ii)(2) will be satisfied because of \eqref{prop:dim-without-psi-eq} in connection with Lemma~\ref{lem:positive_normaliser} (the latter ensures that not only $\varphi$ but in fact $\bar{\varphi}$ maps matrix units to normalisers). 

Now suppose $A$ is not unital. Our assumption $\ddim (D_A \subset A) = 0$ implies that both $A$ and $D_A$ are locally finite dimensional, and since $(D_A \subset A)$ is nondegenerate this in particular entails that $D_A$ contains an approximate unit for $A$ which consists of projections. 

Now if $\mathcal{F} \subset A^1_+$ and $\e$ are as in \ref{thm:zero_ddim}(ii), since we want to approximate elements of $\mathcal{F}$ we may as well assume that there is a projection $p \in D_A \cap \mathcal{F}$ with $p a = a$ for all $a \in \mathcal{F}$. Set $(D_B \subset B) := (pD_Ap \subset pAp)$, then by Theorem~\ref{thm:permanence}(iii) on hereditary subalgebras we have $\ddim(D_B \subset B) = 0$.   Since we already know (i) $\Longrightarrow$ (ii) in the unital case, this means that the assertion of \ref{thm:zero_ddim}(ii) holds with $B$ in place of $A$. By Proposition~\ref{prop:maps-normalisers}(iii) we have $\mathcal{N}_B(D_B) \subset \mathcal{N}_A(D_A)$, so in fact \ref{thm:zero_ddim}(ii) holds as it stands.

(ii) $\Longrightarrow$ (i): If $A$ is unital this is immediate from Proposition~\ref{prop:dim-without-psi} (note that condition (1) of \ref{thm:zero_ddim}(ii) implies condition (3) of \ref{prop:dim-without-psi}, albeit with $2 \e$ instead of $\e$, and that automatically $\unit_F = \unit_A \in D_A$ provided $\e <1/2$.).

Now again suppose $A$ is not unital. Given a finite subset $\mathcal{F}\subset A^1_+$ and $\e > 0$, take $D_F \subset F \subset A$ as in \ref{thm:zero_ddim}(ii). Define $\bar{F}:= F \oplus \mathbb{C} \cdot (\unit_{A^\sim} - \unit_F)$ and $D_{\bar{F}}:= D_F \oplus \mathbb{C} \cdot (\unit_{A^\sim} - \unit_F)$, then $\bar{F} \subset A^\sim$ and $D_{\bar{F}} \subset \bar{F}$ is a diagonal. Note that (1) $\mathrm{dist}(a,\bar{F}) < \e$ for all $a \in \mathcal{F} \cup \{\unit_{A^\sim}\}$ and that (2) every matrix unit in $\bar{F}$ with respect to $D_{\bar{F}}$ belongs to the normaliser $\mathcal{N}_{A^\sim}(D_A^\sim)$.

But this means that \ref{thm:zero_ddim}(ii) holds for $A^\sim$ in place of $A$ (the special form $\mathcal{F} \cup \{\unit_{A^\sim}\}$ of the finite subset is no loss of generality), so by our initial observation ((ii) $\Longrightarrow$ (i) in the unital case) we have $\ddim(D^\sim_A \subset A^\sim) = 0$. Now \ref{thm:zero_ddim}(i) follows from Theorem~\ref{thm:permanence}(iv) on the diagonal dimension of unitisations.

(iii) $\Longrightarrow$ (ii): Given $\mathcal{F}$ and $\e$, since 
\[
 \mathrm{dist}\Big(\mathcal{F}, \overline{ \bigcup\nolimits_{n=1}^\infty F_n }\Big) = \mathrm{dist}(\mathcal{F}, A) = 0,
\]
there is some $\bar{n} \in \mathbb{N}$ such that $\mathrm{dist}(\mathcal{F}, (F_{\bar{n}})^1_+) < \e$. Moreover, the inclusions $\cn_{F_n}(D_{F_n}) \subset \cn_{F_{n+1}}(D_{F_{n+1}})$ together with \eqref{AF-lim} imply $\cn_{F_{\bar{n}}}(D_{F_{\bar{n}}}) \subset \cn_{A}(D_{A})$.

(ii) $\Longrightarrow$ (iii): 
This is slightly more involved, since we have to turn approximate inclusions of finite dimensional $\mathrm{C}^*$-algebras into exact ones, while at the same time respecting the diagonals. 

\underline{Step 1.} 	Consider the set of finite dimensional sub-$\mathrm{C}^*$-algebras
	\begin{align}
	\Lambda &:= 
	\{ (D_F\subset F) \mid F \subset A ,  \, D_F \subset F \mbox{ is a diagonal,  and}\, \nonumber \\
	& \phantom{{ } := { }\{} \cn_F(D_F) \subset \cn_A(D_A)\}. \label{AF-lim1}
	\end{align}
Note that every $(D_F \subset F) \in \Lambda$ has a set $S_F \subset F \setminus D_F$ of matrix units which is \emph{tight} in the following sense: Whenever $e \in F \setminus D_F$ is a matrix unit, then precisely one of $e$ or $e^*$ can be written in a unique manner as a product of a scalar and at most $\dim D_F - 1$ elements of $S_F$. (If $(D_F \subset F) \cong (D_r \subset M_r)$, then $S_{M_r}:= \{e_{i,i+1} \mid i=1, \ldots r-1\}$ is tight in this sense; if $F$ is a direct sum of matrix blocks, take $S_F$ to be a disjoint union of tight sets of matrix units for the individual summands.) 

We claim that if $(D_E \subset E) \in \Lambda$ is another finite dimensional sub-$\mathrm{C}^*$-algebra which contains a set $\{v_e \mid e \in S_F\}$ of partial isometries normalising $D_E$ and such that, for each $e \in S_F$, $\|v_e - e\| < \frac{1}{4}$, then the assignment $e \mapsto v_e$ for $e \in S_F$ extends to a $^*$-homomorphism $\alpha: F \to E$ which maps $\mathcal{N}_F(D_F)$ into $\mathcal{N}_E(D_E)$. 

To prove the claim, observe first that any two elements in the set $\{e^*e, ee^* \mid e \in S_F\} \subset D_F$ of range and orthogonal projections either agree or are orthogonal. And since $\|v_e^*v_e - e^*e\|, \, \|v_e v_e^* - e e^*\| < 1$ for each $e \in S_F$, the pairwise commuting projections in $\{v_e^*v_e, \,  v_e v_e^* \mid e \in S_F\} \subset D_E$ satisfy the same multiplicative relations as those in $\{e^*e, ee^* \mid e \in S_F\}$. From this and the unique decomposition of matrix units in $F$ as products of elements from $S_F$ it follows that products in $v_e$ and $v_e^*$ satisfy the same relations as products in $e$ and $e^*$, whence the assignment $e \mapsto v_e$ indeed yields a $^*$-homomorphism $\alpha$ as claimed. Every matrix unit in $F$ maps to a product of elements $v_e$ or $v_e^*$, hence maps to $\mathcal{N}_E(D_E)$. It now follows from Proposition~\ref{prop:orthogonal-normalisers} that indeed $\alpha(\mathcal{N}_F(D_F)) \subset \mathcal{N}_E(D_E)$.    

\underline{Step 2.} After these preparations we now turn to constructing nested sub-$\mathrm{C}^*$-algebras as in (iii).
We start by inductively constructing $(D_{E_n} \subset E_n) \in \Lambda$ and suitable maps between them; the nested $(D_{F_n} \subset F_n)$ will then arise as small perturbations.

Let $\{a_0, a_1, a_2,\ldots\} \subset A_+^1$ be a countable dense subset of the positive unit ball. We may assume $a_0 = 0$ and define 
\[
(D_{E_0} \subset E_0) := (\{0\} \subset \{0\}) \in \Lambda.
\]
Now suppose $(D_{E_n} \subset E_n) \in \Lambda$ has been defined. Take some
\[
0 < \gamma_n \le \frac{1}{2^n \cdot \dim^{+1}(D_{E_n}) \cdot \dim^{+1}(E_n)}
\]
(here $\dim$ stands for vector space dimension) and some 
\begin{equation}
\label{4-1-5}
0 < \zeta_n < \frac{\gamma_n^2}{144} \le \frac{1}{144},
\end{equation}
and choose $(D_{E_{n+1}} \subset E_{n+1}) \in \Lambda$ such that 
\begin{equation}
\label{4-1-1}
\dist( \{a_0, \ldots, a_n\} \cup (E_n)^1_+, (E_{n+1})_+^1), \, \dist((E_n)^1, \,(E_{n+1})^1) < \zeta_n.
\end{equation}
Induction yields a sequence $((D_{E_n} \subset E_n))_{n\in \mathbb{N}} \subset \Lambda$. 

\underline{Step 3.} Next, for each $n$ we construct $^*$-homomorphisms $\alpha_n: E_n \to E_{n+1}$ such that $\alpha_n(\mathcal{N}_{E_n}(D_{E_n})) \subset \mathcal{N}_{E_{n+1}}(D_{E_{n+1}})$ for each $n$.

Since $E_{n+1}$ is finite dimensional, the identity map extends to some conditional expectation $\psi: A \to E_{n+1}$. Note that $D_{E_{n+1}}$ lies in the multiplicative domain of $\psi$, and that we know from \eqref{AF-lim1} and Lemma~\ref{lem:positive_normaliser} that $D_{E_n} \cup D_{E_{n+1}} \subset D_A$. Therefore, for each $p \in D_{E_n}$ and $q \in D_{E_{n+1}}$ we have $q \psi(p) = \psi(q) \psi(p) = \psi(qp) = \psi(pq) = \psi(p) \psi(q) = \psi(p) q$. Now since $D_{E_{n+1}}$ is maximal abelian in $E_{n+1}$ this implies that 
\[
\psi(D_{E_n}) \subset D_{E_{n+1}}.
\]
With $\Phi_{E_{n+1}}: E_{n+1} \to D_{E_{n+1}}$ denoting the canonical conditional expectation we then in particular have 
\begin{equation}
\label{4-1-3}
\Phi_{E_{n+1}} \circ \psi|_{D_{E_n}} = \psi|_{D_{E_n}}.
\end{equation}

Now choose a tight set $S_{E_n}$ of matrix units of $(D_{E_n} \subset E_n)$. For each matrix unit $e$ of $S_{E_n}$, by \eqref{4-1-1} there is some $w_e \in (E_{n+1})^1$ with 
\begin{equation}
\label{4-1-2}	
\|w_e - e\| < \zeta_n
\end{equation} 
Noting that
\begin{equation}
\label{4-1-4}
\psi(w_e^*w_e) = w_e^* w_e	
\end{equation}
 we compute
\begin{align}
\|\Phi_{E_{n+1}}(w_e^*w_e) - e^*e\| & \stackrel{\phantom{\eqref{4-1-5}}}{\le} \|\Phi_{E_{n+1}}(w_e^*w_e) - \Phi_{E_{n+1}}(\psi(w_e^*w_e)) \| \nonumber \\
& \phantom{{} \stackrel{\eqref{4-1-5}}{\le} {}} + \|\Phi_{E_{n+1}}(\psi(w_e^*w_e)) - \Phi_{E_{n+1}}(\psi(e^*e))\| \nonumber \\
& \phantom{{} \stackrel{\eqref{4-1-5}}{\le} {}} + \| \Phi_{E_{n+1}}(\psi(e^*e)) - \psi(e^*e) \| \nonumber \\
& \phantom{{} \stackrel{\eqref{4-1-5}}{\le} {}} + \|\psi(e^*e) - \psi(w_e^*w_e) \| \nonumber \\
&\phantom{{} \stackrel{\eqref{4-1-5}}{\le} {}}+ \|\psi(w_e^*w_e) - e^*e\| \nonumber \\
& \stackrel{\phantom{\eqref{4-1-5}}}{<} 0 + 2 \zeta_n + 0 + 2 \zeta_n + 2 \zeta_n \nonumber \\
&\textstyle  \stackrel{\eqref{4-1-5}}{<} 1/2,  \label{AF-lim3}
\end{align} 
where for the estimates of the five summands we have used (in this order) \eqref{4-1-4}, \eqref{4-1-2}, \eqref{4-1-3}, \eqref{4-1-2}, and \eqref{4-1-4} and \eqref{4-1-2}.
This shows that the projection $\chi_{(\frac{1}{2},1]}(\Phi_{E_{n+1}}(w_e^*w_e)) \in D_{E_{n+1}}$ has norm distance less than $1$ from $e^*e$, and since both are projections in the commutative $\mathrm{C}^*$-algebra $D_A$, they have to agree. The respective statements hold for $ee^*$, and in particular this entails that $D_{E_n} \subset D_{E_{n+1}}$. 

At this point we could deform $w_e$ into a partial isometry using functional calculus. However, without further ado this will not necessarily normalise $D_{E_{n+1}}$, as is required to apply Step 1 above. Fortunately we can arrange for a more refined deformation. To this end, write $ee^* = q_1 + \ldots + q_m$ as a sum of pairwise orthogonal rank one projections $q_i \in D_{E_{n+1}}$ and for each $i$ set $p_i := e^* q_i e$. We have $p_i \in D_A$ since $e$ normalises $D_A$, and similarly to \eqref{AF-lim3} we estimate 
\[
\|\Phi_{E_{n+1}}(w_e^* q_i w_e) - p_i\| = \|\Phi_{E_{n+1}}(w_e^* q_i w_e) - e^*q_i e\| < 6 \zeta_n < 1/2,
\]
whence $\Phi_{E_{n+1}}(w_e^* q_i w_e)$ is close to a projection in $D_{E_{n+1}}$, which then has to agree with $p_i$. This shows $p_i \in D_{E_{n+1}}$. From 
\begin{equation}
\|p_i - p_i w_e^* q_i w_e p_i\| = \|p_i e^* q_i e p_i - p_i w_e^* q_i w_e p_i\|\stackrel{\eqref{4-1-2}}{<} 2 \zeta_n
\end{equation} we see that $p_i$ is a rank one projection, and that there is 
\begin{equation}
\label{4-1-7}
1- 2 \zeta_n \le \nu \le 1
\end{equation} 
such that 
\begin{equation}
\label{4-1-6}
v_i := \nu^{-\frac{1}{2}} \cdot q_i w_e p_i \in E_{n+1}
\end{equation} 
is a matrix unit with $v_i^* v_i = p_i$ and $v_i v_i^* = q_i$, and with 
\begin{align}
\label{4-1-8}
\|v_i - q_i e p_i\| & \le (\nu^{-\frac{1}{2}} -1)  +\zeta_n \le ((1- 2 \zeta_n)^{-\frac{1}{2}} -1) +  \zeta_n \le 3 \zeta_n^{\frac{1}{2}},
\end{align}
where the first estimate uses \eqref{4-1-6} and \eqref{4-1-2}, the second uses \eqref{4-1-7}, and the third estimate holds by a small calculation and the fact that $\zeta_n < \frac{1}{2}$ (see \eqref{4-1-5}). 

For $i \neq j$ we have $q_i e p_j = q_i e e^* q_j e = q_i q_j e e^* e = 0$ and $q_i v_i p_j = v_i p_i p_j = v_i e^* q_i e p_j = 0$. Therefore, for
\[
v_e:= \sum\nolimits_i v_i
\]
we have 
\begin{align*}
\|v_e - e\| & \stackrel{\phantom{\eqref{4-1-8}}}{=} \Big\| \Big( \sum\nolimits_i q_i \Big) (v_e - e) \Big(\sum\nolimits_i p_i \Big) \Big\| \\ 
& \stackrel{\phantom{\eqref{4-1-8}}}{=} \max_i \|v_i - q_i e p_i\| \\
& \stackrel{\eqref{4-1-8}}{\le} 3 \zeta_n^{\frac{1}{2}} \\
& \stackrel{\, \eqref{4-1-5} \,}{<} 1/4.
\end{align*}
Finally, since $v_i \in \mathcal{N}_{E_{n+1}}(D_{E_{n+1}})$, it follows from Proposition~\ref{prop:orthogonal-normalisers} that also $v_e \in \mathcal{N}_{E_{n+1}}(D_{E_{n+1}})$. 

We are now in position to apply the claim in Step 1 above (with $E_n$ in place of $F$ and $E_{n+1}$ in place of $E$), which makes sure that the assignment $e \mapsto v_e$ extends to a $^*$-homomorphism $\alpha_n: E_n \to E_{n+1}$ such that $\alpha_n(\mathcal{N}_{E_n}(D_{E_n})) \subset \mathcal{N}_{E_{n+1}}(D_{E_{n+1}})$. Note that for each $x \in E_n^1$ we have 
\begin{equation}
\label{AF-lim2}
\textstyle
\|\alpha_n(x) - x\| \le \dim(D_{E_n}) \cdot \dim(E_n) \cdot 3 \zeta_n^\frac{1}{2} \le 1/2^n
\end{equation}
and that $\alpha_n|_{D_{E_n}} = \id_{D_{E_n}}$.

\underline{Step 4.} It only remains to turn the sequence $((D_{E_n} \subset E_n))_{n \in \mathbb{N}}$ into a nested one. 

For each $n \in \mathbb{N}$ and $x \in E_n$ observe that $(\alpha_m \ldots \alpha_n(x))_{m \ge n}$ is a Cauchy sequence in $A$ by \eqref{AF-lim2}, so we may define a map $\beta_n: E_n \to A$ by setting $\beta_n(x) := \lim_{m} \alpha_m \ldots \alpha_n(x)$. It is then clear that $\beta_n$ is a $^*$-homomorphism, that $\beta_n|_{D_{E_n}} = \id_{D_{E_n}}$, and that the $\beta_m(E_m)$ are increasing with $\overline{\bigcup_m \beta_m(E_m)} = A$. Moreover, for any $x \in \mathcal{N}_{E_n}(D_{E_n})$ the image $\beta_n(x)$ normalises $\beta_{n+1} (D_{E_{n+1}})$ in $\beta_{n+1} ({E_{n+1}})$ , i.e., $\beta_n(\mathcal{N}_{E_n}(D_{E_n})) \subset \mathcal{N}_{\beta_{n+1}(E_{n+1})}(\beta_{n+1}(D_{E_{n+1}}))$. We may therefore set $F_n:= \beta_n(E_n)$ and $D_{F_n} := \beta_n(D_{E_n}) = D_{E_n}$, and are done.
\end{proof}

\begin{rem}\label{AFdiag}
	It follows from Theorem~\ref{thm:zero_ddim} in connection with \cite[Theorem~5.7]{Power:Pitman} that if $A$ is a separable AF algebra, then $A$ admits a diagonal $D_A$ (called a \emph{regular canonical masa} in \cite{Power:Pitman}, now often referred to as \emph{standard AF diagonal}) which satisfies $\ddim(D_A\subset A) = 0$ and is unique (up to approximately inner automorphisms of $A$) with this property. 
	
	It was shown in \cite{MS22} that there is a simple separable unital monotracial AF algebra $A$ containing a Cartan subalgebra $D$ which (as an abelian $\mathrm{C}^*$-algebra) is isomorphic to a regular canonical masa $D_A$, but which does not have the unique extension property, hence is not a diagonal, hence the pairs $(D \subset A)$ and $(D_A \subset A)$ are not isomorphic.  
\end{rem}

\section{Topological dynamical systems}
\label{sec5}

\noindent
In this section we link diagonal dimension to Kerr's notion of (fine) tower dimension for topological dynamical systems.

\begin{defn} \cite[Definition 4.1]{Ker17}
	Let $G\curvearrowright X$ be an action of a countable discrete group on a compact Hausdorff space $X$. 
	
	 A \emph{tower} is a pair $(V,S)$ consisting of a subset $V$ of $X$ and a finite subset $S$ of $G$ such that the sets $sV$, called the \emph{levels}, for $s\in S$ are pairwise disjoint. $V$ is called the \emph{base} and $S$ is called the \emph{shape} of the tower. We write $SV := \{sv \mid s \in S,\, v \in V\}$ for the union of the levels.
		
		 A \emph{castle} is a family $((V_j, S_j))_{j\in J}$ of towers such that all the levels are pairwise disjoint.
	
	A tower $(V,S)$ is \emph{open} if $V$ is an open subset of $X$. A castle whose towers are open will be referred to as an \emph{open castle}.	
\end{defn}

\begin{defn} \cite[Definition 4.3, Definition 4.10]{Ker17}
\label{def:tower-dimension}
	Let $G\curvearrowright X$ be an action of a countable discrete group on a compact Hausdorff space $X$. The \emph{tower dimension} of the action $G\curvearrowright X$ is said to be at most $d$, written $\tdim(X,G) \leq d$, if for every finite subset $E\subset G$ there exist a finite family of open towers $((V_j,S_j))_{j\in J}$ and a partition $J = J^{(0)}\sqcup \ldots \sqcup J^{(d)}$ such that 
	\begin{enumerate}
		\item $\bigcup_{j\in J}S_jV_j = X$,
		\item for every $x\in X$ there are $j\in J$ and $t\in S_j$ such that $x\in tV_j$ and $Et\subset S_j$ (this is called $E$-Lebesgueness in \cite{Ker17}),
		\item for each $i\in \{0,\ldots,d\}$ the sets $S_jV_j$ are pairwise disjoint for $j\in J^{(i)}$ (in other words, for each $i\in \{0,\ldots,d\}$ the family $((V_j,S_j))_{j\in J^{(i)}}$ forms a castle).
	\end{enumerate}
	If, in addition, for any given open cover $\mathcal{U}$ of $X$ one can choose the family of open towers $((V_j,S_j))_{j\in J}$ so that the open cover $(sV_j)_{j\in J, s\in S_j}$ refines $\mathcal{U}$, then the action is said to have \emph{fine tower dimension} at most $d$, and in this case we write $\ftdim(X,G) \leq d$. If no such $d$ exists, then the tower dimension (or the fine tower dimension, respectively) is said to be infinite.
\end{defn}

\begin{rems} \label{rems:towerdim}
(i) The definitions above insist on $G$ being countable discrete, since non-discrete topological groups will require a different setup; cf.\ \cite{HSWW17} vs.\ \cite{HWZ15}.

(ii) We have stated the definition of fine tower dimension in a version slightly different from \cite[Definition~4.10]{Ker17} in order to avoid asking $X$ to carry a metric (but in the metrisable case the two formulations are equivalent). 

This additional generality will not cause technical difficulties, and will prove useful when we look at uniform Roe algebras in Section~\ref{sec:examples}. The only \emph{caveat} is that for non-metrisable spaces the various notions of covering dimension are no longer equivalent, so that one has to make a choice. The appropriate version for our purposes is the one characterised in \cite[Proposition~1.5]{KW04}: every finite open cover of $X$ has a finite open refinement with chromatic number $d+1$. 

(iii) Unlike \cite[Definition~4.10]{Ker17} we explicitly ask the index set $J$ to be finite to begin with, whereas in \cite{Ker17} this is derived from compactness of $X$ whenever necessary. 

(iv) Finiteness of tower dimension requires the action to be \emph{free}. Indeed, let $g\in G\setminus \{e\}$, let $x\in X$, and let $((V_j,S_j))_{j\in J}$ be a family witnessing the tower dimension for the finite subset $E := \{ g, g^{-1}, e\}$. By the $E$-Lebesgueness condition \ref{def:tower-dimension}(2) there are $j\in J$ and $t\in S_j$ such that $x\in tV_j$ and $Et\subset S_j$. Since $gtV_j$ and $tV_j$ are disjoint, $x$ cannot be fixed by $g$.

It will therefore in particular follow from the theorem below that finite diagonal dimension of $(C(X) \subset C(X)\rtimes_{\mathrm{r}} G)$ can occur only for free actions. In upcoming work, the third named author will introduce a variant of diagonal dimension which will, at least to some extent, also allow isotropy.
\end{rems}

From \cite[Proposition~4.11]{Ker17} and the definition of fine tower dimension we know that 
\begin{align}
&\max\{  \tdim^{+1}(X,G), \, \dim^{+1}(X)\} \nonumber \\
&\qquad \qquad \le \ftdim^{+1}(X,G) \nonumber \\
&\qquad \qquad \qquad \qquad \le   \tdim^{+1}(X,G)\cdot \dim^{+1}(X). \label{tdim-ftdim}
\end{align}
The main result of this section fits diagonal dimension into this chain of estimates.

\begin{thm} \label{thm:diag_vs_tower}
	Let $\alpha:G\curvearrowright X$ be an action of a countable discrete group on a compact Hausdorff space. Then 
	\begin{align}
	&\ftdim^{+1}(X,G) \nonumber \\
	&\qquad \qquad \leq  \ddim^{+1}( C(X)\subset C(X)\rtimes_{\mathrm{r}} G ) \label{ddim-lower}\\
	&\qquad \qquad \qquad \qquad \leq   \tdim^{+1}(X,G)\cdot \dim^{+1}(X). \label{ddim-upper}
	\end{align}
	In particular, if $X$ is totally disconnected\footnote{Recall that a compact Hausdorff space is totally disconnected if and only if it has covering dimension zero.} then
	\begin{equation}
	\tdim(X,G) = \ftdim(X,G) = \ddim( C(X) \subset C(X)\rtimes_{\mathrm{r}} G). \label{ddim-zero}
	\end{equation}
\end{thm}

The proof is divided into two parts. To implement the upper bound \eqref{ddim-upper}, we only need to carefully follow Kerr's proof of \cite[Theorem 6.2]{Ker17}, which establishes the respective bound for the nuclear dimension of the crossed product, but in fact keeps track of the diagonal by design. 

The key new feature in this section is the lower bound \eqref{ddim-lower}, which extracts purely dynamical information from $\mathrm{C}^*$-algebraic data. 

The first equality in \eqref{ddim-zero} is already contained in \eqref{tdim-ftdim}, and the second one follows upon combining the latter with \eqref{ddim-lower} and \eqref{ddim-upper}.

 \bigskip

We start by isolating the lemma below from the proof of \cite[Theorem 6.2]{Ker17}.  
Roughly speaking, it expresses tower dimension in terms of partitions of unity (as opposed to open covers) which are approximately compatible with the group action. We spell out the proof partly for the convenience of the reader and partly because \cite[Theorem 6.2]{Ker17} is stated and proven under the slightly more special assumption that $X$ is metrisable.

\begin{lem} \label{lem:inv_pou}
	Let $\alpha:G\curvearrowright X$ be as before. Suppose $\tdim(X,G) \leq d < \infty$. Then for every finite subset $E\subset G$ containing the identity $e \in G$ and for every $\eta > 0$, there exist a finite family of open towers $((V_j,S_j))_{j\in J}$, a partition $J = J^{(0)}\sqcup \ldots \sqcup J^{(d)}$, and a partition of unity $(h_j)_{j\in J}$ for $X$ such that
	\begin{enumerate}
			\item[\rm{(1)}] $(S_jV_j)_{j\in J}$ forms an open cover for $X$ to which the partition of unity $(h_j)_{j\in J}$ is subordinate,
		\item[\rm{(2)}] for each $i \in \{0,\ldots,d\}$, $j \neq j' \in J^{(i)}$, and for $g,g' \in E$, $\alpha_g(h_j) \alpha_{g'}(h_{j'}) = 0$ (we abuse notation and write $\alpha$ also for the induced action on $C(X)$ here),
		\item[\rm{(3)}] for each $i\in \{0,\ldots,d\}$ the family $((V_j,S_j))_{j\in J^{(i)}}$ forms a castle, 
		\item[\rm{(4)}] $| h_j(x) - h_j(g x)| < \eta$ for all $j \in J$, $x\in X$ and $g\in E$.
	\end{enumerate}
\end{lem}

\begin{proof}[Proof (cf.\ that of Theorem~6.2 in \cite{Ker17})]
We may assume $E$ to be symmetric by replacing it  with $E \cup E^{-1}$, if necessary. Take $n \in \mathbb{N}$ such that $(d+2) / n < \eta$.

Choose a finite family of open towers $((V_j,S_j))_{j\in J}$ and a partition $J = J^{(0)}\sqcup \ldots \sqcup J^{(d)}$ satisfying conditions \ref{def:tower-dimension}(1) and \ref{def:tower-dimension}(3) (so that in particular \ref{lem:inv_pou}(3) holds), and such that \ref{def:tower-dimension}(2) holds with $E$ replaced by $E^n$. For each $j \in J$ define the $E^n$-core of $S_j$ as 
\[
B_{j,n} := \{ s \in S_j \mid E^n s \subset S_j\}.
\]
Upon shifting the $S_j$ from the right and the $V_j$ from the left, if necessary, we may assume each of the nonempty $B_{j,n}$ to contain $e$.

By \ref{def:tower-dimension}(2), the family $(sV_j)_{j \in J, s \in B_{j,n}}$ is an open cover of $X$, and since $X$ is normal we can find a partition of unity $(\hat{f}_{j,s})_{j \in J, s \in B_{j,n}}$ subordinate to it. For every $j \in J$ define $f_j \in C(X)^1_+$ by 
\begin{equation}
\label{f-pre-pou}
f_j := \max_{s \in B_{j,n}} \alpha_{s^{-1}}(\hat{f}_{j,s})
\end{equation}
provided $B_{j,n} \neq \emptyset$; otherwise set $f_j := 0$. 
Note that for every $j \in J$ and $s \in B_{j,n}$ we have $0 \le \hat{f}_{j,s} \le \alpha_s(f_j) \in C_0(sV_j)$, whence 
\[
\sum\nolimits_{j \in J} \sum\nolimits_{s \in B_{j,n}} \alpha_s(f_j) \ge \unit_X.
\]
Next, for each $j$ define further subsets $B_{j,k}$ of $S_j$ by 
\begin{align}
B_{j,0} &  := S_j \setminus \bigcap\nolimits_{g \in E} g S_j  \mbox{ and } \label{Bjk-def}\\
B_{j,k} &  :=  \Big( \bigcap\nolimits_{g \in E^k} g S_j \Big) \setminus \Big( \bigcap\nolimits_{g \in E^{k+1}} g S_j \Big) \mbox{ for } k=1,\ldots,n-1 \nonumber 
\end{align}
and note that 
\[
B_{j,n} = \bigcap\nolimits_{g \in E^n} g S_j.
\]
This in particular means that for each $j$ the family $(B_{j,k})_{k=0,\ldots,n}$ forms a partition of $S_j$.
Moreover, by construction we have for all $j \in J$ and $g \in E$ 
\begin{align}
g B_{j,k} &  \subset B_{j,k-1} \cup B_{j,k} \cup B_{j,k+1} \mbox{ for } k=1,\ldots,n-1,  \mbox{ and } \nonumber \\
g B_{j,n} &  \subset B_{j,n-1} \cup B_{j,n}. \label{Bjn-shift}
\end{align}
For each $j \in J$ we may now define the function
\[
\hat{h}_{j} := \sum\nolimits_{k=0}^n \sum\nolimits_{s \in B_{j,k}} \frac{k}{n} \cdot \alpha_s(f_j) \in C(X).
\]
Note that for $x \in X$ and $g \in E$ by \eqref{Bjn-shift} we have
\[
| \hat{h}_{j}(g^{-1}x) - \hat{h}_{j}(x) | \le {1}/{n}.
\]
Now set $H := \sum_{j \in J} \hat{h}_{j}$ and observe that for any $x \in X$
\begin{align*}
H(x) &  = \sum\nolimits_{j \in J}\sum\nolimits_{k=0}^n \sum\nolimits_{s \in B_{j,k}} \frac{k}{n} \cdot \alpha_s(f_j)(x) \\
&   \ge \sum\nolimits_{j \in J} \sum\nolimits_{s \in B_{j,n}} \alpha_s(f_j)(x) \ge 1.
\end{align*}
We may therefore define
\[
\textstyle
h_j := H^{-1} \cdot \hat{h}_{j} 
\]
to obtain a partition of unity $(h_j)_{j \in J}$ for $X$. Note that for each $j$ the open support of $h_j$ (which is the same as that of $\hat{h}_{j}$) is contained in $(\bigcap_{g \in E} g S_j)V_j \subset S_jV_j$ (cf.\ \eqref{Bjk-def}), so \ref{lem:inv_pou}(1) holds. Since $\bigcap_{g \in E} g S_j$ agrees with the $E$-core of $S_j$, this implies that for each $g \in E$ the function $\alpha_g(h_{j})$ is supported in $S_j V_j$; in connection with \ref{lem:inv_pou}(3) this yields \ref{lem:inv_pou}(2).

From here on, one checks that
\[
|\alpha_{g^{-1}}(H)(x) - H(x)| \le ({d+1})/{n}
\]
for $x \in X$ and $g \in E$ exactly as in the paragraph preceding the inequality numbered (6) in the proof of \cite[Theorem 6.2]{Ker17}. That same inequality also yields 
\[
|\alpha_{g^{-1}}(h_j)(x) - h_j(x)| \le ({d+2})/{n} < \eta
\]
for $x \in X$ and $g \in E$, so \ref{lem:inv_pou}(4) holds and our proof is complete.
\end{proof}

\begin{proof}[Proof of \eqref{ddim-upper}]
	Write $d := \tdim(X,G)$ and $c := \dim(X)$. We may assume  that both numbers are finite, so that in particular $\alpha$ is a  free action by Remark~\ref{rems:towerdim}(iv). Let $\mathcal{F}$ be a finite subset of $C(X)\rtimes_{\mathrm{r}} G$ and let $\e >0$ be given. We are looking for c.p.c.\ approximations for $\mathcal{F}$ within $\e$, and so by linearity and continuity we may assume that  our finite subset is of the form 
	\[
	\mathcal{F} = \{ f u_s \mid f \in \mathcal{E},\, s \in E \},
	\]
	for finite subsets $\mathcal{E} \subset C(X)_+^1$ and $E \subset G$. (We use the common notation $u_g$, $g \in G$, to denote the left regular unitaries in $C(X)\rtimes_{\mathrm{r}} G$ implementing the action by conjugation; with this, $\mathrm{span} \{ f u_g \mid f \in C(X), \, g \in G\}$ is dense in in $C(X)\rtimes_{\mathrm{r}} G$.) We may moreover assume that $E$ is symmetric and contains the identity $e$ of $G$, and that $\unit_{C(X)} \in \mathcal{E}$.

	Choose some 
	\begin{equation}
	\label{5-3-2}
	0< \eta < \frac{\e}{2(d+1)}
	\end{equation} 
	so that, whenever $h$ and $a$ are positive contractions in a $\mathrm{C}^*$-algebra such that $\|[a,h]\| \le \eta$, then $\|a, h^{\frac{1}{2}}]\| < \frac{\e}{2(d+1)}$ (this is possible since the function $(t \mapsto t^{\frac{1}{2}})$ may be approximated uniformly on the interval $[0,1]$ by polynomials in $t$). Find a family of open towers $((V_j,S_j))_{j\in J}$ with a partition $J = J^{(0)}\sqcup \ldots \sqcup J^{(d)}$ and a partition of unity $(h_j)_{j\in J}$ as in Lemma \ref{lem:inv_pou} for $E$ and $\eta$. Note that it follows from \ref{lem:inv_pou}(4) that for each $j \in J$ and $s \in E$
	\[
	\| u_s h_j - h_j u_s \| \le \eta. 
	\]
	Setting 
	\[
	\textstyle
	h^{(i)} := \sum_{j\in J^{(i)}} h_j
	\]
	for $i\in \{0,\ldots,d\}$, it moreover follows from \ref{lem:inv_pou}(1), (2) and (3) that 
	\begin{equation}
	\label{5-3-1}
	\| u_s h^{(i)} - h^{(i)} u_s \| \le \eta. 
	\end{equation} 
	
	For each $j\in J$ define
	\[
	A_j := \mathrm{C}^*( u_sC_0(V_j)u_t^* \mid s,t\in S_j  ).
	\]
	Note that $C_0(S_jV_j)$ embeds into $A_j$ and the map $M_{|S_j|}\otimes C_0(V_j) \to A_j$ defined by $e_{s,t}\otimes f \mapsto u_sfu_t^*$ is an isomorphism mapping the abelian subalgebra $D_{|S_j|}\otimes C_0(V_j)$ onto $C_0(S_jV_j)$. Write 
	\[
	A^{(i)} := \bigoplus\nolimits_{j\in J^{(i)}} A_j,
	\]
	which is identified with the (finite) direct sum $\bigoplus_{j\in J^{(i)}} M_{|S_j|}\otimes C_0(V_j) $. Now by Theorem~\ref{thm:permanence}(i) and (vii) we have the estimate
	\begin{align*}
	& \ddim\Big(  C_0 \Big( \bigsqcup\nolimits_{j\in J^{(i)}} S_jV_j \Big)  \subset A^{(i)}  \Big)\\
	& =	\ddim\Big( \bigoplus\nolimits_{j\in J^{(i)}} D_{|S_j|}\otimes C_0(V_j) \subset \bigoplus\nolimits_{j\in J^{(i)}} M_{|S_j|}\otimes C_0(V_j)  \Big) \\
	&\textstyle = \max_{j\in J^{(i)}} \ddim(C_0(V_j) \subset C_0(V_j)) \\
	& \leq \dim(X) \\
	& = c.
	\end{align*}
	 Next observe that for every $fu_s \in \mathcal{F}$, for each $i$ the compression $(h^{(i)})^{\frac{1}{2}}fu_s (h^{(i)})^{\frac{1}{2}}$ belongs to $A^{(i)}$. Let $(F^{(i)}, D_{F^{(i)}}, \theta^{(i)}, \varphi^{(i)})$ be a completely positive approximation witnessing $\ddim\big(   C_0( \bigsqcup_{j\in J^{(i)}} S_jV_j ) \subset A^{(i)}  \big)\leq c$ for the finite set $\{(h^{(i)})^{1/2}fu_s (h^{(i)})^{1/2} \mid  f \in \mathcal{E},\, s \in E \}$ within $\eta$. By Arveson's extension theorem we can extend the map $\theta^{(i)}:A^{(i)}\to F^{(i)}$ to a c.p.c.\ map $\tilde{\theta}^{(i)} : C(X)\rtimes_{\mathrm{r}} G\to F^{(i)}$.	Now define c.p.c.\ maps
	\[
	\psi^{(i)} :C(X)\rtimes_{\mathrm{r}} G\to F^{(i)}
	\]
	by
	\[
	\psi^{(i)} := \tilde{\theta}^{(i)}( (h^{(i)})^{\frac{1}{2}} \, . \, (h^{(i)})^{\frac{1}{2}}).
	\]
	Viewing the $A^{(i)}$ as subalgebras of $C(X)\rtimes_{\mathrm{r}} G$, we arrive at the diagram
	\begin{center}
		\begin{tikzpicture} [node distance = 2.5cm, auto]
		\node (1) {$C(X)\rtimes_{\mathrm{r}} G$};
		\node (2) [right of =1, below of =1] {$F := F^{(0)}\oplus \ldots \oplus F^{(d)}$};
		\node (3) [right of =2, above of =2] { $C(X)\rtimes_{\mathrm{r}} G$};
		\draw [->][swap] (1) to node {${\psi}:= \bigoplus_{i=0}^d \psi^{(i)}$}  (2);
		\draw [->][swap] (2) to node {$\varphi:= \sum_{i=0}^d \varphi^{(i)}$} (3);
		\end{tikzpicture}
	\end{center}
	and may compute for $f u_s \in \mathcal{F}$
	\begin{align*}
		\|\varphi{\psi}(fu_s) - fu_s\|&= \Big\| \varphi\Big( \bigoplus\nolimits_{i=0}^d \tilde{\theta}^{(i)}( (h^{(i)})^{\frac{1}{2}}fu_s(h^{(i)})^{\frac{1}{2}} )  \Big) - fu_s \Big\| \\
		& = \Big\| \varphi\Big( \bigoplus\nolimits_{i=0}^d \theta^{(i)}( (h^{(i)})^{\frac{1}{2}}fu_s(h^{(i)})^{\frac{1}{2}} )   \Big) - fu_s \Big\| \\
		& = \Big\| \sum\nolimits_{i=0}^d \varphi^{(i)} \theta^{(i)}( (h^{(i)})^{\frac{1}{2}}fu_s(h^{(i)})^{\frac{1}{2}} ) - fu_s \Big\| \\
		& \le \Big\| \sum\nolimits_{i=0}^d (h^{(i)})^{\frac{1}{2}}fu_s (h^{(i)})^{\frac{1}{2}} - fu_s \Big\| + (d+1) \eta \\
		& \le \Big\| \sum\nolimits_{i=0}^d h^{(i)} fu_s - fu_s \Big\| + \e/2 + (d+1) \eta\\
		& < \e,
	\end{align*}
	where for the second inequality we have used \eqref{5-3-1} so that $\|u_s (h^{(i)})^\frac{1}{2} - (h^{(i)})^\frac{1}{2} u_s\| < \frac{\varepsilon}{2 (d+1)}$, and for the last inequality we have used \eqref{5-3-2} together with the fact that $(h^{(i)})_i$ forms a partition of unity.
	As each $\varphi^{(i)}$ is a sum of at most $(c+1)$ c.p.c.\ order zero maps, $\varphi$ is a sum of at most $(d+1)(c+1)$ c.p.c.\ order zero maps. Moreover, each $\varphi^{(i)}$ maps normalisers of $D_F^{(i)}$ in $F^{(i)}$ to normalisers of $C_0( \bigsqcup_{j\in J^{(i)}} S_jV_j )$ in $A^{(i)}$. But since $C_0( \bigsqcup_{j\in J^{(i)}} S_jV_j )$ is a hereditary subalgebra of $C(X)$, by Proposition~\ref{prop:maps-normalisers}(iii) such normalisers also normalise $C(X)$ in $C(X) \rtimes_{\mathrm{r}} G$. It then follows that in fact each $\varphi^{(i)}$ maps normalisers of  $D_F^{(i)}$ in $F^{(i)}$ to normalisers of $C(X)$ in $C(X) \rtimes_\mathrm{r} G$.
	
For each $f u_s \in \mathcal{F}$ take $b_{f u_s} := \psi(f u_s) \in F^1$ and note that $\psi(\mathbf{1}_{C(X)})$ belongs to $D_F$. The assertion \eqref{ddim-upper} now follows from Proposition~\ref{prop:dim-without-psi}.
\end{proof}

We now turn to the lower bound \eqref{ddim-lower}. For each $\delta \in (0,1)$ define piecewise linear continuous functions $f_\delta, g_\delta :[0,1]\to \R$ by
	\begin{equation} \label{eq:f_delta}
f_\delta(t) = \begin{cases} 0 & 0\leq t\leq \delta \\
\text{linear} &  \delta < t \leq  2 \delta \\
t & 2 \delta < t \le 1
\end{cases}
\end{equation}
and
	\begin{equation} \label{eq:g_delta}
g_\delta(t) = \begin{cases} 0 & 0\leq t\leq \delta/2   \\
\text{linear} & \delta/2 < t \leq \delta \\
1 & \delta < t \leq  1 .
\end{cases}
\end{equation}
For later use we note that the equation
\begin{equation}
\label{eq:h_delta}
t \cdot h_\delta(t) = g_\delta(t)
\end{equation}
uniquely determines a continuous function $h_\delta :[0,1]\to \R$.

\begin{lem} \label{lem:local_iso}
Let $(D_A \subset A)$ be a sub-$\mathrm{C}^*$-algebra with $D_A$ abelian, and let $\varphi:M_r\to A$ be a c.p.c.\ order zero map. Given $k,l\in \{1,\ldots,r\}$ and $\delta\in (0,1)$, define a c.p.c.\ map $\sigma_{kl}:A\to A$ by 
\begin{equation}
\label{5-6-1}
\sigma_{kl}(a) = g_\delta(\varphi)(e_{l k}) a g_\delta(\varphi)(e_{l k})^*.
\end{equation} 
Then $\sigma_{kl}$ restricts to a $^*$-isomomorphism 
\[
\overline{ f_\delta(\varphi)(e_{kk}) A f_\delta(\varphi)(e_{kk}) } \stackrel{\cong}{\longrightarrow} \overline{ f_\delta(\varphi)(e_{ll}) A f_\delta(\varphi)(e_{l l}) }.
\] 
If $\varphi(e_{l k})$ belongs to $\mathcal{N}_A(D_A)$, then $\sigma_{kl}$ restricts to a $^*$-isomorphism 
\[
\overline{ f_\delta(\varphi)(e_{kk}) D_A f_\delta(\varphi)(e_{kk}) } \stackrel{\cong}{\longrightarrow} \overline{ f_\delta(\varphi)(e_{l l}) D_A f_\delta(\varphi)(e_{l l}) }.
\]
\end{lem}

\begin{proof} 
From order zero functional calculus we have 
\[
g_{\delta}(\varphi)(e_{kl})f_\delta(\varphi)(e_{lm}) = f_\delta(\varphi)(e_{km})
\]
for all $k,l,m \in \{1,\ldots,n\}$. From this it follows that $\sigma_{kl}$ is multiplicative on the $^*$-subalgebra $f_\delta(\varphi)(e_{kk}) A f_\delta(\varphi)(e_{kk})$ and that $\sigma_{lk} \circ \sigma_{kl}$ restricts to the identity on that subalgebra. Moreover, 
\[
\sigma_{kl}\big(\overline{f_\delta(\varphi)(e_{kk}) A f_\delta(\varphi)(e_{kk})}\big) \subset \overline{f_\delta(\varphi)(e_{ll}) A f_\delta(\varphi)(e_{ll})},
\]
and so by continuity $\sigma_{kl}$ indeed restricts to an isomorphism with inverse $\sigma_{lk}$ between $\overline{ f_\delta(\varphi)(e_{kk}) A f_\delta(\varphi)(e_{kk}) }$ and  $\overline{ f_\delta(\varphi)(e_{l l}) A f_\delta(\varphi)(e_{l l}) }$.

For the second part, note that by order zero functional calculus and \eqref{eq:h_delta} 
\begin{equation}
\label{5-6-1-1}	
g_\delta(\varphi)(e_{lk}) = \varphi(e_{lk}) h_\delta(\varphi)(e_{kk}),
\end{equation}
 and that 
\[
h_\delta(\varphi)(e_{kk}) = h_\delta\big((\varphi(e_{lk})^*\varphi(e_{lk}))^\frac{1}{2}\big) \in \mathcal{N}_A(D_A) 
\] 
by Lemma~\ref{lem:positive_normaliser}. But then $g_\delta(\varphi)(e_{lk}) \in \mathcal{N}_A(D_A)$, being a product of normalisers by \eqref{5-6-1-1}, and so $\sigma_{kl}$ maps the diagonal into the diagonal.
\end{proof}

We will apply the lemma to the sub-$\mathrm{C}^*$-algebra $ (C(X)\subset C(X)\rtimes_{\mathrm{r}} G)$, and a c.p.c.\ order zero map $\varphi: M_r \to C(X)\rtimes_{\mathrm{r}} G$ with $\varphi(e_{lk})$ normalising $C(X)$ for all $k,l \in \{1,\ldots,r\}$ in order to obtain $^*$-isomorphisms
\[
\sigma_{1k}: \overline{ f_\delta(\varphi)(e_{11}) C(X) f_\delta(\varphi)(e_{11}) } \stackrel{\cong}{\longrightarrow} \overline{ f_\delta(\varphi)(e_{kk}) C(X) f_\delta(\varphi)(e_{kk}) }.
\]

For any positive function $f$ on $[0,1]$ note that $f(\varphi)(e_{kk}) \in C(X)$ since $f(\varphi)^\frac{1}{2}(e_{kk})$ normalises   $C(X)$ and the latter is unital. For any constant $\eta > 0$ we write
\[
\mathrm{supp}^{\circ}_{\eta}(f) := \{ x\in X: f(x) > \eta  \}
\]
for the open $\eta$-support of $f$, in other words, the open support of the function $(f-\eta)_+$. If, for some $\delta \in (0,1)$, we set $U_k := \mathrm{supp}^{\circ}_\eta( f_\delta(\varphi)(e_{kk} ))$, then the map $\sigma_{1k}|_{C_0(U_1)}: C_0(U_1) \cong C_0(U_k)$ induces a homeomorphism
\[
\bar{\sigma}_{1k}:U_k\to U_1
\]
with inverse $\bar{\sigma}_{k1}: U_1 \to U_k$.

\begin{lem} \label{lem:homeo_by_gp_element}
	Let $G\curvearrowright X$ be as before, $\varphi: M_r \to C(X) \rtimes_{\mathrm{r}} G$ c.p.c.\ order zero, $\delta, \eta \in(0,1)$ some numbers, and let $U_1, \ldots, U_r$ be defined as above. Then there exists a finite subset $E$ of $G$ such that for every $x\in U_1$, there is an element $g_x\in E$ such that $\alpha_{g_x}(x) = \bar{\sigma}_{k1}(x)$, where $k\in \{1,\ldots,r\}$.
	
	Moreover, if the action is free, then for every $x \in U_1$ there is precisely one such group element $g_x$ and the assignment $x\mapsto g_x$ is a well-defined continuous map from $U_1$ into $E$.
\end{lem}

\begin{proof}
	Approximate $g_\delta(\varphi)(e_{k1})$ by a finite sum 
	\begin{equation}
	\label{5-7-2}
	\sum\nolimits_{m=1}^M f_m u_{g_m}\in C_{\mathrm{c}}(G,C(X))
	\end{equation}
	within $\frac{1}{2}$, and let $E := \{g_1,\ldots,g_M\}\subset G$. Assume, for the sake of contradiction, that there exists an $x\in U_1$ such that $\alpha_g(x) \neq \bar{\sigma}_{k1}(x)$ for every $g\in E$. Then there is an open neighbourhood $W$ of $x$ such that $\alpha_{g_m}(W) \cap \bar{\sigma}_{k1}(W) = \emptyset$ for all $m\in \{1,\ldots,M\}$. Let $h\in C_0(W)$ be a positive function with $h(x) = 1$ and $\|h\| \leq 1$. By the choice of $W$ we have
	\begin{equation}
	\label{5-7-1}
	\sigma_{1k}(h) 	u_{g_m} h u_{g_m}^* = 0, \quad m=1,\ldots,M.
	\end{equation}
	Then we compute
	\begin{align*}
		1 & \stackrel{\phantom{\eqref{5-6-1}}}{=} \| \sigma_{1k}(h)^2 \| \\
		& \stackrel{\eqref{5-6-1}}{=} \| \sigma_{1k}(h) g_\delta(\varphi)(e_{1k})^* h g_\delta(\varphi)(e_{1k}) \| \\
		& \stackrel{\eqref{5-7-2}}{\leq} \Big\| \sigma_{1k}(h) \Big( \sum\nolimits_{m=1}^M f_mu_{g_m} \Big) h g_\delta(\varphi)(e_{1k}) \Big\| + {1}/{2} \\
		& \stackrel{\phantom{\eqref{5-6-1}}}{=} \Big\|  \sum\nolimits_{m=1}^M f_m \sigma_{1k}(h) u_{g_m} h g_\delta(\varphi)(e_{1k})  \Big\| + {1}/{2} \\
		& \stackrel{\eqref{5-7-1}}{=} {1}/{2},
	\end{align*}
	a contradiction, so that we have proved the first statement.
	
	Suppose now the action is free. Then for each $x$ the element $g_x$ is uniquely determined, so that the assignment $x\mapsto g_x$ is well-defined. The set $V_m := \{x\in U_1 \mid \alpha_{g_m}(x) = \bar{\sigma}_{k1}(x) \}$ is precisely the preimage of the diagonal $\Delta_X := \{(x,x)\in X\times X\}$ under the continuous map from $U_1$ into $X\times X$ which sends $x$ to $(g_m(x), \bar{\sigma}_{k1}(x))$. Since $X$ is Hausdorff, the diagonal $\Delta_X$ is closed and therefore $V_m$ is closed in the relative topology on $U_1$. As $\{V_1,\ldots,V_m\}$ forms a finite partition of $U_1$, we conclude that each $V_m$ is open in $U_1$. This shows that the map $x\mapsto g_x$ is continuous.
\end{proof}

Recall that if $\mathcal{U} = (U_\beta)_{\beta}$ and $\mathcal{V} = (V_\gamma)_{\gamma}$ are open covers of a topological space $X$, then the \emph{join} of $\mathcal{U}$ and $\mathcal{V}$, written as $\mathcal{U} \vee \mathcal{V}$, is defined by
\[
\mathcal{U} \vee \mathcal{V} := ( U_\beta\cap V_\gamma  )_{\beta,\gamma}.
\]
Note that the join $\mathcal{U} \vee \mathcal{V}$ is an open cover refining both $\mathcal{U}$ and $\mathcal{V}$.

\bigskip
After these preparations we are now ready to prove the main new feature of this section, the lower bound for diagonal dimension in Theorem~\ref{thm:diag_vs_tower}.

\begin{proof}[Proof of \eqref{ddim-lower}]
	For notational convenience we write $A := C(X)\rtimes_{\mathrm{r}} G$, $D_A := C(X)$, and we may assume $d:= \ddim(D_A\subset A) < \infty$, for otherwise there is nothing to show. Let $E\subset G$ be a finite subset and $\mathcal{U}$ a finite open cover of $X$ as in the definition of fine tower dimension; see \ref{def:tower-dimension}. We may assume that $E = E^{-1}$ and $e\in E$. Find a partition of unity $\{f_1,\ldots,f_M\}$ of $X$ subordinate to the open cover $\bigvee_{g\in E} \alpha_g(\mathcal{U})$, and set
	\begin{equation}
	\label{v10-5-2-2}
	\mathcal{F} := \{ \mathbf{1}_A, f_1,\ldots,f_M  \} \cup \{ u_g \mid g\in E^2\}.
	\end{equation}
	Define the constants
	\begin{equation}
	\label{5-2-2}
	\delta:= \frac{1}{16(d+1)}, \quad \eta:= \frac{1}{8(d+1)}, \quad \text{and} \quad \e := \frac{\delta^3}{4M}.
	\end{equation}
	Let $(F,D_F, \psi, \varphi  )$ be a c.p.\ approximation for $(\mathcal{F}^2,{\e^2}/{9})$ which witnesses $\ddim(D_A \subset A ) = d$  and let $\hat{\psi}:A \to \hat{F}$, $\hat{\varphi}: \hat{F} \to A$ be the c.p.c.\ maps defined in Remark \ref{rem:almost_order_zero}(ii) with respect to $\mathbf{1}_A$ in place of $h$, so in particular
	\begin{equation}
	\label{v10-5-2-3}
	\hat{\varphi} \hat{\psi} = \varphi \psi.
	\end{equation}	
By \cite[Lemma 3.5]{KW04} we then have  
	\begin{align}
	\label{5-2-7}
	& \| \hat{\varphi}(\hat{\psi}(a)b) - \hat{\varphi}\hat{\psi}(a)\hat{\varphi}(b) \| \nonumber \\
    &\leq 3^{\frac{1}{2}} \max \{ \| \hat{\varphi}(\hat{\psi}(a)) - a\|, \|\hat{\varphi}(\hat{\psi}(a^2)) - a^2\|\}^\frac{1}{2} \nonumber \\
    & < \e
	\end{align}
	for all $a\in \mathcal{F}$ and $b\in \hat{F}^1$. 
		
	Let 
	\begin{equation}
	\label{5-2-1}
	q:= \chi_{(\delta,1]}(\psi(\mathbf{1}_A)),
	\end{equation}
	where $\chi_{(\delta,1]}$ is the characteristic function on the interval $(\delta,1]$.
	
		For each $i\in \{0,\ldots,d\}$ identify the sub-$\mathrm{C}^*$-algebra $(qD_{F^{(i)}}q \subset qF^{(i)}q )$ with a direct sum $   \big(\bigoplus_{j=1}^{r^{(i)}} D_{s^{(i),j}} \subset \bigoplus_{j=1}^{r^{(i)}} M_{s^{(i),j}}  \big)$. For $ j= \{1, \ldots, r^{(i)}\}$,  $i = \{0, \ldots, d\}$ we write $\{e^{(i),j}_{kl}\}_{k,l \in \{1, \ldots,s^{(i),j}\}}$ for the matrix units in $M_{s^{(i),j}}$ and $\varphi^{(i),j}$ for the restriction of $\varphi$ to the summand $M_{s^{(i),j}}$. 
	
	Let $f_\delta, g_\delta:[0,1]\to \R$ be the piecewise linear continuous functions defined by \eqref{eq:f_delta} and \eqref{eq:g_delta}. By Lemma \ref{lem:local_iso}, for each $i \in \{0,\ldots,d\}$, $j \in \{1,\ldots,r^{(i)}\}$ and $k\in \{1,\ldots,s^{(i),j}\}$ we have a $^*$-isomorphism
	\begin{align*}
	\sigma_{1k}^{(i),j} & : \overline{ f_\delta(\varphi^{(i),j})(e_{11}^{(i),j}) A f_\delta(\varphi^{(i),j})(e_{11}^{(i),j}) } \\
	& \qquad \stackrel{\cong}{\longrightarrow} \overline{ f_\delta(\varphi^{(i),j})(e_{kk}^{(i),j}) A f_\delta(\varphi^{(i),j})(e_{kk}^{(i),j}) }
	\end{align*}
	given by
	\[
	\sigma_{1k}^{(i),j}(a) = g_\delta(\varphi^{(i),j})(e_{k1}^{(i),j}) a g_\delta(\varphi^{(i),j})(e_{1k}^{(i),j}).
	\]
	Write $U_k^{(i),j} := \mathrm{supp}^{\circ}_\eta \big( f_\delta(\varphi^{(i),j})(e_{kk}^{(i),j})  \big)$, and let $\bar{\sigma}_{k1}^{(i),j}:U_1^{(i),j}\to U_k^{(i),j}$ be the inverse of the  homeomorphism induced by $\sigma_{1k}^{(i),j}$. By Lemma \ref{lem:homeo_by_gp_element} for every $x\in U_1^{(i),j}$ and every $k\in \{2,\ldots,s^{(i),j}\}$ there exists a unique group element $g$ such that $\alpha_g(x) = \bar{\sigma}_{k1}^{(i),j}(x)$. For each $(s^{(i),j}-1)$-tuple $(g_2,\ldots,g_{s^{(i),j}})$ define 
	\begin{align*}
	V^{(i),j}_{(g_2,\ldots,g_{s^{(i),j}}) } & := \big\{x\in U_1^{(i),j} \mid \bar{\sigma}_{k1}^{(i),j}(x) = \alpha_{g_k}(x) \\
	& \phantom{{ } := { } \big\{x\in U_1^{(i),j} \mid} \text{ for all } k\in \{2,\ldots,s^{(i),j} \}   \big\}.
	\end{align*}
	Then by Lemma \ref{lem:homeo_by_gp_element} each set $V^{(i),j}_{(g_2,\ldots,g_{s^{(i),j}}) }$ is a finite intersection of open sets, and hence is open itself. Freeness of the action ensures that the collection $\big(V^{(i),j}_{(g_2,\ldots,g_{s^{(i),j}}) }\big)_{(g_2,\ldots,g_{s^{(i),j}})}$, where the index runs over all $(s^{(i),j}-1)$-tuples in $G$, forms a finite partition of $U_1^{(i),j}$. For  ease of notation we write $(V_r^{(i),j})_{r=1}^{R^{(i),j}}$ for the collection $\big(V^{(i),j}_{(g_2,\ldots,g_{s^{(i),j}}) }\big)_{(g_2,\ldots,g_{s^{(i),j}})}$.
	
	 Let $S_r^{(i),j}$ be the finite subset of $G$ formed by the $(s^{(i),j}-1)$-tuple corresponding to $V_r^{(i),j}$.  
	 
	 We claim that 
	\begin{align*}
	\mathcal{C} & := \big\{ (V_r^{(i),j}, ES_r^{(i),j} ) \mid i\in \{0,\ldots,d\}, j\in \{1,\ldots,r^{(i)}\}, \\
	& \phantom{{ } := { } \big\{ (V_r^{(i),j}, ES_r^{(i),j} ) \mid { }} r\in \{1,\ldots,R^{(i),j} \} \big\}
	\end{align*}
	is a collection of open towers witnessing the fine tower dimension for the finite subset $E$ and the open cover $\mathcal{U}$.
	
	We first show that the collection 
	\begin{align}
	\big\{ \alpha_g(V^{(i),j}_r) & \mid  i \in \{0, \ldots,d\}, \, j \in \{1,\ldots,r^{(i)}\},\, \nonumber \\
	&\phantom{{ } \mid { }} r \in \{1,\ldots,R^{(i),j}\}, \, g \in S^{(i),j}_r \big\} \label{5-2-3}
	\end{align}
	 already forms a cover of $X$, which will then immediately imply that the collection $\mathcal{C}$ is $E$-Lebesgue and covers $X$, i.e., conditions (1) and (2) of Definition~\ref{def:tower-dimension} hold. 
	 
	 To see the former, it suffices to show that the collection 
	 \begin{align}
	 \big\{\mathrm{supp}^{\circ}_{2\eta}(f_{\delta}(\varphi^{(i),j})(e^{(i),j}_{kk})) & \mid i\in \{0,\ldots,d\}, \, j\in \{1,\ldots,r^{(i)}\}, \, \nonumber\\
	 &\phantom{{ } \mid { }} k\in \{1,\ldots,s^{(i),j} \}  \big\} \label{5-2-4}
	 \end{align}
	  forms an open cover of $X$. We compute
	\begin{align}
		&  \sum\nolimits_{i=0}^d \sum\nolimits_{j=1}^{r^{(i)}}\sum\nolimits_{k=1}^{s^{(i),j}} f_{\delta} (\varphi^{(i),j})(e^{(i),j}_{kk}) \nonumber \\
		& \stackrel{\phantom{\eqref{5-2-2}}}{\geq} \sum\nolimits_{i=0}^d \sum\nolimits_{j=1}^{r^{(i)}}\sum\nolimits_{k=1}^{s^{(i),j}} \varphi(e^{(i),j}_{kk}) - \delta (d+1)\cdot \mathbf{1}_A \nonumber \\
		&\stackrel{\phantom{\eqref{5-2-2}}}{=} \varphi(q) - \delta(d+1)\cdot \mathbf{1}_A \nonumber \\
		&\stackrel{\phantom{\eqref{5-2-1}}}{\geq} \varphi(q \psi(\mathbf{1}_A)q) - \delta(d+1)\cdot \mathbf{1}_A \nonumber \\
		&\stackrel{\eqref{5-2-1}}{\geq} \varphi\psi(\mathbf{1}_A) - 2\delta(d+1)\cdot \mathbf{1}_A \nonumber \\
		&\stackrel{\phantom{\eqref{5-2-2}}}{\geq} \mathbf{1}_A - ( \e + 2\delta(d+1)  )\cdot \mathbf{1}_A \nonumber \\
		& \stackrel{\eqref{5-2-2}}{\geq} {1}/{2} \cdot \mathbf{1}_A, \label{5-2-10}
	\end{align}
	where for the first inequality we have used \eqref{eq:f_delta} and the fact that for each $i$ the map $\varphi^{(i)}|_{qF^{(i)}q} = \sum_{j=1}^{r^{(i)}} \varphi^{(i),j}$ is c.p.c.\ order zero; this latter fact (together with \eqref{5-2-1}) yields the third inequality. 
	
	By orthogonality of the summands for any fixed $i$, for each $x\in X$ there exist some indices $i$, $j$, and $k$ such that
	\[
	\textstyle
	f_\delta(\varphi^{(i),j})(e^{(i),j}_{kk})(x) \geq {1}/{(2(d+1))} \stackrel{\eqref{5-2-2}}{>} 2\eta,
	\]
	which proves our claim that the collection in \eqref{5-2-4}, hence also that in \eqref{5-2-3}, indeed covers $X$.
	
	\bigskip
	
	Next, we show that for each $i\in \{0,\ldots,d\}$ the collection 
	\[
	\big\{ \alpha_s (V^{(i),j}_r) \mid j\in \{1,\ldots,r^{(i)}\},\, r\in \{1,\ldots,R^{(i),j}\}, \, s\in ES^{(i),j}_r \big\}
	\]
	 consists of pairwise disjoint sets. Assume, for the sake of contradiction, that for some $i$ there exist $j,j'$, $r,r'$, and $t,t'\in E$, $s\in S_r^{(i),j}$, and $s'\in S^{(i),j'}_{r'}$ such that $(ts, j, r) \neq (t's', j', r')$ and $\alpha_{ts}(V_r^{(i),j}) \cap \alpha_{t's'}(V^{(i),j'}_{r'}) \neq \emptyset$. Then we can find $x\in V^{(i),j}_r$ and $y\in V^{(i),j'}_{r'}$ satisfying	$\alpha_{ts}(x) = \alpha_{t's'}(y)
	$.	Note that $x$ and $y$ must be distinct: If $(j,r) = (j',r')$, then $x = y$ would imply $ts = t's'$ because the action is free. And if $(j,r)\neq (j',r')$ then $V^{(i),j}_r$ and $V^{(i),j'}_{r'}$ are disjoint subsets. Either way, we have $x \neq y$.
	 	
	Now define 
	\begin{equation}
	\label{v10-5-2-4}
	g := (t')^{-1}t \in E^2.
	\end{equation} 
	Then $\alpha_{(s')^{-1}gs}(x) = y \neq x$, and we can find an open neighbourhood $U_x$ of $x$ such that
	$U_y := \alpha_{(s')^{-1}gs}(U_x)$ is disjoint from $U_x$. Let $h_x$ be a function in $C_0(U_x)^1_+$ satisyfing $h_x(x) = 1$. Set $h_y := h_x(\alpha_{ (s')^{-1}gs  }(\, . \,))$. Then $h_y \in C_0(U_y)_+^1$ and $h_y(y) = 1$. By construction we have $u_g(\alpha_s(h_x)  )u_g^* = \alpha_{s'}(h_y)$, where again we abuse notation and write $\alpha$ also for the induced action on $C(X)$. Therefore,
	\[
	\| \alpha_{s'}(h_y) u_g \alpha_s(h_x) \| = 1.
	\]
	Let $k\in \{1,\ldots,s^{(i),j}\}$ be the (unique) index so that $\varphi(e_{kk}^{(i),j})$ dominates $\alpha_s(h_x)$, and let $k'\in \{1,\ldots,s^{(i),j'}\}$ be the similarly corresponding index for $\alpha_{s'}(h_y)$. By construction,
	 \begin{equation}
	 \label{5-2-5}
	\varphi(e_{kk}^{(i),j})|_{U_k^{(i),j}} \geq \delta 
	\end{equation}
	and
	\begin{equation}
	\label{5-2-6}
	\hat{\varphi}(e_{kk}^{(i),j}) = \varphi( \psi(\mathbf{1}_A) e_{kk}^{(i),j} ) \geq \delta \cdot \varphi(e_{kk}^{(i),j}) \geq \delta^2
	\end{equation}
	on $U^{(i),j}_k$. The same lower bounds hold for $\varphi(e^{(i),j'}_{k'k'})$ on $U^{(i),j'}_{k'}$. With these estimates in hand, we compute	
	\begin{align}
		1 &= \| \alpha_{s'}(h_y) u_g \alpha_s(h_x) \| \nonumber \\
		&  \leq \frac{1}{\delta^3} \| \alpha_{s'}(h_y)\varphi(e^{(i),j'}_{k'k'}) u_g \hat{\varphi}(e^{(i),j}_{kk})\alpha_s(h_x) \| \nonumber \\
		&  \leq \frac{1}{\delta^3} \| \alpha_{s'}(h_y) \varphi(e_{k'k'}^{(i),j'}) \hat{\varphi}\hat{\psi}(u_g) \hat{\varphi}(e^{(i),j}_{kk})\alpha_s(h_x) \| + \frac{\e}{\delta^3} \nonumber \\
		&  \le \frac{1}{\delta^3} \| \alpha_{s'}(h_y) \varphi(e_{k'k'}^{(i),j'}) \hat{\varphi}( \hat{\psi}(u_g)e^{(i),j}_{kk} )  \alpha_s(h_x) \| + 2\frac{\e}{\delta^3}, \label{5-2-8}
	\end{align}
	where for the first inequality we have used \eqref{5-2-5} and \eqref{5-2-6}, for the second inequality we have used \eqref{v10-5-2-2}, \eqref{v10-5-2-3}, \eqref{v10-5-2-4}, and for the last inequality we have used \eqref{5-2-7}.
		If $j\neq j'$, then the first term in the last row is zero, which leads to a contradiction. Therefore let us consider $j = j'$. We then have
	\begin{align*}
        & \alpha_{s'}(h_y) \varphi(e_{k'k'}^{(i),j})  \hat{\varphi}( \hat{\psi}(u_g)e_{kk}^{(i),j} )  \alpha_s(h_x) \\ 
        &= \alpha_{s'}(h_y) \varphi(e_{k'k'}^{(i),j})  \varphi( \psi(1_A)^{\frac{1}{2}} \psi(1_A)^{-\frac{1}{2}}\psi(u_g) \psi(1_A)^{-\frac{1}{2}}e_{kk}^{(i),j} \psi(1_A)^{\frac{1}{2}})  \alpha_s(h_x) \\ 
		&= \alpha_{s'}(h_y) \varphi(e_{k'k'}^{(i),j})  \varphi( \psi(u_g)e_{kk}^{(i),j} )  \alpha_s(h_x) \\ 
		&= \alpha_{s'}(h_y)\varphi(\mathbf{1}_{F^{(i)}}) \varphi( e_{k'k'}^{(i),j} \psi(u_g) e_{kk}^{(i),j} ) \alpha_s(h_x) \\
		&= \varphi(\mathbf{1}_{F^{(i)}})\alpha_{s'}(h_y) \varphi( e_{k'k'}^{(i),j} \psi(u_g) e_{kk}^{(i),j} ) \alpha_s(h_x),
	\end{align*}
	where we have used that $e_{kk}^{(i),j}$ and $\psi(1_A)$ commute, and that $\varphi(\mathbf{1}_{F^{(i)}}) \in C(X)$, hence commutes with $\alpha_{s'}(h_y)$.
	
	Since $e_{k'k'}^{(i),j} \psi(u_g) e_{kk}^{(i),j} = \lambda e_{k'k}^{(i),j}$ for some $\lambda\in \C$, we can compute
	\begin{align*}
		&\alpha_{s'}(h_y) \varphi( e_{k'k'}^{(i),j} \psi(u_g) e_{kk}^{(i),j} ) \alpha_s(h_x)\\ 
		&= \lambda\cdot \alpha_{s'}(h_y) \varphi(  e_{k'k}^{(i),j} ) \alpha_s(h_x) \\
		&= \lambda \cdot g_\delta(\varphi)(e_{k'1}^{(i),j} ) h_y \big( g_\delta(\varphi)(e_{1k'}^{(i),j} ) \varphi(e_{k'k}^{(i),j}) g_\delta(\varphi)(e_{k1}^{(i),j} ) \big) h_x g_\delta(\varphi)(e_{1k}^{(i),j} )\\
		&= \lambda \cdot g_\delta(\varphi)(e_{k'1}^{(i),j} ) h_y \big( g_\delta(\varphi)(e_{11}^{(i),j} )^2 \varphi(e_{11}^{(i),j}) ) \big) h_x g_\delta(\varphi)(e_{1k}^{(i),j} ).
	\end{align*}
	The term in the middle bracket belongs to $C(X)$, hence commutes with $h_x$. Since the construction ensures that $h_yh_x = 0$, we have from \eqref{5-2-8}
	\[
	1 = \| \alpha_{s'}(h_y) u_g \alpha_s(h_x) \|  \leq 2 \frac{\e}{ \delta^3} \stackrel{\eqref{5-2-2}}{\leq} \frac{1}{2},
	\]
	again a contradiction. This establishes our claim on the pairwise disjointness of the levels.
	\bigskip
	
	Finally, we show that the open cover 
	\begin{align*}
	\big\{\alpha_{gs}(V^{(i),j}_r) & \mid g \in E,\, i \in \{0,\ldots,d\},\, j \in \{1,\ldots, r^{(i)}\},\\
	&\phantom{{ } \mid { }} r \in \{1,\ldots,R^{(i),j}\}, \, s \in S^{(i),j}_r \big\}
	\end{align*}
	refines the given cover $\mathcal{U}$. By construction, given any $s\in S_r^{(i),j}$ there exists a (unique) $\bar{k}\in \{1,\ldots,s^{(i),j}\}$ such that
	\begin{align*}
		\alpha_s(V_r^{(i),j}) \subset \mathrm{supp}^{\circ}_\eta ( f_\delta(\varphi^{(i)})(e^{(i),j}_{\bar{k}\bar{k}}) )  \subset  \big\{x\in X \mid \varphi(e^{(i),j}_{\bar{k}\bar{k}})(x) > \delta \big\}.
	\end{align*}
	Since
	\[
	\delta\cdot \sum\nolimits_{i,j,k} e^{(i),j}_{kk} = \delta\cdot \mathbf{1}_{qFq} \stackrel{\eqref{5-2-1}}{\leq} \psi(\mathbf{1}_A) = \sum\nolimits_{m=1}^M \psi(f_m),
	\]
	there exists $\bar{m}\in \{1,\ldots,M\}$ such that
	\begin{equation}
	\label{5-2-9}
	\psi(f_{\bar{m}} ) \ge \psi(f_{\bar{m}} ) e^{(i),j}_{\bar{k}\bar{k}} \geq \frac{\delta}{M} \cdot e^{(i),j}_{\bar{k}\bar{k}}
	\end{equation}
	(using that all the $\psi(f_m)$ and $e^{(i),j}_{kk}$ are in $D_F$, hence commute).
	
	It follows that for any $x\in \alpha_s(V_r^{(i),j})$, we have
	\[
	\delta \stackrel{\eqref{5-2-5}}{\le} \varphi(e^{(i),j}_{\bar{k}\bar{k} })(x) \stackrel{\eqref{5-2-9}}{\le} \frac{M}{\delta}\cdot \varphi\psi(f_{\bar{m}})(x)  \leq \frac{M}{\delta} \cdot f_{\bar{m}}(x) + \frac{M\e}{\delta}.
	\]
	As $M\e < \delta^2$, we see that $\alpha_s(V_r^{(i),j})$ is entirely contained in the support of $f_{\bar{m}}$, which in turn is contained in some member of the open cover  $\bigvee_{g\in E} \alpha_g(\mathcal{U})$. 
	
	It remains to point out that for any $t\in E$, the open set $\alpha_t(\alpha_s(V_r^{(i),j}))$ is a subset of some member of $\mathcal{U}$. This completes the proof.
\end{proof}

\begin{rem}
It would be interesting to see if the preceding results can be generalised to cocycle actions; at this point it seems not so clear whether the arguments carry over directly.   \end{rem}

\section{Groupoids}
\label{sec6}

\noindent
We now take a more general point of view than in the previous section, by considering \emph{groupoid} $\mathrm{C}^*$-algebras. This will pave the grounds for a substantial number of further applications. For the time being we will focus on a \emph{lower} bound for diagonal dimension in terms of \emph{dynamic asymptotic dimension}, analogous to \eqref{ddim-lower} in Theorem \ref{thm:diag_vs_tower}.

Throughout this section all groupoids are assumed to be locally compact and Hausdorff, as these have naturally associated sub-$\mathrm{C}^*$-algebras for which  diagonal dimension encodes interesting information. We start by establishing our notation; we also briefly recall the construction of the reduced $\mathrm{C}^*$-algebra associated to a groupoid since it is used later. For more details we refer the reader to \cite{Renault:LNM1980} or \cite{Sims:Notes}, for example.

For a groupoid $\mathcal{G}$, we write $\mathcal{G}^{(0)}$ for the unit space, and $r,s:\mathcal{G}\to \mathcal{G}^{(0)}$ for the range and source map, respectively.  The composition of two elements $g,h\in \mathcal{G}$ is written $gh$, and the inverse of $g\in \mathcal{G}$ is denoted $g^{-1}$. For a unit $x\in \mathcal{G}^{(0)}$, define $\mathcal{G}^x := r^{-1}(\{x\})$, $\mathcal{G}_x := s^{-1}(\{x\})$, and $\mathcal{G}_x^x := \mathcal{G}^x\cap \mathcal{G}_x$. A subset $S$ of $\mathcal{G}$ is called a \emph{bisection} if there is an open subset $U$ containing $S$ such that $r:U\to r(U)$ and $s:U\to s(U)$ are homeomorphisms onto open subsets of $\mathcal{G}^{(0)}$.

A groupoid $\mathcal{G}$ is called 
\begin{itemize}
	\item \emph{principal} if the map $g\mapsto (r(g),s(g))$ is injective; 
	\item \emph{\'etale} if the maps $r$ and $s$ are local homeomorphisms;
	\item \emph{ample} if there is a basis consisting of compact open bisections for its topology.
\end{itemize}

Let $\mathcal{G}$ be a locally compact, Hausdorff, \'etale groupoid and consider the complex vector space $C_{\mathrm{c}}(\mathcal{G})$ of compactly supported complex-valued continuous functions on $\mathcal{G}$. Define a convolution product on $C_{\mathrm{c}}(\mathcal{G})$ by the formula
\[
a_1\ast a_2(g)=\sum\nolimits_{h\in \mathcal{G}^{r(g)}}a_1(h)a_2(h^{-1}g)
\]
and an involution by the formula
\[ a^*(g)=\overline{a(g^{-1})}.\]
With these operations $C_{\mathrm{c}}(\mathcal{G})$ becomes a $^*$-algebra.

For each $x\in \mathcal{G}^{(0)}$, the regular representation $\pi_x:C_{\mathrm{c}}(\mathcal{G})\rightarrow \mathcal{B}(\ell^2(\mathcal{G}_x))$ is given by 
\[
\pi_x(a)\delta_g=\sum\nolimits_{s(h)=r(g)}a(h)\delta_{hg}.
\]
The reduced $\mathrm{C}^*$-norm on $C_{\mathrm{c}}(\mathcal{G})$ is defined by
\[
\|a\|_{\mathrm{r}} := \sup_{x\in G^{(0)}}\|{\pi_x(a)}\|
\]
and the reduced groupoid $\mathrm{C}^*$-algebra $\mathrm{C}_{\mathrm{r}}^*(\mathcal{G})$ is the completion of $C_c(\mathcal{G})$ with respect to $\|\, . \,\|_{\mathrm{r}}$. Since $\mathcal{G}$ is \'etale, there is a natural inclusion $C_{\mathrm{c}}(\mathcal{G}^{(0)})\to C_{\mathrm{c}}(\mathcal{G})$ which extends to an embedding $C_0(\mathcal{G}^{(0)})\to \mathrm{C}^*_{\mathrm{r}}(\mathcal{G})$, and we view $C_0(\mathcal{G}^{(0)})$ as an abelian sub-$\mathrm{C}^*$-algebra of $\mathrm{C}^*_{\mathrm{r}}(\mathcal{G})$. By \cite[Lemma 2.1 (5)]{BCS15} $C_0(\mathcal{G}^{(0)})$ contains an approximate unit for $\mathrm{C}^*_{\mathrm{r}}(\mathcal{G})$, so the sub-$\mathrm{C}^*$-algebra $( C_0(\mathcal{G}^{(0)}) \subset \mathrm{C}^*_{\mathrm{r}}(\mathcal{G}))$ is nondegenerate.

One can show that when the groupoid is principal, the abelian subalgebra $ C_0(\mathcal{G}^{(0)})$ is diagonal in $\mathrm{C}^*_{\mathrm{r}}(\mathcal{G})$. In fact, the reconstruction theorem \cite[Theorem 5.9]{Ren08} of Renault, which builds on earlier work of Kumjian (\cite{Kum86}), shows that there is a one-to-one correspondence between Cartan subalgebras and {\it twisted}, locally compact, Hausdorff, \'etale, topologically principal groupoids (see \cite{Ren08} for relevant definitions). Moreover, a Cartan subalgebra has the unique extension property (hence is a diagonal) precisely when the associated twisted groupoid is principal (instead of just topologically principal). Upon combining this discussion with Theorem \ref{principal grp}, we obtain the following:

\begin{prop}\label{prop:principal-groupoid} Let $(D_A\subset A)$ be a nondegenerate sub-$\mathrm{C}^*$-algebra with finite diagonal dimension. Then
there is an -- up to isomorphism uniquely determined -- twisted, \'etale, locally compact, Hausdorff, principal groupoid $(\mathcal{G},\Sigma)$ such that $(D_A\subset A)$ is isomorphic to $(C_0(\mathcal{G}^{(0)}) \subset \mathrm{C}_{\mathrm{r}}^*(\mathcal{G},\Sigma))$.
\end{prop}

Below we will investigate how diagonal dimension is related to existing dimension-type properties for groupoids, where for the time being we stick to the untwisted case. We start with dimension zero, which characterises AF algebras on the $\mathrm{C}^*$-algebra side and has a natural groupoid analogue.

\begin{defn}\cite[Definition~III.1.1]{Renault:LNM1980}\cite[Definition~3.7]{GPS04}\cite[Definition~2.2]{MR2876963}
	Let $\mathcal{G}$ be an ample, (principal,) second countable, locally compact, Hausdorff, \'etale groupoid. $\mathcal{G}$ is said to be an \emph{AF groupoid} if it can be written as a union of an increasing sequence of open principal subgroupoids $\{\mathcal{G}_n\}_{n\in \N}$ such that $\mathcal{G}_n^{(0)}=\mathcal{G}^{(0)}$ and $\mathcal{G}_n\backslash \mathcal{G}_n^{(0)}$ is compact for each $n\in \N$. 
\end{defn}

\begin{prop}\label{prop:uniAF}
	Let $(D\subset A)$ be a nondegenerate sub-$\mathrm{C}^*$-algebra with $A$ separable and $D$ abelian. 
	Then $\ddim(D\subset A) = 0$ if and only if there is an AF groupoid $\mathcal{G}$ such that $( D\subset A  )$ is isomorphic to $( C_0(\mathcal{G}^{(0)}) \subset \mathrm{C}^*_{\mathrm{r}}(\mathcal{G})  )$. Up to isomorphism $\mathcal{G}$ is uniquely determined by these properties.
	\end{prop}
	
	\begin{proof}
		If $(D\subset A)$ has diagonal dimension zero, then the construction in the proof of \cite[Proposition III.1.15]{Renault:LNM1980}, together with the implication (i) $\Longrightarrow$ (ii) in Theorem \ref{thm:zero_ddim}, produces an AF groupoid whose associated sub-$\mathrm{C}^*$-algebra is $(D\subset A)$.  Conversely, the proof of \cite[Proposition III.1.15]{Renault:LNM1980} shows that the sub-$\mathrm{C}^*$-algebra associated to an AF groupoid satisfies condition (ii) in Theorem \ref{thm:zero_ddim}, hence has diagonal dimension zero.
\end{proof}

Recall from \cite{TW:ssa} that a UHF algebra $U$ is \emph{of infinite type} if $U \cong U \otimes U$. The corollary below establishes the corresponding notion at the level of groupoids by means of fixing a regular canonical masa (cf.\ Remark~\ref{AFdiag}). 

\begin{cor}
Let $U$ be a UHF algebra of infinite type with a regular canonical masa $D_U$. Then there exists a second countable, minimal, principal, AF groupoid $\mathcal{G}$ with compact unit space $\mathcal{G}^{(0)}$ such that $(D_U\subset U)\cong(C(\mathcal{G}^{(0)}) \subset \mathrm{C}^*_{\mathrm{r}}(\mathcal{G}) )$ and $\mathcal{G}\times \mathcal{G}\cong \mathcal{G}$ as topological groupoids. Up to isomorphism this AF groupoid is uniquely determined.
\end{cor}

\begin{proof}
By Remark~\ref{AFdiag} we have $\ddim( D_U \subset U)=0$. By Theorem~\ref{thm:permanence}(ii) we have $\ddim(D_U\otimes D_U \subset U\otimes U) = 0$, whence again by Remark~\ref{AFdiag} $(D_U\otimes D_U \subset U\otimes U)$ is a regular canonical masa (one can of course also check this fact directly). Now if $\varphi: U \otimes U \to U$ is some isomorphism, then $(\varphi(D_U\otimes D_U)\subset U)$ is again a regular canonical masa, and by the uniqueness statement in Remark~\ref{AFdiag} there is an automorphism $\psi$ of $U$ such that $\psi(\varphi(D_U\otimes D_U)) = D_U$. We therefore have an isomorphism $\psi \circ \varphi: (D_U\otimes D_U \subset U\otimes U)\cong (D_U \subset U)$. 

On the other hand, by Proposition~\ref{prop:uniAF} there is an (up to isomorphism  uniquely determined) AF groupoid $\mathcal{G}$ such that $(D_U\subset U)\cong(C(\mathcal{G}^{(0)}) \subset \mathrm{C}^*_{\mathrm{r}}(\mathcal{G}) )$. Via the  identification
\begin{align*}
(C((\mathcal{G}\times \mathcal{G})^{(0)}) \subset \mathrm{C}^*_{\mathrm{r}}(\mathcal{G}\times \mathcal{G}) ) &\cong (C(\mathcal{G}^{(0)})\otimes C(\mathcal{G}^{(0)}) \subset \mathrm{C}^*_{\mathrm{r}}(\mathcal{G})\otimes \mathrm{C}^*_{\mathrm{r}}(\mathcal{G}))\\
&\cong (D_U \otimes D_U \subset U \otimes U) \\
&\cong (D_U \subset U)\\
&\cong (C(\mathcal{G}^{(0)}) \subset \mathrm{C}^*_{\mathrm{r}}(\mathcal{G}) )
\end{align*}
we conclude from \cite[Proposition~4.13]{Ren08} that $\mathcal{G}\times \mathcal{G}\cong \mathcal{G}$ as topological groupoids.
\end{proof}

Motivated by the Baum--Connes conjecture, Guentner, Willett, and Yu in \cite{GWY17} introduced the notion of \emph{dynamic asymptotic dimension} for groupoids. The concept is a groupoid version of Gromov's asymptotic dimension. The idea is to approximately exhaust $\mathcal{G}$ by a bounded number of ``relative AF'' groupoids.

\begin{defn} \cite[Definition 5.1]{GWY17} \label{defn:dad}
	Let $\mathcal{G}$ be a locally compact,  Hausdorff, \'etale groupoid. $\mathcal{G}$ has \emph{dynamic asymptotic dimension} at most $d$, written $\dad(\mathcal{G}) \leq d$, if for every open relatively compact subset $K$ of $\mathcal{G}$ there exist open subsets $U^{(0)},\ldots,U^{(d)}$ of $\mathcal{G}^{(0)}$ such that
	\begin{enumerate}
		\item the union $\bigcup_{i=0}^d U^{(i)}$ covers $s(K)\cup r(K)$, and
		\item for each $i\in \{0,\ldots,d\}$ the set	$\{ g\in K \mid s(g),r(g)\in U^{(i)} \}$ generates a relatively compact subgroupoid.
	\end{enumerate}
\end{defn}

For principal groupoids condition (2) of Definition \ref{defn:dad} may be re\-phrased using the language of equivalence relations. Suppose $\mathcal{H}$ is a subgroupoid of $\mathcal{G}$. For each $x \in \mathcal{G}^{(0)}$, we denote by $[x]_\mathcal{H}$ the equivalence class of $x$ with respect to the equivalence relation $\sim_\mathcal{H}$ on $\mathcal{G}^{(0)}$ induced by $\mathcal{H}$, i.e., $x\sim_\mathcal{H} y$ if and only if there is an $h\in \mathcal{H}$ with $s(h) = x$ and $r(h) = y$.

\begin{prop} \label{lem:dad_orbit}
	Let $\mathcal{G}$ be a principal, locally compact, Hausdorff, \'etale groupoid. Then the following are equivalent:
	\begin{enumerate}
		\item[{\rm(i)}] $\dad(\mathcal{G})\leq d$;
		\item[{\rm (ii)}] for every open relatively compact subset $K$ of $\mathcal{G}$ there exist open subsets $U^{(0)},\ldots,U^{(d)}$ of $\mathcal{G}^{(0)}$ such that
		\begin{enumerate}
			\item[{\rm (1)}] the union $\bigcup_{i=0}^d U^{(i)}$ covers $s(K)\cup r(K)$, and
			\item[{\rm (2)}] for each $i\in \{0,\ldots,d\}$ we have $\sup_{x\in (\mathcal{H}^{(i)})^{(0)}} |[x]_{\mathcal{H}^{(i)}} | < \infty$, where $\mathcal{H}^{(i)}$ is the subgroupoid generated by the set $\{ g\in K \mid s(g),r(g)\in U^{(i)} \}$.
		\end{enumerate}
	\end{enumerate}
\end{prop}

\begin{proof}
	(i) $\Longrightarrow$ (ii): This follows from \cite[Lemma 8.10]{GWY17}. In fact, this implication does not require $\mathcal{G}$ to be principal.
	
	(ii) $\Longrightarrow$ (i): For all $i\in \{0,\ldots,d\}$ we define numbers $M^{(i)} := \sup_{x\in (\mathcal{H}^{(i)})^{(0)}} |[x]_{\mathcal{H}^{(i)}}|$. Since $\mathcal{G}$ is principal, we have $\mathcal{H}^{(i)}\subset K\cdots K$ ($M^{(i)}$ times). By \cite[Lemma 5.2]{GWY17} the product of finitely many relatively compact subsets is relatively compact, so $\mathcal{H}^{(i)}$ is relatively compact. 
\end{proof}

The main result of this section provides a lower bound for the diagonal dimension of the sub-$\mathrm{C}^*$-algebra $( C_0(\mathcal{G}^{(0)}) \subset \mathrm{C}^*_{\mathrm{r}}(\mathcal{G})) $ in terms of the dynamic asymptotic dimension of the groupoid $\mathcal{G}$. At a technical level the proof is very similar to that of \eqref{ddim-lower} in Theorem \ref{thm:diag_vs_tower}, but the two do not factorise through one another: Theorem \ref{thm:diag_vs_dad} works in the more general context of groupoid $\mathrm{C}^*$-algebras, but the lower bound in terms of dynamic asymptotic dimension is a priori weaker than the one in terms of (fine) tower dimension in Theorem \ref{thm:diag_vs_tower}.

\begin{thm} \label{thm:diag_vs_dad}
	Let $\mathcal{G}$ be a locally compact, Hausdorff, \'etale groupoid. Then
	\[
	\dad(\mathcal{G})\leq \ddim( C_0(\mathcal{G}^{(0)})\subset \mathrm{C}^*_{\mathrm{r}}(\mathcal{G}) ).
	\]
\end{thm}

\begin{proof}
We may assume that $\ddim( C_0(\mathcal{G}^{(0)})\subset \mathrm{C}^*_{\mathrm{r}}(\mathcal{G})  ) = d <\infty$, for otherwise there is nothing to show. But then $\mathcal{G}$ is principal by Proposition~\ref{prop:principal-groupoid}, and it will suffice to verify condition (ii) of Proposition~\ref{lem:dad_orbit}. We write $A:=\mathrm{C}^*_{\mathrm{r}}(\mathcal{G})$ and $D_A := C_0(\mathcal{G}^{(0)})$. Given an open relatively compact subset $K\subset \mathcal{G}$ as in \ref{lem:dad_orbit}(ii), we can cover $\overline{K}$ by finitely many open bisections $S_1,\ldots,S_M$, because $\mathcal{G}$ is \'etale and $\overline{K}$ is compact. Since $\mathcal{G}$ is locally compact and Hausdorff, there exist functions $a_1,\ldots,a_M$ in  $C_{\mathrm{c}}(\mathcal{G})$ such that 
\begin{itemize}
	\item $\supp(a_m)\subset S_m$ for all $m \in \{1,\ldots ,M \}$, and
	\item $\sum_{m=1}^M a_m(x) = 1$ for all $x\in \overline{K}$.
\end{itemize}
Let $0 \le f\in D_A$ be a function satisfying $\|f\| \leq 1$ and $f\vert_{ s(\overline{K})\cup r(\overline{K}) } \equiv 1$. Define the constants
	\begin{equation}
	\label{6-7-9}
	\delta := \frac{1}{8(d+1)}, \quad \eta := \frac{1}{4(d+1)}, \quad \e := \frac{\delta^3}{6M^4}.
	\end{equation}
	Using nondegeneracy of the sub-$\mathrm{C}^*$-algebra $(D_A\subset A)$ and functional calculus we can find positive contractions $h,\tilde{f} \in D_A$ and contractions $b_1,\ldots, b_M \in A$ such that
	\[
	\| \tilde{f} - f \| < \e, \quad h\tilde{f} = \tilde{f},
	\]
	and
	\begin{equation}
	\label{6-7-7}
	\| b_m - a_m \| < \e, \quad hb_m = b_m
	\end{equation}
	for all $m\in \{1,\ldots, M\}$. Write $\mathcal{F} := \{ \tilde{f}, b_1,\ldots,b_M \}$ and
	let $(F, D_F, \psi, \varphi  )$ be a c.p.\ approximation witnessing $\ddim(D_A\subset A ) = d$ for $(\mathcal{F}^2,\e^2/9)$. Let
	\[
	\hat{\psi}:A \longrightarrow \hat{F}, \quad \hat{\varphi}:\hat{F} \longrightarrow A
	\]
	be the c.p.\ maps defined in Remark \ref{rem:almost_order_zero}(ii) with respect to $h$. Then $\hat{\varphi}$ is contractive and 
	\begin{equation}
	\label{6-7-8}
\hat{\varphi}\hat{\psi}(b) = \varphi\psi(b)
\end{equation}
for all $b\in A$ satisfying $hb=b$. By \cite[Lemma 3.5]{KW04}, just as in \eqref{5-2-7} above we can even conclude
	\begin{equation}
	\label{6-7-6}
	\| \hat{\varphi}(\hat{\psi}(b)x) - \hat{\varphi}\hat{\psi}(b)\hat{\varphi}(x) \| < \e
	\end{equation}
	for all $b\in \mathcal{F}$ and $x\in \hat{F}^1$.
	
	Let 
	\begin{equation}
	\label{6-7-5}
	q:= \chi_{(\delta,1]}(\psi(h)) \in \hat{F} \cap D_F,
	\end{equation}
	where $\chi_{(\delta,1]}$ denotes the characteristic function on the interval $(\delta,1]$. 
	
		For each $i\in \{0,\ldots,d\}$ identify the sub-$\mathrm{C}^*$-algebra $( qD_{F^{(i)}}q \subset qF^{(i)}q  )$ with the direct sum $\big( \bigoplus_{j=1}^{r^{(i)}} D_{s^{(i),j}} \subset  \bigoplus_{j=1}^{r^{(i)}} M_{s^{(i),j}} \big)$. As in the proof of  \eqref{ddim-lower}, we will write $\big(e_{kl}^{(i),j}\big)$ for the standard matrix units in $M_{s^{(i),j}}$ and $\varphi^{(i),j}$ for the respective restrictions of $\varphi$.
	
	Let $f_\delta, g_\delta:[0,1]\to \R$ be the piecewise linear continuous functions defined in \eqref{eq:f_delta} and \eqref{eq:g_delta}, and let 
	\begin{align*}
	{\sigma}_{kl}^{(i),j}  &: \overline{ f_\delta(\varphi^{(i),j})(e_{kk}^{(i),j}) D_A f_\delta(\varphi^{(i),j})(e_{kk}^{(i),j}) } \\
	& \qquad \longrightarrow 
	\overline{		f_\delta(\varphi^{(i),j})(e_{ll}^{(i),j}) D_A f_\delta(\varphi^{(i),j})(e_{ll}^{(i),j}) }
	\end{align*}
	be the $^*$-isomorphism given by Lemma \ref{lem:local_iso}. Define
	\[
	U_k^{(i),j} := \mathrm{supp}^\circ_\eta  (f_\delta(\varphi^{(i),j})(e_{kk}^{(i),j} )) \subset \mathcal{G}^{(0)}.
	\]
	Then each ${\sigma}_{kl}^{(i),j}$ induces a homeomorphism $\bar{\sigma}_{lk}^{(i),j}:U_k^{(i),j}\to U_l^{(i),j}$. Define
	\[
	U^{(i)} := \bigcup\nolimits_{j=1}^{r^{(i)}}\bigcup\nolimits_{k=1}^{s^{(i),j}} U_k^{(i),j}.
	\]
	We claim that $U^{(0)},\ldots,U^{(d)}$ form the desired cover of $s(K)\cup r(K)$.
	
	We first show that the union $\bigcup_{i=0}^d U^{(i)}$ indeed covers $s(K)\cup r(K)$. For each $x\in s(K)\cup r(K)$, we compute, analogously to \eqref{5-2-10},
	\begin{align*}
		& \sum\nolimits_{i=0}^d \sum\nolimits_{j=1}^{r^{(i)}}\sum\nolimits_{k=1}^{s^{(i),j}} f_\delta(\varphi^{(i),j})( e^{(i),j}_{kk} ) (x) \\
		& \geq \sum\nolimits_{i=0}^d \sum\nolimits_{j=1}^{r^{(i)}}\sum\nolimits_{k=1}^{s^{(i),j}} \varphi( e^{(i),j}_{kk} ) (x) - \delta(d+1) \\
		&= \varphi(q)(x) - \delta(d+1) \\
		&\geq \varphi(\psi(h)q)(x) - \delta(d+1)  \\
		&\geq \varphi\psi(h)(x) - 2\delta(d+1) \\
		&\geq \varphi\psi(\tilde{f})(x) - 2\delta(d+1) \\
		&\geq \tilde{f}(x) - (\e + 2\delta(d+1)) \\
		&\geq f(x) - (2\e + 2\delta(d+1)) \\
		&\textstyle \geq {1}/{2}.
	\end{align*}
	By orthogonality, for each $x \in s(K) \cup r(K)$ there exist some $i,j,k$ such that
	\[
	\textstyle
	f_\delta(\varphi^{(i),j})(e^{(i),j}_{kk})(x) \geq {1}/{(2(d+1))} > 2 \eta,
	\]
	whence the union $\bigcup_{i=0}^d U^{(i)}$ indeed covers the set $s(K)\cup r(K)$, i.e., condition \ref{lem:dad_orbit}(ii)(1) holds.

	\bigskip

	Next, we verify condition \ref{lem:dad_orbit}(ii)(2). 
	For each $i \in \{0,\ldots,d\}$ let $\mathcal{H}^{(i)}$ be the subgroupoid generated by $\{g\in K \mid s(g),r(g)\in U^{(i)} \}$. We need to show that  $\sup_{x\in (\mathcal{H}^{(i)})^{(0)}} | [x]_{\mathcal{H}^{(i)}} | < \infty$ (recall that by definition $x,y\in \mathcal{G}^{(0)}$ belong to the same equivalence class if and only if there is $g\in \mathcal{H}^{(i)}$ with $s(g) = x$ and $r(g) = y$). In fact,  it suffices to prove the following: for each $i$, $j$, and $x\in U_1^{(i),j}$, if $g\in K$ is a groupoid element satisfying $s(g) = \bar{\sigma}_{k1}^{(i),j}(x)$ for some $k$ and $r(g) \in U^{(i)}$, then $r(g)$ lies in the set $\{ \bar{\sigma}_{l1}^{(i),j}(x) \mid l =1,\ldots,s^{(i),j}  \}$. This statement will then imply that $\sup_{x\in (\mathcal{H}^{(i)})^{(0)}}| [x]_{\mathcal{H}^{(i)}} | \leq \max\{ s^{(i),j} \mid j=1,\ldots,r^{(i)} \}$.
	
	\bigskip
	
	So let us fix indices $i$ and $j$, a point $x\in U_1^{(i),j}$, and a groupoid element $g\in K$ such that $s(g) = \bar{\sigma}^{(i),j}_{k1}(x)$ and $r(g) \in U^{(i)}$. Assume, for the sake of contradiction, that $r(g) = \bar{\sigma}_{k'1}^{(i),j'}(y)$ for some indices $j'$ and $k'$, and $y\in U_1^{(i),j'}$ with $y\neq x$.

	By construction, there exists an element $a\in \{a_1,\ldots,a_M\}$ such that 
	\begin{equation}
	\label{6-7-1}
a(g) \geq \frac{1}{M}.
\end{equation} 
Let $S$ be the open support of $a$ (which is an open bisection), and write $\alpha_S:s(S)\to r(S)$ for the canonical homeomorphism given by
$$
x\mapsto r((s|_S)^{-1}(x))\ \text{for $x\in s(S)$}.
$$

By \cite[Proposition 1.6]{Kum86}  (see also \cite[Proposition 4.7]{Ren08}) we have
	\[
	a^*ca(z) = c(\alpha_S(z))a^*a(z)
	\]
	for all $z\in \dom(a) := \{ x\in \mathcal{G}^{(0)} \mid a^*a(x) > 0  \}$ and $c \in D_A$.
	
	Since $x$ and $y$ are distinct, there exists an open neighborhood $U_x$ of $x$ inside $\bar{\sigma}^{(i),j}_{k1}(\dom(a))$ such that
	\begin{equation}
	\label{6-7-10}
	[\bar{\sigma}_{1k'}^{(i),j'}\vert_{r(S)}  \circ \alpha_S \circ \bar{\sigma}_{1k}^{(i),j}(U_x)] \cap U_x = \emptyset.
	\end{equation}
	Write $U_y := \bar{\sigma}_{1k'}^{(i),j'}\vert_{r(S)}  \circ \alpha_S \circ \bar{\sigma}_{1k}^{(i),j}(U_x)$, and let $h_y$ be a function in $C_0(U_y)^1_+$ satisfying 
	\begin{equation}
	\label{6-7-2}
h_y(y) = 1.
\end{equation} 
Define
	\begin{equation}
	\label{6-7-4}
	h_x := \sigma_{k1}^{(i),j}(  a^* \sigma_{1k'}^{(i),j'}(h_y) a   ) \in C_0(U_x)_+^1.
	\end{equation}
	Then
	\begin{align}
		h_x(x) &\stackrel{\phantom{\eqref{6-7-1},\eqref{6-7-2}}}{=} [a^* \sigma_{1k'}^{(i),j'}(h_y)a] ( s(g)  ) \nonumber \\
		& \stackrel{\phantom{\eqref{6-7-1},\eqref{6-7-2}}}{=} \sigma_{1k'}^{(i),j'}(h_y)(r(g)) a^*a( s(g)) \nonumber \\
		&\stackrel{\phantom{\eqref{6-7-1},\eqref{6-7-2}}}{=} h_y(y) |a|^2(g)  \nonumber \\
		& \stackrel{\eqref{6-7-1},\eqref{6-7-2}}{\geq} {1}/{M^2}. \label{6-7-3}
	\end{align}

	 Let $b\in \{b_1,\ldots,b_M\}$ be an element satisfying $\|b-a\| < \e$. We compute
	\begin{eqnarray}
	\frac{1}{M^4} &\stackrel{\eqref{6-7-3}}{\leq}& \|h_x^2\|	\nonumber\\
	& \stackrel{\eqref{6-7-4}}{\leq}& \| \sigma_{1k}^{(i),j}(h_x) a^* \sigma_{1k'}^{(i),j'}(h_y) \| \nonumber\\	
	&=& \| \sigma_{1k'}^{(i),j'}(h_y) a \sigma_{1k}^{(i),j}(h_x) \| \nonumber	\\	
	&\stackrel{\eqref{5-6-1},\eqref{eq:f_delta}}{\leq} & \frac{1}{\delta^2} \| \sigma_{1k'}^{(i),j'}(h_y) \varphi(e_{k'k'}^{(i),j'}) a \varphi( e_{kk}^{(i),j} ) \sigma_{1k}^{(i),j}(h_x) \| \nonumber\\	
	&\stackrel{\eqref{6-7-5},\eqref{2-2-1}}{\leq} &\frac{1}{\delta^3} \| \sigma_{1k'}^{(i),j'}(h_y) \varphi(e_{k'k'}^{(i),j'}) a \hat{\varphi}( e_{kk}^{(i),j} ) \sigma_{1k}^{(i),j}(h_x) \| \nonumber\\	
	&\leq& \frac{1}{\delta^3} \| \sigma_{1k'}^{(i),j'}(h_y) \varphi(e_{k'k'}^{(i),j'}) b \hat{\varphi}( e_{kk}^{(i),j} ) \sigma_{1k}^{(i),j}(h_x) \| + \frac{\e}{\delta^3} \nonumber\\
	& \stackrel{\eqref{6-7-7},\eqref{6-7-8}}{\leq}& \frac{1}{\delta^3} \| \sigma_{1k'}^{(i),j'}(h_y) \varphi(e_{k'k'}^{(i),j'}) \hat{\varphi}\hat{\psi}(b) \hat{\varphi}( e_{kk}^{(i),j} ) \sigma_{1k}^{(i),j}(h_x) \| + \frac{2\e}{\delta^3} \nonumber\\	
	&\stackrel{\eqref{6-7-6}}{\leq}& \frac{1}{\delta^3} \| \sigma_{1k'}^{(i),j'}(h_y) \varphi(e_{k'k'}^{(i),j'}) \hat{\varphi}(\hat{\psi}(b)  e_{kk}^{(i),j} ) \sigma_{1k}^{(i),j}(h_x) \| + \frac{3\e}{\delta^3} \nonumber\\	
	&=& \frac{1}{\delta^3} \| \sigma_{1k'}^{(i),j'}(h_y) \varphi(e_{k'k'}^{(i),j'}) \varphi(\psi(b)  e_{kk}^{(i),j} ) \sigma_{1k}^{(i),j}(h_x) \| + \frac{3\e}{\delta^3}, \label{6-7-11} \nonumber \\
	& & 
	\end{eqnarray}
	where for the last equality we have used the definition \eqref{2-2-1} of $\hat{\psi}$ and of $\hat{\varphi}$ and the fact that $e^{(i),j}_{kk}$ and $\psi(h)$ commute. 
	
	As in the proof of \eqref{ddim-lower}, if $j\neq j'$ then by orthogonality the first summand in the last line vanishes, and we already obtain a contradiction to \eqref{6-7-9}. Now consider $j=j'$. Then (using that $\varphi^{(i),j}$ is order zero)
	\begin{align*}
		& \sigma_{1k'}^{(i),j}(h_y) \varphi^{(i),j}(e^{(i),j}_{k'k'}) \varphi^{(i),j}(\psi(b)  e^{(i),j}_{kk} ) \sigma_{1k}^{(i),j}(h_x) \\
		&= \varphi^{(i),j}(\mathbf{1}_{F^{(i)}}) \sigma_{1k'}^{(i),j}(h_y) \varphi^{(i),j}(e^{(i),j}_{k'k'} \psi(b)  e^{(i),j}_{kk} ) \sigma_{1k}^{(i),j}(h_x).
	\end{align*}
	Writing $e_{k'k'}^{(i),j} \psi(b)  e_{kk}^{(i),j} = \lambda \cdot e_{k'k}^{(i),j}$ for some $\lambda\in \C$, we can continue as follows:
	\begin{eqnarray*}
		\lefteqn{ \sigma_{1k'}^{(i),j}(h_y)  \varphi^{(i),j}(e_{k'k'}^{(i),j} \psi(b)  e_{kk}^{(i),j} ) \sigma_{1k}^{(i),j}(h_x)}\\
		&= &\lambda\cdot \sigma_{1k'}^{(i),j}(h_y)  \varphi^{(i),j}(e_{k'k}^{(i),j} ) \sigma_{1k}^{(i),j}(h_x) \\
		&\stackrel{\eqref{5-6-1}}{=} &\lambda\cdot g_\delta(\varphi^{(i),j})(e_{k'1}^{(i),j})h_y g_\delta(\varphi^{(i),j})(e_{1k'}^{(i),j})\\
		& & \varphi^{(i),j}(e_{k'k}^{(i),j}) g_\delta(\varphi^{(i),j})(e_{k1}^{(i),j})h_y g_\delta(\varphi^{(i),j})(e_{1k}^{(i),j}) \\
		&\stackrel{\eqref{6-7-10}}{=}& 0.
	\end{eqnarray*}
	Summarising, we have that
	\[
	\frac{1}{M^4} \leq \| h_x^2\| \leq \frac{3\e}{\delta^3} = \frac{1}{2M^4},
	\]
	a contradiction, so our proof is complete.
\end{proof}

\begin{rem}\label{ample-groupoids}
	One should expect that in Theorem~\ref{thm:diag_vs_dad}, if $\mathcal{G}$ is principal and ample, one actually has equality. This is analogous to Theorem~\ref{thm:diag_vs_tower}, where \eqref{ddim-lower} and \eqref{ddim-upper} turn into the equality \eqref{ddim-zero} if the underlying space is zero-dimensional. That situation of transformation groupoids will suffice to address the applications we are mostly interested in; cf.\ Corollary \ref{ample-transformation-groups} below and \ref{thm:uniformRoe}. Therefore we will not pursue the  statement in the full generality of  groupoids, and we skip the rather technical proof for the time being.  
\end{rem}

\begin{rem}\label{transformation-group-chain}
	Let $\mathcal{G}$ be the transformation groupoid arising from a free action of a countable discrete group $G$ on a compact metrisable space $X$. By \cite[Lemma 5.4]{GWY17} and \cite[Theorem 5.14]{Ker17}  we have
	\begin{align}
	&\dad^{+1}(\mathcal{G}) \nonumber \\
	&\qquad \leq \tdim^{+1}(G\curvearrowright X)\nonumber \\ 
	&\qquad \qquad \leq \ftdim^{+1}( G\curvearrowright X ) \nonumber \\
	& \qquad \qquad \qquad	\leq \dad^{+1}(\mathcal{G})\cdot \dim^{+1}(X). \label{6-9-1}
	\end{align}
	In view of Theorems~\ref{thm:diag_vs_tower} and \ref{thm:diag_vs_dad} it seems plausible that $\ddim^{+1} (C(X) \subset C(X)\rtimes_{\mathrm{r}} G)$ sits just between the last two terms in this chain of inequalities (it is larger than fine tower dimension by \eqref{ddim-lower}). 
	\end{rem}
\begin{rem}
In \cite[Theorem~7.1.1]{Bl90}, Blackadar has shown that the CAR algebra can be realised as a crossed product $\mathrm{C}^*$-algebra $C(X)\rtimes_\mathrm{r}G$, where $G$ is a locally finite countable group acting freely and minimally on $X=S^1\times \Omega$ where $\Omega$ is the Cantor set. By \cite[Remark~2.2(i) and Lemma~5.4]{GWY17}, the transformation groupoid $\mathcal{G}=X\rtimes G$ has dynamic asymptotic dimension zero, and we see from \eqref{6-9-1} in connection with \eqref{tdim-ftdim} that

$$\tdim(G\curvearrowright X) \leq \ftdim( G\curvearrowright X ) = \dim(X)=1.
$$
Now Remark~\ref{rem:almost_order_zero}(i) and Theorem~\ref{thm:diag_vs_tower} together imply 
\[
1 = \dim X \leq \ddim (C(X) \subset C(X)\rtimes_{\mathrm{r}} G)\leq 3.
\]
(We leave the exact value of the diagonal dimension unspecified at this point. A positive answer to Question~\ref{question:ddim-dad} below would imply that it is one; this also seems plausible in view of Blackadar's construction.) 

Upon taking tensor products one obtains  diagonals in the CAR algebra with spectrum $(S^1)^n\times \Omega$ for any $n\in \mathbb{N}\cup\{\infty\}$. 

This discussion shows that there are non-AF diagonals in AF algebras with arbitrarily high (but finite) diagonal dimension. 

In earlier versions of this paper we raised the question whether a separable AF algebra $A$ can admit a diagonal $D_A$ which itself is AF but is not a  regular canonical masa, i.e., with $0< \ddim (D_A \subset A)$ (cf.\ Remark~\ref{AFdiag}). In the meantime, however, Kopsacheilis and the third named author managed to construct such examples; see \cite{KW}. In fact, for every $d \in \mathbb{N} \cup \{\infty\}$, there is a Cantor spectrum diagonal in the CAR algebra $M_{2^\infty}$, with diagonal dimension precisely $d$.  
\end{rem}

\begin{question}	
\label{question:ddim-dad}
In \cite[Theorem 8.6]{GWY17}, the nuclear dimension of a groupoid $\mathrm{C}^*$-algebra in large generality is bounded by a term of the form $\dad^{+1}(\mathcal{G})\, \cdot \, \dim^{+1}(\mathcal{G}^{(0)})$, and it would be interesting to see whether along the same lines one can also obtain an upper bound for diagonal dimension. More precisely:
 Is it true that for any \'etale principal groupoid $\mathcal G$ we have the estimate \[
 \ddim^{+1}( C_0(\mathcal G^{(0)} ) \subset \mathrm{C}^*_r(\mathcal G)) \leq \dad^{+1}(\mathcal G)\cdot \dim^{+1}(\mathcal G^{(0)})\,?
 \]
 \end{question}

\begin{rem}
If, in the situation of Remark~\ref{transformation-group-chain}, the space $X$ is totally disconnected, then we may combine the chain of inequalities \eqref{6-9-1} above with Theorem~\ref{ddim-zero} and obtain the equality $\dad(\mathcal{G}) = \ddim (C_0(X) \subset C_0(X)\rtimes_{\mathrm{r}} G)$,  
	as predicted in Remark~\ref{ample-groupoids} above. This argument factorises through \cite[Theorem 5.14]{Ker17}, which is stated and proven under the hypothesis  that $X$ is metrisable. However, this latter assumption only enters through the estimate $\ftdim^{+1}( G\curvearrowright X ) \le \dad^{+1}(\mathcal{G})\cdot \dim^{+1}(X)$. When $\dim X = 0$, metrisability can be avoided as follows: 
	In the proof of \cite[Lemma 5.11]{Ker17}, one may skip the application of \cite[Lemma 5.9]{Ker17} (which allows to refine the castle $\{(O_i,S_i)\}_{1\le i \le q}$) and arrive at the same conclusion except for condition (i) of \cite[5.11]{Ker17} on the diameters of the levels $sV_i$.	The proof of \cite[Lemma 5.12]{Ker17} runs as stated; the only difference is that now \cite[Lemma 5.11]{Ker17} does not yield the diameter condition (iii) of \cite[Lemma 5.12]{Ker17}. Now the proof of the last inequality of \cite[Theorem 5.14]{Ker17} works verbatim, again with the only exception that it does not yield the diameter condition on the levels of the castles covering $X$. This argument yields the estimate $\tdim(G \curvearrowright X) \le \dad (\mathcal{G})$, without assuming $X$ to be metrisable. The reverse inequality $\dad (\mathcal{G}) \le \ddim( C_0(\mathcal{G}^{(0)})\subset \mathrm{C}^*_{\mathrm{r}}(\mathcal{G}) ) = \tdim(G \curvearrowright X)$  
	follows upon combining Theorems \ref{thm:diag_vs_dad} and \ref{thm:diag_vs_tower}, neither of which requires $X$ to be metrisable. We summarise this discussion as follows:
 \end{rem}

\begin{cor}\label{ample-transformation-groups}
	Let $G$ be a countable discrete group acting freely on a compact, totally disconnected, Hausdorff space $X$ (which is not necessarily metrisable), and let $\mathcal{G}$ be the associated transformation groupoid. Then
	\begin{align*}
	\dad (\mathcal{G}) = \ddim (C(X) \subset C(X)\rtimes_{\mathrm{r}} G).
	\end{align*}
\end{cor}

Corollary~\ref{ample-transformation-groups} answers 
Question~\ref{question:ddim-dad} in the case of transformation groupoids with totally disconnected
unit spaces. This is likely to hold more generally, at least for ample groupoids (using ideas from \cite{GWY17} and \cite{BDGW:PLMS}), and it will be pursued in future work.
\bigskip

\section{Further Examples}
\label{sec:examples}

\noindent
Below we describe various (classes of) examples in order to highlight the scope of diagonal dimension, and to showcase how it carries more refined information when compared to nuclear dimension, in particular.

\bigskip

In Section~\ref{section-AF} we have already characterised diagonal dimension zero in terms of AF sub-$\mathrm{C}^*$-algebras. When the $\mathrm{C}^*$-pair comes from a dynamical system, this type of characterisation carries over to the group: 

\begin{prop}
\label{free action of finite group}
	Let $G$ be a countable discrete group acting freely on a compact  Hausdorff space $X$. 
	
	Then $\ddim(C(X) \subset C(X)\rtimes_{\mathrm{r}} G)=0$ if and only if $X$ is totally disconnected and $G$ is locally finite, i.e., every finite subset of $G$ is contained in a finite subgroup.
\end{prop}

\begin{proof}
	By Corollary~\ref{ample-transformation-groups}, in our situation diagonal dimension agrees with dynamic asymptotic dimension, and by  \cite[Remark~2.2 (i)]{GWY17} dynamic asymptotic dimension zero is equivalent to $G$ being locally finite.
\end{proof}

We have just used that by Corollary~\ref{ample-transformation-groups}, for free actions on zero-dimensional spaces dynamic asymptotic dimension of the groupoid and diagonal dimension of the $\mathrm{C}^*$-pair agree. However, it turns out that in large generality dynamic asymptotic dimension of the transformation groupoid is essentially determined by the asymptotic dimension of the group; cf.\ \cite{Gromov:GGT1993}. As a consequence we obtain:

\begin{prop}\label{virtually nilpotent}
	Let $G$ be a finitely generated group acting freely on a compact, metrisable, and totally disconnected space $X$. Then
	$$
	\ddim(C(X) \subset C(X)\rtimes_{\mathrm{r}} G )\in \{\asdim (G),\infty\}.
	$$
	If in addition $G$ is finitely generated and virtually nilpotent, then we have $$\ddim (C(X) \subset C(X)\rtimes_{\mathrm{r}} G)=\asdim (G).$$
\end{prop}

\begin{proof}
It was shown in \cite[Corollary~8.8]{Saw17} (which  is based on a yet unpublished result by Wu and Zacharias; see \cite[Theorem~8.1]{Saw17}) that $\dad (G \curvearrowright X)\in \{\asdim (G),\infty\}$. This is also an easy consequence of \cite[Theorem~B]{B24} (we refer the reader to \cite{CJMST23} for the recent development in this direction). Therefore, the first statement follows from Corollary \ref{ample-transformation-groups}. 

 If $G$ is a finitely generated virtually nilpotent group, then by \cite[Corollary~1.10]{Bar17} (applied to the special case of a free action) the amena\-bility dimension of the action is finite. But by \cite[Corollary~5.14]{Ker17} the latter agrees with the diagonal dimension, and so we conclude from the first part that $\ddim(C(X) \subset C(X)\rtimes_{\mathrm{r}} G ) = \asdim (G)$.
 \end{proof}

\begin{examples}
\label{ddim-dimnuc-examples}
Every countably infinite group $G$ admits a free and minimal action on a totally disconnected, compact, metrisable, Hausdorff space $X$ by \cite{Elek:ETDS2021}. If $G$ locally has subexponential growth, then every such action yields a crossed product with nuclear dimension at most one: Indeed, by \cite[Theorem~6.33]{DZ23} and \cite[Theorem~8.1]{KerrSzabo18} the action must be almost finite. In particular, $C(X)\rtimes_{\mathrm{r}} G$ is $\mathcal{Z}$-stable by \cite[Theorem~12.4]{Ker17}. Hence, $\ndim (C(X)\rtimes_{\mathrm{r}} G)\leq 1$ by \cite[Theorem~B]{CETWW19}. (We note that the main idea behind the dimension estimate in \cite{CETWW19} was introduced in \cite{MS:Duke}.)

 On the other hand, if $G = \mathbb{Z}^d$ then $\ddim (C(X) \subset C(X)\rtimes_{\mathrm{r}} G) = \asdim(G) = d$ by Proposition~\ref{virtually nilpotent}. If $G$ is $\bigoplus_{n=1}^\infty \Z$ or the first Grigorchuk group, which has local subexponential growth, then $\ddim (C(X) \subset C(X)\rtimes_{\mathrm{r}} G) = \asdim(G) = \infty$; cf.\ \cite{Smi07}. In summary we have:
 \end{examples}

 \begin{prop}
 For every $d \in \{1,2,\ldots\} \cup \{\infty\}$ there is a free and minimal Cantor system $G \curvearrowright X$ with $G$ of local subexponential growth such that $\ndim (C(X) \rtimes_{\mathrm{r}} G) = 1$ and $\ddim (C(X) \subset C(X) \rtimes_{\mathrm{r}} G) = d$. 	
 \end{prop}

The phenomenon in the proposition above is not at all limited to the groups appearing in \ref{ddim-dimnuc-examples}, as shown by the next two classes of examples.

\begin{example}
Let $G$ be an arbitrary non-amenable countable group with asymptotic dimension $d$. Then one can combine the idea in \cite[Section~6]{RS12} with  \cite[Theorem~6.6]{GWY17} in order to construct a free, minimal and amenable action on the Cantor set $X$ which has dynamic asymptotic dimension $d$ and such that $C(X)\rtimes_{\mathrm{r}} G$ is a Kirchberg algebra in the UCT class (see also \cite[Section~10]{Ele18}; the UCT follows from \cite{Tu}). Then, $\ndim (C(X)\rtimes_{\mathrm{r}} G)=1$ by \cite{RSS15} and $\ddim (C(X) \subset C(X)\rtimes_{\mathrm{r}} G)=d$ by Corollary~\ref{ample-transformation-groups}.

It is worth noting that there indeed are examples of non-amenable groups with asymptotic dimension exactly $d$ for all
$d\in \N$ by using \cite[Theorem~1.5]{BL08} and cocompact groups of isometries of classical hyperbolic space.

	\end{example}

	\begin{example}
	Now let $G$ be an arbitrary amenable countably infinite group with asymptotic dimension $d$. The space of actions of $G$ on the Cantor space $X$ carries a natural Polish topology. It was shown in \cite{CJKMST17} that the free and minimal actions contain a dense $G_\delta$ set of almost finite actions, which then yield simple, nuclear, and $\mathcal{Z}$-stable crossed products; the latter have nuclear dimension at most one by \cite[Theorem~B]{CETWW19} (cf.\ \cite{MS:Duke} for the idea behind the dimension estimate). On the other hand, as before we have $\ddim (C(X) \subset C(X)\rtimes_{\mathrm{r}} G)\geq \asdim (G)$.
\end{example}

We now look at metric spaces with bounded geometry and their associated $\mathrm{C}^*$-algebras. In the breakthrough rigidity result of \cite{BFKVW:2021} it was shown that these uniform Roe algebras determine the underlying space up to coarse equivalence. Since asymptotic dimension is a coarse invariant, it follows that the asymptotic dimension of a metric space with bounded geometry is encoded in its uniform Roe algebra. It is, however, an altogether quite different problem how to read off the actual value from the $\mathrm{C}^*$-algebra. In \cite{WZ10}, the nuclear dimension of any  uniform Roe algebra was bounded above by the asymptotic dimension of the space. The precise value of the nuclear dimension remains unknown, even for concrete and supposedly easy examples, but in view of the examples above as well as \cite{BraWin:BLMS} it seems plausible that it may take values strictly smaller than the asymptotic dimension, at least in certain cases. 
 On the other hand, if one keeps track of the canonical diagonal of a uniform Roe algebra, then diagonal dimension of the $\mathrm{C}^*$-pair indeed captures the  asymptotic dimension of the space on the nose: 

\begin{thm}\label{thm:uniformRoe}
Let $X$ be a metric space with bounded geometry and let $\mathrm{C}_{\mathrm{u}}^*(X)$ be its uniform Roe algebra. Then
	$$
	\ddim(\ell^\infty(X) \subset \mathrm{C}_{\mathrm{u}}^*(X))=\asdim (X).
	$$	
\end{thm}

\begin{proof}
		Let $\mathcal{G}(X)$ denote the coarse groupoid associated with $X$. This is a principal, \'etale, locally compact, $\sigma$-compact, Hausdorff, topological groupoid with unit space $\mathcal{G}^{(0)} = \beta X$; see \cite[Proposition~3.2]{STY02} and  \cite[Theorem~10.20]{Roe03}. Since the unit space $\beta X$ is totally disconnected, $\mathcal{G}(X)$ is ample; see \cite[Proposition 4.1]{Exe10}. Moreover, the uniform Roe algebra $\mathrm{C}_{\mathrm{u}}^*(X)$ of $X$ is naturally isomorphic to the reduced groupoid $\mathrm{C}^*$-algebra of $\mathcal{G}(X)$, the isomorphism mapping $\ell^\infty(X)$ onto $C(\mathcal{G}^{(0)})$; see \cite[Proposition~10.29]{Roe03}. We know from \cite[Theorem~6.4]{GWY17} that $\dad (\mathcal{G}(X))=\asdim (X)$, and so by Theorem \ref{thm:diag_vs_dad} we have 
		$$
	\ddim(\ell^\infty(X) \subset \mathrm{C}_{\mathrm{u}}^*(X)) \ge \asdim (X).
	$$	
	
The reverse inequality follows from inspection of the proof of \cite[Theorem~8.5]{WZ10}. That theorem says that $\ndim(\mathrm{C}_{\mathrm{u}}^*(X)) \le \asdim (X)$, but the same proof in fact yields the respective estimate for diagonal dimension. To see this, take a finite subset $\mathcal{F} \subset (\mathrm{C}_{\mathrm{u}}^*(X))^1_+$ and a tolerance $\e > 0$. As in the proof of \cite[Theorem~8.5]{WZ10} take c.p.\ approximations for $\mathcal{F}$ within $\e /2$ of the form
\[
\mathrm{C}_{\mathrm{u}}^*(X) \stackrel{\Psi}{\longrightarrow} A^{(0)} \oplus \ldots \oplus A^{(n)}\stackrel{\Phi}{\longrightarrow} \mathrm{C}_{\mathrm{u}}^*(X).
\]
 Then each $A^{(i)}$ is an AF subalgebra of $\mathrm{C}_{\mathrm{u}}^*(X)$ of the form $A^{(i)} = \prod_{U \in \mathcal{U}^{(i)}} M_{|B_{r-1}(U)|}$, with each $\mathcal{U}^{(i)}$ being a uniform $r$-disjoint family of subsets of $X$. The latter in particular means that for each $i$ the matrix subalgebras $M_{|B_{r-1}(U)|}$ of $\mathrm{C}_{\mathrm{u}}^*(X)$ have uniformly bounded size and are pairwise orthogonal, so $A^{(i)}$ indeed is AF. But for each of these matrix algebras we have a canonical diagonal $D_{|B_{r-1}(U)|}$, so that $D^{(i)} := \prod_{U \in \mathcal{U}^{(i)}} D_{|B_{r-1}(U)|}$ becomes a diagonal pair of AF algebras. It follows from Theorem~\ref{thm:zero_ddim} that $\ddim (D^{(i)} \subset A^{(i)}) = 0$. Moreover, the diagonal $D^{(i)}$  sits in $\ell^\infty(X)$, and each normaliser of $D^{(i)}$ in $A^{(i)}$ is also a normaliser of $\ell^\infty(X)$ in $\mathrm{C}_{\mathrm{u}}^*(X)$, i.e.\ $\mathcal{N}_{A^{(i)}}(D^{(i)}) \subset \mathcal{N}_{\mathrm{C}_{\mathrm{u}}^*(X)}(\ell^\infty(X))$. This implies that if we have a c.p.\ approximation $(F^{(i)},D_{F^{(i)}},\varrho^{(i)},\sigma^{(i)})$ for $\Psi^{(i)}(\mathcal{F})$ within $\e / 2(n+1)$ witnessing $\ddim (D^{(i)} \subset A^{(i)}) = 0$, then we may take $F:= \bigoplus_i F^{(i)}$, $D_F:= \bigoplus_i D_{F^{(i)}}$, $\psi:= (\oplus_i \varrho^{(i)}) \circ \Psi$ and $\varphi:= \Phi \circ (\oplus_i \sigma^{(i)})$ so that $(F,D_F,\psi,\varphi)$ is a c.p.\ approximation of $(\ell^\infty(X) \subset \mathrm{C}_{\mathrm{u}}^*(X))$ satisfying properties (1), (2), (3), and (5) of Definition~\ref{defn:dim_diag}. To verify property  \ref{defn:dim_diag}(4) it only remains to observe that $\Psi(\ell^\infty(X)) \subset D^{(i)}$. But this is also built into the proof of \cite[Theorem~8.5]{WZ10}, since each summand $\Psi^{(i)}$ is a compression with an element $h_i \in \ell^\infty(X)$. 
\end{proof}

If $(B \subset \mathrm{C}_{\mathrm{u}}^*(X))$ is another Roe Cartan pair in the sense of \cite[Definition 4.20]{WW18}, then it is isomorphic to a pair $(\ell^\infty(Y) \subset \mathrm{C}_{\mathrm{u}}^*(Y))$ for some metric space $Y$ of bounded geometry which is coarsely equivalent to $X$ by \cite[Theorem~B]{WW18} and \cite[Theorem~1.2]{BFKVW:2021}. It follows that, given $\mathrm{C}_{\mathrm{u}}^*(X)$, in order to extract the asymptotic dimension of $X$ one only needs an abstract Roe Cartan subalgebra, as opposed to the concrete copy of $\ell^\infty(X)$:

\begin{cor}\label{cor:RoeCartan}
Let $X$ be a metric space with bounded geometry and let $\mathrm{C}_{\mathrm{u}}^*(X)$ be its uniform Roe algebra. Then, for every Roe Cartan subalgebra $(B \subset \mathrm{C}_{\mathrm{u}}^*(X))$ we have
	$$
	\ddim(B \subset \mathrm{C}_{\mathrm{u}}^*(X))=\asdim (X).
	$$	
\end{cor}

\begin{example}
The preceding corollary would not work for ordinary Cartan subalgebras: There exists a Cartan pair $(B \subset \mathrm{C}_{\mathrm{u}}^*(X))$ for $X=\{n^2 \mid n\in \N\}$ which does not have the unique extension property hence is not a diagonal pair (see \cite[Example~3.3]{WW18}). Thus, $\ddim(B \subset \mathrm{C}_{\mathrm{u}}^*(X))=\infty$ by Proposition~\ref{unique ext}. On the other hand,  one concludes from Theorem~\ref{thm:uniformRoe} and Remark~\ref{rem:almost_order_zero}(i) that
$$
\ndim(\mathrm{C}_{\mathrm{u}}^*(X)) =\ddim(\ell^\infty(X) \subset \mathrm{C}_{\mathrm{u}}^*(X))=\asdim (X)=0.
$$ 
\end{example}

\begin{example}
	Let $G$ be a countably infinite, residually finite and amenable group and let $\tilde{G}$ be a profinite completion of $G$ associated to a separating nested sequence of finite index normal subgroups of $G$; cf.\ \cite{Orf10}. Then, $\tilde{G}$ is homeomorphic to the Cantor set and the action $G \curvearrowright \tilde{G}$ by left multiplication on the finite index subgroups is minimal (because $G$ is dense in $\tilde{G}$) and free (because the nested sequence of finite index normal subgroups $G_n$ of $G$ is separating in the sense that $\cap_n G_n=\{e\}$).
	
	The resulting crossed product $C(\tilde{G})\rtimes_{\mathrm{r}} G$, called a generalised Bunce--Deddens algebra, was shown in \cite{Orf10} to be a separable, unital, nuclear, simple, quasidiagonal, monotracial $\mathrm{C}^*$-algebra satisfying the UCT. The action of $G$ on $\tilde{G}$ is also almost finite (see the proof of \cite[Proposition~12.6]{Ker17}), so it follows that $C(\tilde{G})\rtimes_{\mathrm{r}} G$ is $\mathcal{Z}$-stable by \cite[Theorem~5.3]{CJKMST17}. Hence, $\ndim(C(\tilde{G})\rtimes_{\mathrm{r}} G)\leq 1$ by \cite[Theorem~F]{BBSTWW16} (again cf.\ \cite{MS:Duke} for the idea behind the dimension estimate).

For diagonal dimension, by Theorem~\ref{thm:diag_vs_tower}, 
 \cite[Corollary~5.15]{Ker17} and \cite[Theorem~7.2 and Remark~6.3]{SWZ17} we have the  estimate
$$
	\ddim(C(\tilde{G}) \subset C(\tilde{G})\rtimes_{\mathrm{r}} G)= \dad (\tilde{G}\rtimes G) \leq \asdim (\Box G),
	$$ 
	where $\Box G$ denotes the box space of $G$ associated to the given separating nested sequence of finite index normal subgroups; cf.\ \cite{SWZ17}. Actually, it follows from the recent result \cite[Theorem~D]{Pil23} that $\dad (\tilde{G}\rtimes G)= \asdim (\Box G)$ so that we always have $\ddim(C(\tilde{G}) \subset C(\tilde{G})\rtimes_{\mathrm{r}} G)= \asdim (\Box G)$.
	
	We also note that, if $G$ is in addition finitely generated and virtually nilpotent, then we actually have equality by  Proposition~\ref{virtually nilpotent} since $\asdim (\Box G)=\asdim (G)$ by \cite{DT18} in this situation.
\end{example}

\begin{example}
Every countable discrete group $G$ admits a universal minimal $G$-space, say $M$, which is an -- up to isomorphism uniquely determined -- minimal closed invariant subset of the Stone--\v Cech compactification $\beta G$ with respect to the left-translation action $\alpha$ of $G$; see \cite{MR0117716} for details. It is well-known that $C(M)\rtimes_{\mathrm{r}} G$ is nuclear if and only if $\alpha:G \curvearrowright M$ is topologically amenable if and only if $G$ is an exact group. We claim that in this situation we in fact have 
\[
\ndim(C(M)\rtimes_{\mathrm{r}} G)\leq 1.
\] 
To prove this it suffices to show $\ndim(C(M)\rtimes_{\mathrm{r}} G)\leq 1$ whenever $G$ is a countably infinite exact group. The dynamical system $(\alpha:G \curvearrowright M)$ is an inverse limit of a net $(\alpha_i:G \curvearrowright X_i)$, where the $\alpha_i$ can either be chosen to be almost finite if $G$ is amenable or purely infinite otherwise; see  \cite[Lemma~5.1 and Theorem~5.4]{Suzuki} and \cite[Corollary~5.9]{ABBL23} for the first case, and see \cite[Section~6]{RS12} for the second case. Either way, $A_i := C(X_i) \rtimes_{\alpha_i} G$ is $\mathcal{Z}$-stable and hence has nuclear dimension at most one by \cite[Theorem~12.4]{Ker17} and \cite[Theorem~5]{JS99} in combination with \cite{BBSTWW16}. It follows that $C(M)\rtimes_{\mathrm{r}} G = \varinjlim A_i$ also has nuclear dimension at most one.

On the other hand, Corollary~\ref{ample-transformation-groups} yields $\ddim(C(M) \subset C(M)\rtimes_{\mathrm{r}}G)=\dad(G \curvearrowright M)$. Hence, by \cite[Theorem~6.5]{GWY17} we have  
\[
\ddim(C(M) \subset C(M)\rtimes_{\mathrm{r}} G)=\asdim(G).
\]
\end{example}

\end{document}